\documentclass[10pt]{article}
\usepackage[left=1.0in,top=1.0in,right=1.0in,bottom=1.0in]{geometry}

\usepackage{graphicx}
\usepackage{subcaption}
\usepackage{amsmath}
\usepackage{amssymb}
\usepackage{amsthm}
\usepackage{latexsym}
\usepackage{bm}
\usepackage{array,supertabular}
\usepackage{color}
\usepackage{framed}
\usepackage{setspace}
\usepackage{fancyhdr}
\usepackage{stmaryrd}
\usepackage{mathrsfs}
\usepackage{mdframed}
\usepackage{url}
\usepackage{multirow}
\usepackage{algorithm}
\usepackage{algpseudocode}
\newcolumntype{P}[1]{>{\centering\arraybackslash}p{#1}}

\SetSymbolFont{stmry}{bold}{U}{stmry}{m}{n}

\newtheorem{lemma}{Lemma}
\newtheorem{proposition}{Proposition}
\newtheorem{remark}{Remark}

\numberwithin{equation}{section}

\begin{document}
\title{On the design of energy-decaying momentum-conserving integrator for nonlinear dynamics using energy splitting and perturbation techniques}
\author{Ju Liu \\
\textit{\small Department of Mechanics and Aerospace Engineering,}\\
\textit{\small Southern University of Science and Technology,}\\
\textit{\small 1088 Xueyuan Avenue, Shenzhen, Guangdong 518055, China}\\
\textit{\small Guangdong-Hong Kong-Macao Joint Laboratory for Data-Driven Fluid Mechanics and Engineering Applications,}\\
\textit{\small Southern University of Science and Technology}\\
\textit{\small 1088 Xueyuan Avenue, Shenzhen, Guangdong 518055, China}\\
\small \textit{E-mail address:} liuj36@sustech.edu.cn, liujuy@gmail.com
}
\date{}
\maketitle

\section*{Abstract}
This work proposes a suite of numerical techniques to facilitate the design of structure-preserving integrators for nonlinear dynamics. The celebrated LaBudde-Greenspan integrator and various energy-momentum schemes adopt a difference quotient formula in their algorithmic force definitions, which suffers from numerical instability as the denominator gets close to zero. There is a need to develop structure-preserving integrators without invoking the quotient formula. In this work, the potential energy of a Hamiltonian system is split into two parts, and specially developed quadrature rules are applied separately to them. The resulting integrators can be regarded as classical ones perturbed with first- or second-order terms, and the energy split guarantees the dissipative nature in the numerical residual. In the meantime, the conservation of invariants is respected in the design. A complete analysis of the proposed integrators is given, with representative numerical examples provided to demonstrate their performance. They can be used either independently as energy-decaying and momentum-conserving schemes for nonlinear problems or as an alternate option with a conserving integrator, such as the LaBudde-Greenspan integrator, when the numerical instability in the difference quotient is detected.

\vspace{5mm}

\noindent \textbf{Keywords:} Nonlinear dynamics, time integration, quadrature, energy-momentum scheme, many-body dynamics

\section{Introduction}
\label{sec:introduction}
Nonlinear dynamics presents itself as a prototypical model problem in sciences and engineering, as it is related to molecular dynamics, celestial dynamics, structural dynamics after Galerkin projection, and computer graphics, among others \cite{Hairer2006,Sharma2020}. It is anticipated that the properties of continuum models can be inherited to the discrete model, which is critical and beneficial for practical predictions. Among different attributes, stability plays a central role in guaranteeing reliable calculations, and energy is a natural candidate as the metric for stability characterization. There have been a few different approaches for ensuring energy stability, either in terms of energy conservation or energy decaying. The Lagrange multiplier was introduced to enforce energy conservation as a constraint for the trapezoidal scheme \cite{Hughes1978}, and additional constraints were introduced to enforce momentum conservation \cite{Kuhl1996}. Nevertheless, adding additional constraints has been observed to suffer from convergence issues within the nonlinear solver \cite{Kuhl1996}. There has been another branch of methods developed, which is known as the energy-momentum conserving algorithms. The design of such schemes is typically based on the Hamiltonian systems with symmetry, and the concept of stability is generalized to the concept of preserving invariants of the motion.  In particle dynamics, the conservative force is represented by a difference quotient formula that guarantees the conservation of total energy and angular momentum ab initio \cite{Chorin1978,Greenspan1984,LaBudde1974,LaBudde1976}. This strategy has inspired the development of energy-momentum conserving algorithms in nonlinear dynamics as well as elastodynamics \cite{Simo1992,Simo1992a}. In particular, a collocation scheme was proposed by Simo with a parameter determined by the mean value theorem \cite{Simo1992}. A discrete gradient formula was later proposed by Gonzalez that is applicable to arbitrary nonlinear elastic materials \cite{Gonzalez1996a,Gonzalez2000}. A variety of formulas was developed following the work of Gonzalez \cite{Armero2007,Bui2007,Orden2019}, and it has been recently found that there exist infinitely many energy-momentum formulas \cite{Romero2012}. Oftentimes, it is desirable to introduce numerical dissipation, and this leads to the concept of energy-decaying and momentum-preserving algorithms. The introduction of numerical dissipation for nonlinear problems was initially developed by Armero and Romero \cite{Armero2001,Armero2001a}. Analogous to the HHT-$\alpha$ scheme \cite{Hilber1977}, its design goal was to damp the poorly resolved high-frequency modes without harming the low-frequency modes.

To date, the energy-momentum consistent integrators have been applied to a variety of areas \cite{Sharma2020,VuQuoc1993}. Here, we want to point out an issue that has plagued practical calculations. The aforementioned energy-momentum consistent integrators can be viewed as generalizations of the LaBudde-Greenspan integrator, which approximate, for example, the conservative force $\hat{V}'(\|\bm q\|)$ for a single particle in a central force field as
\begin{align*}
\frac{\hat{V}(\|\bm q_{n+1}\|) - \hat{V}(\|\bm q_n\|)}{\|\bm q_{n+1}\| - \|\bm q_{n}\|}.
\end{align*}
Similar quotient formulas can be found in the discrete gradient approaches \cite{Romero2012}. An apparent pitfall is that the quotient formula will become numerically unstable when $\|\bm q_{n+1}\|$ gets too close to $\|\bm q_n\|$. There are several circumstances that this becomes unavoidable. Besides attractors existing in the phase space, the orbit may be driven to a limiting point under an external forcing field. This is partly the reason that the energy-momentum consistent integrators are rarely used to pursue steady-state calculations. Also, when the physical system has a dissipative mechanism, this issue may arise after long time calculations. Even in Hamiltonian systems, the transient analysis may experience this issue occasionally. A solution to this problem is to switch the energy-momentum consistent integrator to a traditional integrator that does not invoke the quotient formula, with a pre-defined tolerance set for the denominator \cite{Gonzalez1996a,Janz2019}. Arguably, switching to an alternative integrator may cause a loss of the algorithmic conservation property. In stretch-based elastic problems, for example, it was noticed that significant errors already arise when the denominator approaches $\mathcal{O}(10^{-2})$, due to the tensorial spectral decomposition algorithm adopted (see Fig. 5 of \cite{Mohr2008}). Setting the switching tolerance to $\mathcal{O}(10^{-2})$ inevitably leads to a noticeable loss of the discrete conservation property. Therefore, it is desirable to construct an energy-momentum consistent integrator without invoking the quotient formula.

Based on the well-established LaBudde-Greenspan integrator, we propose three energy-decaying and momentum-conserving integrators by invoking a combination of energy-splitting and perturbation techniques. The design techniques are inspired by the work on energy- and entropy-stable schemes for diffuse-interface models \cite{Eyre1998,Gomez2011,Liu2013,Liu2015}. The basic idea is that one can always split the energy into two parts. A pair of specifically developed quadrature rules can be separately applied to the two parts. The split of the energy is made to guarantee the dissipative nature of the quadrature residuals. This idea can be traced to the work of Eyre \cite{Eyre1998}, in which the energy was split into convex and concave parts. Yet, Eyre's approach is only first-order accurate. Recently, two suites of quadrature rules were developed based on the mid-point and trapezoidal schemes \cite{Gomez2011,Liu2013}. As those quadrature rules have a higher-order residual term, the split of the energy function has to be performed to control the fourth-order derivative. Therefore, the concept of super-convex and super-concave parts of the energy was introduced \cite{Liu2013}. The resulting schemes can be viewed as the mid-point and trapezoidal integrators with second-order perturbations, and the perturbation terms guarantee energy dissipation without harming the rest conservation properties. The three developed integrators can be used independently for nonlinear dynamics. Or, they can be used as a switching option when the energy-conserving integrator faces the aforesaid numerical pitfall, and the combined algorithm guarantees that the energy does not exhibit pathological growth.

It is also worth mentioning that the invariant energy quadratization \cite{Yang2017} and scalar auxiliary variable \cite{Shen2018} techniques originally developed for gradient flow problems have also been introduced to integrate Hamiltonian systems with attractive features in recent years. These techniques led to linearly implicit and explicit integrators, which conserve energy \cite{Bilbao2023}. Additionally, an explicit integrator was also recently designed that preserves a pseudo energy \cite{Marazzato2019}. A systematic approach for deriving conservative integrators was constructed using the concept of conservation law multiplier \cite{Wan2017,Wan2022}, and it is related to the LaBudde-Greenspan integrator.

The remainder of this article is organized as follows. In Section \ref{sec:dynamics}, we formulate the nonlinear dynamics based on the Hamiltonian system, with a brief discussion of the conservation properties. In Section \ref{sec:single-particle-dynamics}, we restrict our discussion to a single particle subjecting to a central forcing field. Based on the LaBudde-Greenspan integrator, we construct and analyze the three integrators using the energy split and perturbed quadrature rules. After a discussion of the algorithm implementation, we demonstrate numerical results to justify the claimed properties. In Section \ref{sec:many-body-dynamics}, we generalize our approach to many-body dynamics, which enjoys more conserved quantities. It is shown that the additional invariants, such as linear momentum and the center of mass, can also be respected in the proposed integrators. We draw conclusions in Section \ref{sec:conclusion}.

\section{Nonlinear dynamics}
\label{sec:dynamics}
We consider a mechanical system with $N$ material points. The configuration space of the material points is denoted by $\bm Q$, and the points in $\bm Q$ can be coordinated as 
\begin{align*}
\bm q = \left( \bm q_1, \bm q_2, \cdots, \bm q_N \right) \in \bm Q = \mathbb R^{3N}.
\end{align*}
The cotangent bundle associated with $\bm Q$ is known as the phase space, which is denoted by $\bm P$. For an arbitrary point $(\bm q, \bm p) \in \bm P$ in the phase space, the position vector $\bm q$ and its linear momentum $\bm p$ can be characterized by
\begin{align*}
\bm q = \left( \bm q_1, \bm q_2, \cdots, \bm q_N \right) \in \mathbb R^{3N}, \quad \mbox{ and } \quad \bm p = \left( \bm p^1, \bm p^2, \cdots, \bm p^{N} \right) \in \mathbb R^{3N}.
\end{align*}
The motion of the particles can be characterized by a Hamiltonian $H:\bm P \rightarrow \mathbb R$ and the canonical equations
\begin{align*}
\frac{d}{dt}\bm q_A = \frac{\partial H}{\partial \bm p^A}, \qquad \frac{d}{dt}\bm p^A = - \frac{\partial H}{\partial \bm q_A}, \quad \mbox{ for } A = 1,\cdots, N.
\end{align*}
In this work, we further assume that the Hamiltonian takes a separable form $H = K + V$, where $K:\bm P \rightarrow \mathbb R$ is the kinetic energy and $V: \bm Q \rightarrow \mathbb R$ is the potential energy. While the form of the potential energy $V(\bm q)$ can be arbitrary, the form of the kinetic energy adopts the following quadratic form,
\begin{align}
\label{eq:Hamiltonian_K}
K(\bm p) := \frac12 \sum_{A=1}^{N} \sum_{B=1}^{N} \bm m^{-1}_{AB} \bm p^A \cdot \bm p^B = \frac12 \bm p \cdot \bm M^{-1} \bm p.
\end{align}
The matrix $\bm M$ is constant, symmetric, and positive definite. With the Hamiltonian defined, the canonical equations can be specialized into the following form,
\begin{align}
\label{eq:Hamiltonian_canonical}
\frac{d}{dt}\bm q_A = \sum_{B=1}^{N} \bm m^{-1}_{AB} \bm p^{B}, \qquad \frac{d}{dt}\bm p^{A} = - \frac{\partial V}{\partial \bm q_A}, \qquad \mbox{ for } A = 1,\cdots, N.
\end{align}
Let $\mathbb I := (0,T)$ be the time interval of interest, and let $\left( \bm q(t), \bm p(t) \right) \in \bm P$ be the solutions of \eqref{eq:Hamiltonian_canonical} for $t\in \mathbb I$. The system \eqref{eq:Hamiltonian_canonical} is endowed with the following properties.  The Hamiltonian $H$ is conserved over time, which can be straightforwardly verified by inspecting its time derivative,
\begin{align*}
\frac{d}{dt}H = \sum_{A=1}^{N} \left( \frac{\partial H}{\partial \bm q_A} \cdot \frac{d}{dt}\bm q_A + \frac{\partial H}{\partial \bm p^A} \cdot \frac{d}{dt}\bm p^A \right)  = -\frac{d}{dt}\bm p^{A} \cdot \frac{d}{dt} \bm q_A + \frac{d}{dt}\bm q_A \cdot \frac{d}{dt}\bm p^A = 0.
\end{align*}
The angular momentum map $\bm J: \bm P \rightarrow \mathbb R^3$ and the linear momentum map $\bm L: \bm P \rightarrow \mathbb R^3$ are defined as
\begin{align}
\label{eq:def_angular_linear_momenta}
\bm J(\bm q, \bm p) := \sum_{A=1}^{N} \bm q_A \times \bm p^A \qquad \mbox{and} \qquad \bm L(\bm q, \bm p) := \sum_{A=1}^{N} \bm p^A,
\end{align}
respectively. Due to the celebrated Noether's theorem, the SO(3)-invariance of the Hamiltonian $H$ implies the conservation of the angular momentum map $\bm J$,
\begin{align*}
\frac{d}{dt} \bm J(\bm q(t), \bm p(t)) = \bm 0, \qquad \mbox{ for } t \in \mathbb I.
\end{align*}
In the same vein, the translational invariance of the Hamiltonian $H$ implies the conservation of the linear momentum map $\bm L$, 
\begin{align*}
\frac{d}{dt} \bm L(\bm q(t), \bm p(t)) = \bm 0, \qquad \mbox{ for } t \in \mathbb I.
\end{align*}
For a rigorous derivation of the above arguments, one may consult the work by Marsden and Ratiu \cite[pp.372-374]{Marsden1998}. The SO(3)-invariance of $H$ is satisfied if the potential energy $V$ is SO(3)-invariant. In the sequel, the potential energy $V$ is assumed to be SO(3)-invariant, as this coincides with the physical requirement of frame indifference. The Hamiltonian system can be further generalized with the aid of the General Equations for Non-Equilibrium Reversible Irreversible Coupling (GENERIC) formalism. In addition to a reversible part corresponding to the conservative structure, an irreversible part is incorporated through a dissipative bracket. Numerical strategies developed based on the Hamiltonian system can be conveniently extended to more general evolution systems within its framework \cite{Romero2009,Betsch2019}.

\section{Single particle dynamics}
\label{sec:single-particle-dynamics}
In this section, we present a suite of integrators based on a model that is simple but still exhibits representative nonlinear behaviors. In doing so, we may concentrate our discussion on the motivation and technical strategies. 

\subsection{The model problem}
Consider a single particle in $\mathbb R^3$ subjecting to a central force field. For notational simplicity, we temporarily neglect the subscripts and superscripts used for identifying particles in this subsection. The single particle's mass, position, and linear momentum are denoted as $m \in \mathbb R$, $\bm q \in \mathbb R^3$, and $\bm p \in \mathbb R^3$, respectively. Due to the considered central force field, its potential field can be written as $V(\bm q) = \hat{V}(\|\bm q\|)$, where $\hat{V}: \mathbb R_{+} \rightarrow \mathbb R$, and $\|\cdot \|$ represents the standard Euclidean norm in $\mathbb R^3$. It needs to be pointed out that translational invariance is lost in this potential energy and the linear momentum is not a conserved quantity in this case. With the particular choice of the potential energy, the force can be explicitly written as
\begin{align*}
\bm F(\bm q) = \hat{V}'(\|\bm q\|) \frac{\bm q}{\|\bm q \|},
\end{align*}
where $\hat{V}'$ is the first derivative of $\hat{V}$. In the following, we need the second, third, and fourth derivatives of the potential energy, which are denoted by $\hat{V}''$, $\hat{V}'''$, and $\hat{V}^{(4)}$, respectively. The phase space for this problem can be defined as $\bm P = \mathbb R^3 \times \mathbb R^3$, and the Hamiltonian canonical equations can be simplified as
\begin{align}
\label{eq:single-particle-strong}
\frac{d}{dt} \bm q = \frac{\bm p}{m}, \qquad \frac{d}{dt} \bm p = - \bm F(\bm q),
\end{align}
with the Hamiltonian $H:\bm P \rightarrow \mathbb R$ and angular momentum $\bm J: \bm P \rightarrow \mathbb R$ given by 
\begin{align*}
H(\bm q, \bm p) = \frac{1}{2m} \| \bm p \|^2 + V(\|\bm q\|), \quad \mbox{and} \quad J(\bm q, \bm p) = \bm q \times \bm p,
\end{align*}
both of which are invariants of the system to be inherited in the discrete model.

This model problem comes from a variety of physical sources. For example, the reduction of the two-body problem in celestial mechanics results in the above model problem, in which the potential energy is inversely proportional to $\|\bm q\|$. The model problem may also characterize intermolecular interaction in a dilute gas with the potential taking the form 
\begin{align}
\label{eq:lennard_jones_12_6}
\hat{V}(\|\bm q\|) = 4\varepsilon \left( \left( \sigma/\|\bm q\|\right)^{12} - \left( \sigma/\|\bm q\| \right)^6 \right),
\end{align}
with $\varepsilon$ and $\sigma$ being constants related to the molecules. The above is known as the Lennard-Jones 12-6 potential \cite{Allen2017}. As a third example, the potential energy can be derived from a one-dimensional elastic spring model. Interested readers may refer to \cite{Gonzalez1996} for a  nonlinear spring model $\hat{V}$ taking a quartic form and \cite{Betsch2000} for a model of $\hat{V}$ derived from a neo-Hookean type spring. 

\subsection{A summary of classical integrators}
\label{sec:summary_classcial_integrators}
We review three well-established integrators that are closely relevant to our methodology development. The discussion starts from the time finite element formulation that is conditionally conservative; based on that, the implicit mid-point scheme will be recovered; in the last, the LaBudde-Greenspan integrator will be presented as a special case of the time finite element formulation. The quotient formula in the LaBudde-Greenspan integrator inspires the subsequent discussion. We consider the time interval of interest $\mathbb I$ can be subdivided into $n_{\mathrm{ts}}$ sub-intervals $\mathbb I_n := (t_n, t_{n+1})$ of size $\Delta t_n := t_{n+1} - t_n$. To facilitate our discussion, we introduce the following notations,
\begin{align*}
\bm q_{n+\alpha} := (1-\alpha)\bm q_n + \alpha \bm q_{n+1}, \quad \bm p_{n+\alpha} := (1-\alpha)\bm p_n + \alpha \bm p_{n+1}, \quad \|\bm q\|_{n+\alpha} := (1-\alpha)\|\bm q_n\| + \alpha \|\bm q_{n+1}\|,
\end{align*}
where the parameter $\alpha \in [0,1]$. In particular, it is worth pointing out that $\|\bm q_{n+\frac12}\| \neq \|\bm q\|_{n+\frac12}$ in general. 

\subsubsection{The low-order time finite element integrator}
\label{sec:low-order-space-time-formulation}
A discrete formulation can be constructed by a Petrov-Galerkin method, using continuous $k$-th order polynomials for the trial solution and discontinuous $k-1$-th order polynomials for the test function space \cite{Betsch2000}. If $k=1$, the test function is piecewise constant over the sub-interval, and the solutions can be represented as
\begin{align*}
\bm q(t) = \left( \bm q_{n+1} - \bm q_{n} \right) \frac{t-t_n}{\Delta t_n} + \bm q_n, \quad \mbox{ and } \quad \bm p(t) = \left( \bm p_{n+1} - \bm p_{n} \right) \frac{t-t_n}{\Delta t_n} + \bm p_n,
\end{align*}
in which $\bm q_{n} := \bm q(t_n)$ and $\bm p_n := \bm p(t_n)$. The corresponding discrete formulation for the single-particle problem \eqref{eq:single-particle-strong} can be stated as
\begin{align}
\label{eq:space-time-q-single-particle}
\bm q_{n+1} - \bm q_{n} &= \frac{\Delta t_n}{m} \bm p_{n+\frac12}, \\
\label{eq:space-time-p-single-particle}
\bm p_{n+1} - \bm p_{n} &= -\int_{t_n}^{t_{n+1}} \hat{V}'(\|\bm q(t)\|) \frac{\bm q(t)}{\|\bm q(t) \|} dt.
\end{align}
The conservation properties of the above integrator are contingent upon the evaluation of the integral on the right-hand side of \eqref{eq:space-time-p-single-particle}. It has a local truncation error of second order, if the right-hand side of \eqref{eq:space-time-p-single-particle} is treated by a quadrature rule of second- or higher-order accuracy. In the meantime, assuming the integration and the solution of the nonlinear system are arithmetically exact, one can show $H(\bm q_{n+1}, \bm p_{n+1})=H(\bm q_n, \bm p_n)$. Yet, due to the nonlinear nature of the conservative force, an exact integration is often unattainable. Using a higher-order quadrature rule may improve the energy conservation property asymptotically. As will be shown, a specially designed quadrature rule can ensure energy conservation. The momentum conservation property of the integrator \eqref{eq:space-time-q-single-particle}-\eqref{eq:space-time-p-single-particle} can be analyzed by utilizing the following identity \cite{Simo1992a},
\begin{align}
\label{eq:momentum_identity}
\bm J(\bm q_{n+1}, \bm p_{n+1}) - \bm J(\bm q_{n}, \bm p_{n}) = \bm q_{n+\frac12} \times \left( \bm p_{n+1} - \bm p_n \right) + \bm p_{n+\frac12} \times \left( \bm q_{n+1} - \bm q_n \right).
\end{align}
Inserting \eqref{eq:space-time-q-single-particle}-\eqref{eq:space-time-p-single-particle} into the above identity leads to
\begin{align*}
\bm J(\bm q_{n+1}, \bm p_{n+1}) - \bm J(\bm q_{n}, \bm p_{n}) = -\bm q_{n+\frac12} \times \int_{t_n}^{t_{n+1}} \hat{V}'(\|\bm q\|) \frac{\bm q}{\|\bm q \|} dt.
\end{align*}
Therefore, it is necessary and sufficient to guarantee momentum conservation by demanding that $\bm p_{n+1} - \bm p_n$ be parallel to $\bm q_{n+\frac12}$. In this work, the integrators to be discussed are all based on the low-order time finite element formulation. Interested readers may refer to \cite{Gros2005} for the investigation on higher-order time finite element schemes.

\subsubsection{The mid-point integrator}
Applying the one-point Gaussian quadrature on the right-hand side of \eqref{eq:space-time-p-single-particle} results in the following,
\begin{align}
\label{eq:mid-point-q-single-particle}
\bm q_{n+1} - \bm q_{n} &= \frac{\Delta t_n}{m} \bm p_{n+\frac12}, \\
\label{eq:mid-point-p-single-particle}
\bm p_{n+1} - \bm p_{n} &= -\Delta t_n \hat{V}'(\|\bm q_{n+\frac12}\|) \frac{\bm q_{n+\frac12}}{\|\bm q_{n+\frac12} \|}.
\end{align}
This is the implicit mid-point rule and is apparently momentum conserving based on the above discussion. The energy conservation is satisfied if and only if the mid-point rule can exactly integrate the integral of the conservative force, which implies the potential $\hat{V}$ has to take a quadratic form (i.e., linear dynamics). It is well-known that this integrator cannot guarantee energy conservation for general nonlinear dynamics \cite{Simo1992a,Ge1988}. Instead, as its major advantage, this discrete scheme preserves the symplectic two-form in the phase space \cite{Feng1986}.

\subsubsection{The LaBudde-Greenspan integrator}
A numerical integrator that conserves both energy and momentum was proposed by LaBudde and Greenspan \cite{LaBudde1974}, and it can be stated as follows. In each time step, given $\bm q_n$ and $\bm p_n$, determine $\bm q_{n+1}$ and $\bm p_{n+1}$ such that
\begin{align}
\label{eq:LaBudde-Greenspan-q}
\bm q_{n+1} - \bm q_{n} &= \frac{\Delta t_n}{m} \bm p_{n+\frac12}, \\
\label{eq:LaBudde-Greenspan-approximation}
\bm p_{n+1} - \bm p_{n} &= - \Delta t_n \frac{\hat{V}(\|\bm q_{n+1}\|) - \hat{V}(\|\bm q_{n}\|)}{\|\bm q_{n+1}\| - \|\bm q_n\|} \frac{\bm q_{n+\frac12}}{\|\bm q\|_{n+\frac12}}.
\end{align}
The above can be viewed as a result of applying a special quadrature rule to the right-hand side of \eqref{eq:space-time-p-single-particle}. In specific, the terms $\hat{V}'$, $\bm q$, and $\|\bm q\|$ in the denominator are treated by a difference formula, the mid-point rule, the trapezoidal rule, respectively. The conservation property $H(\bm q_{n+1}, \bm p_{n+1})=H(\bm q_n, \bm p_n)$ can be straightforwardly shown by evaluating $\left( \bm q_{n+1} - \bm q_{n} \right) \cdot \left( \bm p_{n+1} - \bm p_n \right)$ using \eqref{eq:LaBudde-Greenspan-q} and \eqref{eq:LaBudde-Greenspan-approximation}. Following the discussion made in Section \ref{sec:low-order-space-time-formulation}, this integrator apparently conserves the momentum $\bm J$ as well. Again, due to the Ge-Marsden Theorem \cite{Ge1988}, this integrator cannot be symplectic. The LaBudde-Greenspan scheme inspires the development of the energy-momentum scheme for elastodynamics, originally designed for the Saint Venant-Kirchhoff model \cite{Simo1992}. Later, several different \textit{discrete gradient} formulas were found \cite{Armero2007,Bui2007,Gonzalez2000,Romero2012} for arbitrary nonlinear material models. It is also worth pointing out that the conservation properties were shown by assuming the resulting nonlinear system is solved exactly. Yet, as the nonlinear problem is typically solved iteratively, the accuracy of the solution procedure has an impact on the conservation property in practical calculations. The quality of the conservation is contingent upon the accuracy of the implicit solver, and readers may refer to an analysis made in \cite[Sec.~4.2]{Betsch2000}.

As was discussed in the introduction, one crucial issue of this integrator is that the finite difference formula $\hat{V}(\|\bm q_{n+1}\|) - \hat{V}(\|\bm q_{n}\|)/\|\bm q_{n+1}\| - \|\bm q_n\|$ becomes numerically unstable when $\|\bm q_{n+1}\|$ gets too close to $\|\bm q_n\|$. A remedy for this issue is switching to its analytic limit \cite{Janz2019}, 
\begin{align}
\label{eq:janz-mid-point-rescue}
\hat{V}'\left( \|\bm q\|_{n+\frac12} \right),
\end{align}
or \cite[p.~211]{Gonzalez1996},
\begin{align}
\label{eq:gonzalez-mid-point-rescue}
\hat{V}'\left( \|\bm q\|_{n+\frac12} \right) + \frac{1}{24} \left( \|\bm q_{n+1}\| - \|\bm q_n\| \right)^2 \hat{V}'''\left( \|\bm q\|_{n+\frac12} \right),
\end{align} 
when the difference between $\|\bm q_{n+1}\|$ and $\|\bm q_n\|$ is detected to be smaller than a prescribed tolerance.  Although the third derivative added to the mid-point evaluation of $\hat{V}'$ in \eqref{eq:gonzalez-mid-point-rescue} may improve the approximation accuracy to the quotient $\hat{V}(\|\bm q_{n+1}\|) - \hat{V}(\|\bm q_{n}\|) / \|\bm q_{n+1}\| - \|\bm q_n\|$ over \eqref{eq:janz-mid-point-rescue}, there is inevitably approximation residuals for a general form of $\hat{V}$, and the energy is thus uncontrolled when the alternate formula is switched on. This observation inspires the strategies and integrators to be discussed and analyzed subsequently.

\subsection{Integrators based on splitting and perturbation techniques}
\label{sec:integrator-split-perturbation}
In this section, we propose a suite of integrators based on the splits of the potential energy and perturbation terms introduced through special quadrature rules. We assume that the energy function $\hat{V}$ can be additively split into a convex part $\hat{V}_{\mathrm c}$ and a concave part $\hat{V}_{\mathrm e}$, both of which are twice differentiable, that is,
\begin{align}
\label{eq:convex-split}
\hat{V}(r) = \hat{V}_{\mathrm c}(r) + \hat{V}_{\mathrm e}(r), \qquad \hat{V}''_{\mathrm c}(r) \geq 0, \qquad \hat{V}''_{\mathrm e}(r) \leq 0.
\end{align}
This split will be shown to result in a first-order perturbation to the LaBudde-Greenspan integrator, which is energy decaying and momentum conserving. Further, if the energy function $\hat{V}$ can be split into a super-convex part $\hat{V}_{+}$ and a super-concave\footnote{The terminologies `super-convex' and `super-concave' were originally introduced in \cite{Liu2013}.} part $\hat{V}_{-}$, both of which are at least fourth differentiable, that is,
\begin{align}
\label{eq:super-convex-split}
\hat{V}(r) = \hat{V}_{+}(r) + \hat{V}_{-}(r), \qquad \hat{V}^{(4)}_{+}(r) \geq 0, \qquad \hat{V}^{(4)}_{-}(r) \leq 0,
\end{align}
it is feasible to improve the accuracy to second order. The integrators to be presented do not involve a quotient formula and do not suffer from unstable numerical behavior when $\|\bm q_{n+1}\|$ gets close to $\|\bm q_n\|$. It needs to be pointed out that the above splits are not necessarily unique, and how different splits may affect the numerical behavior remains an interesting topic to be pursued.

\subsubsection{The generalized Eyre integrator}
Given the convex-concave split of the energy \eqref{eq:convex-split}, an integrator can be presented as follows. In the time subinterval $\mathbb I_n$, given $\bm q_n$ and $\bm p_n$, find $\bm q_{n+1}$ and $\bm p_{n+1}$ such that, 
\begin{align}
\label{eq:convex-concave-scheme-q}
\bm q_{n+1} - \bm q_n &= \frac{\Delta t_n}{m} \bm p_{n+\frac12}, \\
\label{eq:convex-concave-scheme-p}
\bm p_{n+1} - \bm p_n &= -\Delta t_n \left( \hat{V}_{\mathrm c}'(\|\bm q_{n+1}\|) + \hat{V}_{\mathrm e}'(\|\bm q_{n}\|) \right) \frac{\bm q_{n+\frac12} }{\|\bm q\|_{n+\frac12}}.
\end{align}
To analyze the energy behavior of the above scheme, we need to leverage the rectangular quadrature rules, which are stated in the following lemma. The proof of the lemma can be found in Appendix A of \cite{Liu2013}.
\begin{lemma}[Rectangular quadrature rules]
For a function $f \in C^1[a,b]$, there exists $\xi_1, \xi_2 \in (a,b)$ such that,
\begin{align}
\label{eq:rectangular-quad-rule-1}
\int_{a}^{b} f(x)dx &= (b-a) f(b) - \frac{(b-a)^2}{2} f'(\xi_1), \\
\label{eq:rectangular-quad-rule-2}
\int_{a}^{b} f(x)dx &= (b-a) f(a) + \frac{(b-a)^2}{2} f'(\xi_2).
\end{align}
\end{lemma}
The significance of the quadrature rule is that the residuals adopt opposite signs, thus allowing us to make the residual dissipative. This quadrature rule is related to Eyre's scheme originally designed for the Cahn-Hilliard equation \cite{Eyre1998,Liu2013}, and we thus term this one as the generalized Eyre integrator.
\begin{proposition}
The integrator \eqref{eq:convex-concave-scheme-q}-\eqref{eq:convex-concave-scheme-p} is momentum conserving and energy decaying in the following sense,
\begin{align*}
\bm J(\bm q_{n+1}, \bm p_{n+1}) &= \bm J(\bm q_{n}, \bm p_{n}),\\
H(\bm q_{n+1}, \bm p_{n+1}) - H(\bm q_n, \bm p_n) &= \frac{\left(\|\bm q_{n+1}\| - \|\bm q_n\|\right)^2 }{2} \left( -\hat{V}''_{\mathrm c}(\|\bm q\|_{n+\xi_1}) +  \hat{V}''_{\mathrm e}(\|\bm q\|_{n+\xi_2}) \right) \leq 0,
\end{align*}
with $\xi_1, \xi_2 \in (0,1)$.
\end{proposition}
\begin{proof}
The discrete momentum conservation holds because the right-hand side of \eqref{eq:convex-concave-scheme-p} is apparently parallel to $\bm q_{n+\frac12}$. To demonstrate the energy decaying property, we first show that 
\begin{align*}
 & \hat{V}(\|\bm q_{n+1}\|) - \hat{V}(\|\bm q_{n}\|) = \int_{\|\bm q_n\|}^{\|\bm q_{n+1}\|} \hat{V}'(r)dr = \int_{\|\bm q_n\|}^{\|\bm q_{n+1}\|} \hat{V}_{\mathrm c}'(r) + \hat{V}_{\mathrm e}'(r) dr \\
=& \left( \|\bm q_{n+1}\| - \|\bm q_n\| \right) \left( \hat{V}'_{\mathrm c}(\|\bm q_{n+1}\|) + \hat{V}'_{\mathrm e}(\|\bm q_{n}\|) \right) + \frac{\left(\|\bm q_{n+1}\| - \|\bm q_n\|\right)^2}{2} \left( -\hat{V}''_{\mathrm c}(\|\bm q\|_{n+\xi_1}) + \hat{V}''_{\mathrm e}(\|\bm q\|_{n+\xi_2}) \right),
\end{align*}
after applying the rectangular quadrature rule \eqref{eq:rectangular-quad-rule-1} to the convex part $\hat{V}'_{\mathrm c}$ and \eqref{eq:rectangular-quad-rule-2} to the concave part $\hat{V}'_{\mathrm e}$. With the above relation, we may obtain the following according to \eqref{eq:convex-concave-scheme-q}-\eqref{eq:convex-concave-scheme-p},
\begin{align*}
& \frac{\Delta t_n}{2m}\left( \|\bm p_{n+1}\|^2 - \|\bm p_n\|^2 \right) = \left( \bm p_{n+1} - \bm p_{n} \right) \cdot \left( \bm q_{n+1} - \bm q_n \right)\\
=& - \Delta t_n \left(  \hat{V}'_{\mathrm c}(\|\bm q_{n+1}\|) + \hat{V}'_{\mathrm e}(\|\bm q_{n}\|) \right)  \frac{\left(\bm q_{n+1} + \bm q_n \right)}{\|\bm q_{n+1}\| + \|\bm q_n\|} \cdot \left( \bm q_{n+1} - \bm q_n \right) \\
=& -\Delta t_n \left(  \hat{V}'_{\mathrm c}(\|\bm q_{n+1}\|) + \hat{V}'_{\mathrm e}(\|\bm q_{n}\|) \right) \left( \|\bm q_{n+1}\| - \|\bm q_n\| \right) \\
=& - \Delta t_n \left( \hat{V}(\|\bm q_{n+1}\|) - \hat{V}(\|\bm q_{n}\|) \right) + \frac{\left(\|\bm q_{n+1}\| - \|\bm q_n\|\right)^2 \Delta t_n}{2} \left( - \hat{V}''_{\mathrm c}(\|\bm q\|_{n+\xi_1}) + \hat{V}''_{\mathrm e}(\|\bm q\|_{n+\xi_2}) \right).
\end{align*}
After reorganizing terms, one immediately has
\begin{align*}
\frac{1}{2m}\|\bm p_{n+1}\|^2 + \hat{V}(\|\bm q_{n+1}\|) = \frac{1}{2m}\|\bm p_{n}\|^2 + \hat{V}(\|\bm q_{n}\|) + \frac{\left(\|\bm q_{n+1}\| - \|\bm q_n\|\right)^2 }{2} \left( -\hat{V}''_{\mathrm c}(\|\bm q\|_{n+\xi_1}) +  \hat{V}''_{\mathrm e}(\|\bm q\|_{n+\xi_2}) \right).
\end{align*}
The term $\hat{V}''_{\mathrm c}(\|\bm q\|_{n+\xi_1}) -  \hat{V}''_{\mathrm e}(\|\bm q\|_{n+\xi_2})$ is non-positive due to the convex-concave split \eqref{eq:convex-split} of the potential energy.
\end{proof}

In the last, we analyze the accuracy of the integrator \eqref{eq:convex-concave-scheme-q}-\eqref{eq:convex-concave-scheme-p} by comparing it with a known second-order integrator.
\begin{proposition}
\label{prop:convex-concave-accuracy}
The local truncation error associated with \eqref{eq:convex-concave-scheme-q} and \eqref{eq:convex-concave-scheme-p} are denoted as $\bm \tau_{\bm q}(t)$ and $\bm \tau_{\bm p}(t)$, respectively. They are bounded by
\begin{align*}
\|\bm \tau_{\bm q}(t_n)\| \leq K_{\bm q} \Delta t_n^3, \qquad \mbox{and} \qquad \|\bm \tau_{\bm q}(t_n)\| \leq K_{\bm p} \Delta t^2_n,
\end{align*}
for all $t_n \in (0,T)$.
\end{proposition}
\begin{proof}
Let $\bm q(t)$ and $\bm p(t)$ be the exact solutions of the model problem \eqref{eq:single-particle-strong}. Invoking the Taylor expansion, one has
\begin{align}
& \bm q(t_{n+1}) - \bm q(t_n) = \frac{\Delta t}{2m} \left( \bm p(t_{n+1}) + \bm p(t_n) \right) + \mathcal{O}(\Delta t_n^3), \\
\label{eq:convex-concave-accuracy-proof-1}
& \bm p(t_{n+1}) - \bm p(t_n) = -\Delta t_n \hat{V}'(\|\bm q(t_{n+\frac12})\|) \frac{\bm q(t_{n+1}) + \bm q(t_n)}{\|\bm q(t_{n+1})\| + \|\bm q(t_n)\|} + \mathcal{O}(\Delta t_n^3).
\end{align}
Introducing the notation $\tilde{U}(t) := \hat{V}'(\|\bm q(t)\|)$, one has the following relations,
\begin{align*}
& \hat{V}'_{\mathrm c}(\|\bm q(t_{n+1})\|)=\tilde{U}_{\mathrm c}(t_{n+1}) = \tilde{U}_{\mathrm c}(t_{n+\frac12}) + \tilde{U}'_{\mathrm c}(t_{n+\frac12})\frac{\Delta t_n}{2} + \tilde{U}''_{\mathrm c}(t_{n+\frac12})\frac{\Delta t^2_n}{4} + \mathcal O(\Delta t_n^3), \\
& \hat{V}'_{\mathrm e}(\|\bm q(t_{n})\|)=\tilde{U}_{\mathrm e}(t_{n}) = \tilde{U}_{\mathrm e}(t_{n+\frac12}) - \tilde{U}'_{\mathrm e}(t_{n+\frac12})\frac{\Delta t_n}{2} + \tilde{U}''_{\mathrm e}(t_{n+\frac12})\frac{\Delta t^2_n}{4} + \mathcal O(\Delta t_n^3).
\end{align*}
Combining the two results in
\begin{align}
\label{eq:convex-concave-accuracy-proof-2}
\hat{V}'_{\mathrm c}(\|\bm q(t_{n+1})\|) + \hat{V}'_{\mathrm e}(\|\bm q(t_{n})\|) = \hat{V}(\|\bm q(t_{n+\frac12})\|) + \left( \tilde{U}'_{\mathrm c}(t_{n+\frac12}) - \tilde{U}'_{\mathrm e}(t_{n+\frac12}) \right)\frac{\Delta t_n}{2} + \mathcal O(\Delta t_n^2).
\end{align}
Substituting \eqref{eq:convex-concave-accuracy-proof-2} into \eqref{eq:convex-concave-accuracy-proof-1} leads to
\begin{align*}
\bm p(t_{n+1}) - \bm p(t_n) = -\Delta t_n \left( \hat{V}_{\mathrm c}'(\|\bm q(t_{n+1})\|) + \hat{V}_{\mathrm e}'(\|\bm q(t_{n})\|)  \right)\frac{\bm q(t_{n+1}) + \bm q(t_n)}{\|\bm q(t_{n+1})\| + \|\bm q(t_n)\|} + \mathcal{O}(\Delta t_n^2),
\end{align*}
which completes the proof.
\end{proof}
Proposition \ref{prop:convex-concave-accuracy} indicates that the application of the rectangular quadrature rule results in a loss of the accuracy, and the integrator \eqref{eq:convex-concave-scheme-q}-\eqref{eq:convex-concave-scheme-p} is only first-order accurate. In the subsequent discussion, we will introduce higher-order quadrature rules that maintain the second-order accuracy with the sign of the residual controlled.

\subsubsection{Perturbed mid-point integrator}
In this section, we present an integrator based on the energy split \eqref{eq:super-convex-split}, with the super-convex part $\hat{V}_{+}$ and the super-concave part $\hat{V}_{-}$. In each time step, given $\bm q_{n}$ and $\bm p_{n}$, determine $\bm q_{n+1}$ and $\bm p_{n+1}$ from the following equations,
\begin{align}
\label{eq:super-convex-concave-scheme-q}
\bm q_{n+1} - \bm q_n &= \frac{\Delta t_n}{m} \bm p_{n+\frac12}, \\
\label{eq:super-convex-concave-scheme-p}
\bm p_{n+1} - \bm p_n &= - \Delta t_n \left( \hat{V}'(\|\bm q\|_{n+\frac12}) + \frac{\left(\| \bm q_{n+1}\| - \|\bm q_n \| \right)^2}{24} \left( \hat{V}_{+}'''(\|\bm q_{n+1}\|) + \hat{V}_{-}'''(\|\bm q_{n}\|) \right) \right) \frac{\bm q_{n+\frac12}}{\|\bm q\|_{n+\frac12}}.
\end{align}
In \eqref{eq:super-convex-concave-scheme-p}, the conservative force is represented in terms of $\hat{V}'(\|\bm q\|_{n+1/2})$ with a perturbation term, and the perturbation term vanishes for linear or quadratic potential energies. It is worthy to note that $\|\bm q\|_{n+1/2} \neq \|\bm q_{n+1/2}\|$, which differentiates the above integrator from the mid-point integrator \eqref{eq:mid-point-q-single-particle}-\eqref{eq:mid-point-p-single-particle}. The quadrature rule related to the integrator above is a suite of perturbed mid-point rules stated in the following lemma. This quadrature was originally designed for an entropy-stable time integration scheme for the van der Waals fluids \cite{Liu2013,Liu2015}, and its proof can be found in Appendix B of \cite{Liu2013}. 
\begin{lemma}
\label{lemma:perturbed-mid-point}
For a function $f\in C^3([a,b])$, there exists $\xi_1,\xi_2 \in (a,b)$ such that the following quadrature rules hold,
\begin{align}
\label{eq:perturbed-mid-point-a}
\int_a^b f(x)dx &= (b-a) f(\frac{a+b}{2}) + \frac{(b-a)^3}{24}f''(b) - \frac{(b-a)^4}{48}f'''(\xi_1), \\
\label{eq:perturbed-mid-point-b}
\int_a^b f(x)dx &= (b-a) f(\frac{a+b}{2}) + \frac{(b-a)^3}{24}f''(a) + \frac{(b-a)^4}{48}f'''(\xi_2).
\end{align}
\end{lemma}

The motivation of the above quadrature comes from the observation that the mid-point rule has a residual term $(b-a)^3f''(\tilde{\xi})/24$ with $\tilde{\xi} \in (a,b)$. Therefore, perturbing the mid-point rule by a similar term is anticipated to improve the accuracy. Collocating the perturbation term at the two end points leads to fourth-order residual terms with opposite signs, which turns out to be quite beneficial in ensuring energy stability. Let us state and prove the properties of the integrator in the following.

\begin{proposition}
The scheme \eqref{eq:super-convex-concave-scheme-q}-\eqref{eq:super-convex-concave-scheme-p} is momentum conserving and energy decaying, i.e.,
\begin{align}
\bm J(\bm q_{n+1}, \bm p_{n+1}) &= \bm J(\bm q_{n}, \bm p_{n}), \\
H(\bm q_{n+1}, \bm p_{n+1}) - H(\bm q_n, \bm p_n) &= \frac{\left(\|\bm q_{n+1}\| - \|\bm q_n\|\right)^4 }{2} \left( -\hat{V}^{(4)}_{+}(\|\bm q\|_{n+\xi_1}) +  \hat{V}^{(4)}_{-}(\|\bm q\|_{n+\xi_2}) \right) \leq 0,
\end{align}
with $\xi_1, \xi_2 \in (0,1)$.
\end{proposition}
\begin{proof}
The momentum conservation property follows straightforwardly from the identity \eqref{eq:momentum_identity}. To demonstrate the energy dissipation property, we analyze the jump of the potential energy as follows, 
\begin{align}
\label{eq:proof-perturbed-mid-point-1}
 & \hat{V}(\|\bm q_{n+1}\|) - \hat{V}(\|\bm q_{n}\|) = \int_{\|\bm q_n\|}^{\|\bm q_{n+1}\|} \hat{V}'(r)dr = \int_{\|\bm q_n\|}^{\|\bm q_{n+1}\|} \hat{V}_{+}'(r) + \hat{V}_{-}'(r) dr \nonumber \\
=& \left( \|\bm q_{n+1}\| - \|\bm q_n\| \right) \left( \hat{V}'_{+}(\|\bm q\|_{n+\frac12}) + \hat{V}'_{-}(\|\bm q\|_{n+\frac12}) \right) + \frac{\left(\|\bm q_{n+1}\| - \|\bm q_n\|\right)^3}{24} \left( \hat{V}'''_{+}(\|\bm q_{n+1}\|) +  \hat{V}'''_{-}(\|\bm q_{n}\|) \right) \nonumber \\
&+ \frac{\left(\|\bm q_{n+1}\| - \|\bm q_n\|\right)^4}{48} \left( -\hat{V}^{(4)}_{+}(\|\bm q\|_{n+\xi_1}) +  \hat{V}^{(4)}_{-}(\|\bm q\|_{n+\xi_2}) \right).
\end{align}
In the above, we have applied the quadrature rule \eqref{eq:perturbed-mid-point-a} to $\hat{V}'_{+}$ and \eqref{eq:perturbed-mid-point-b} to $\hat{V}'_{-}$, respectively. Noticing that $\hat{V}(\|\bm q\|_{n+\frac12}) = \hat{V}'_{+}(\|\bm q\|_{n+\frac12}) + \hat{V}'_{-}(\|\bm q\|_{n+\frac12})$, one has
\begin{align*}
\hat{V}'(\|\bm q\|_{n+\frac12}) +& \frac{\left(\|\bm q_{n+1}\| - \|\bm q_n\|\right)^2}{24} \left( \hat{V}'''_{+}(\|\bm q_{n+1}\|) +  \hat{V}'''_{-}(\|\bm q_{n}\|) \right)  = \frac{\hat{V}(\|\bm q_{n+1}\|) - \hat{V}(\|\bm q_{n}\|)}{\|\bm q_{n+1}\| - \|\bm q_n\|} \\
& + \frac{\left(\|\bm q_{n+1}\| - \|\bm q_n\|\right)^3}{48} \left( \hat{V}^{(4)}_{+}(\|\bm q\|_{n+\xi_1}) -  \hat{V}^{(4)}_{-}(\|\bm q\|_{n+\xi_2}) \right),
\end{align*}
after rearranging terms in \eqref{eq:proof-perturbed-mid-point-1}. With the above, one may obtain the following relations from \eqref{eq:super-convex-concave-scheme-q}-\eqref{eq:super-convex-concave-scheme-p},
\begin{align*}
& \frac{\Delta t_n}{2m}\left( \|\bm p_{n+1}\|^2 - \|\bm p_n\|^2 \right) = \left( \bm p_{n+1} - \bm p_n \right) \cdot \left( \bm q_{n+1} - \bm q_n \right) \\
=& -\Delta t_n \left( \hat{V}'(\|\bm q\|_{n+\frac12}) + \frac{\left(\| \bm q_{n+1}\| - \|\bm q_n \| \right)^2}{24} \left( \hat{V}_{+}'''(\|\bm q_{n+1}\|) + \hat{V}_{-}'''(\|\bm q_{n}\|) \right) \right) \left( \|\bm q_{n+1}\| - \|\bm q_n\| \right) \\
=& -\Delta t_n \left( \hat{V}(\|\bm q_{n+1}\|) - \hat{V}(\|\bm q_{n}\|) + \frac{\left(\|\bm q_{n+1}\| - \|\bm q_n\|\right)^4}{48} \left( \hat{V}^{(4)}_{+}(\|\bm q\|_{n+\xi_1}) -  \hat{V}^{(4)}_{-}(\|\bm q\|_{n+\xi_2}) \right) \right).
\end{align*}
Reorganizing terms leads to
\begin{align*}
H(\bm q_{n+1}, \bm p_{n+1}) - H(\bm q_n, \bm p_n) = \frac{\Delta t_n}{48} \left(\|\bm q_{n+1}\| - \|\bm q_n\|\right)^4 \left( - \hat{V}^{(4)}_{+}(\|\bm q\|_{n+\xi_1}) +  \hat{V}^{(4)}_{-}(\|\bm q\|_{n+\xi_2}) \right) \leq 0,
\end{align*}
which completes the proof.
\end{proof}

In the last, let us analyze the accuracy of the integrator \eqref{eq:super-convex-concave-scheme-q}-\eqref{eq:super-convex-concave-scheme-p}.
\begin{proposition}
\label{prop:perturbed-mid-point-accuracy}
The local truncation errors associated with \eqref{eq:super-convex-concave-scheme-q} and \eqref{eq:super-convex-concave-scheme-p} are denoted as $\bm \tau_{\bm q}(t)$ and $\bm \tau_{\bm p}(t)$, respectively. They are bounded by
\begin{align*}
\|\bm \tau_{\bm q}(t_n)\| \leq \hat{K}_{\bm q} \Delta t_n^3, \qquad \mbox{and} \qquad \|\bm \tau_{\bm q}(t_n)\| \leq \hat{K}_{\bm p} \Delta t^3_n,
\end{align*}
for all $t_n \in (0,T)$.
\end{proposition}
\begin{proof}
Let $\bm q(t)$ and $\bm p(t)$ be the exact solutions of the model problem \eqref{eq:single-particle-strong}. Invoking the Taylor expansion, one has
\begin{align}
& \bm q(t_{n+1}) - \bm q(t_n) = \frac{\Delta t}{2m} \left( \bm p(t_{n+1}) + \bm p(t_n) \right) + \mathcal{O}(\Delta t_n^3), \\
\label{eq:perturbed-mid-point-accuracy-proof-1}
& \bm p(t_{n+1}) - \bm p(t_n) = -\Delta t_n \hat{V}'(\|\bm q(t_{n+\frac12})\|) \frac{\bm q(t_{n+1}) + \bm q(t_n)}{\|\bm q(t_{n+1})\| + \|\bm q(t_n)\|} + \mathcal{O}(\Delta t_n^3).
\end{align}
Introducing the notation $\tilde{W}(t) := \|\bm q(t)\|$, one may get the following,
\begin{align*}
& \|\bm q(t_{n+1})\|=\tilde{W}(t_{n+1}) = \tilde{W}(t_{n+\frac12}) + \tilde{W}'(t_{n+\frac12})\frac{\Delta t_n}{2} + \tilde{W}''(t_{n+\frac12})\frac{\Delta t^2_n}{4} + \mathcal O(\Delta t_n^3), \\
& \|\bm q(t_{n})\|=\tilde{W}(t_{n}) = \tilde{W}(t_{n+\frac12}) - \tilde{W}'(t_{n+\frac12})\frac{\Delta t_n}{2} + \tilde{W}''(t_{n+\frac12})\frac{\Delta t^2_n}{4} + \mathcal O(\Delta t_n^3).
\end{align*}
Combining the above two relations, it can be shown that
\begin{align*}
\|\bm q(t_{n+1})\| - \|\bm q(t_n)\| = \tilde{W}'(t_{n+\frac12}) \Delta t_n + \mathcal O(\Delta t_n^3),
\end{align*}
and
\begin{align*}
-\Delta t_n \frac{\left(\|\bm q(t_{n+1})\|- \|\bm q(t_n)\|\right)^2}{24} \left( \hat{V}_{+}'''(\|\bm q(t_{n+1})\|) + \hat{V}_{-}'''(\|\bm q(t_n)\|) \right) = \mathcal O(\Delta t_n^3),
\end{align*}
which completes the proof.
\end{proof}

\subsubsection{Perturbed trapezoidal integrator}
In this section, we present a third integrator based on the super-convex and super-concave splits of the energy $\hat{V}$. In each time step, given $\bm q_{n}$ and $\bm p_{n}$, determine $\bm q_{n+1}$ and $\bm p_{n+1}$ from the following equations,
\begin{align}
\label{eq:super-convex-concave-trapezoidal-scheme-q}
\bm q_{n+1} - \bm q_n =& \frac{\Delta t_n}{m} \bm p_{n+\frac12}, \\
\label{eq:super-convex-concave-trapezoidal-scheme-p}
\bm p_{n+1} - \bm p_n =& - \Delta t_n \bigg( \frac{ \hat{V}'(\|\bm q_{n}\|) + \hat{V}'(\|\bm q_{n+1}\|)}{2} \nonumber \\
& \quad - \frac{\left(\| \bm q_{n+1}\| - \|\bm q_n \| \right)^2}{12} \left( \hat{V}_{+}'''(\|\bm q_{n}\|) + \hat{V}_{-}'''(\|\bm q_{n+1}\|) \right) \bigg) \frac{\bm q_{n+\frac12}}{\|\bm q\|_{n+\frac12}}.
\end{align}
The above can be viewed as a modified trapezoidal scheme with additional terms involving third derivatives of $\hat{V}_{+}$ and $\hat{V}_{-}$. The quadrature rules associated with the above scheme were originally introduced for the design of an energy-stable scheme for the Cahn-Hilliard equation \cite{Gomez2011}, and it was later successfully applied to non-gradient-flow systems, such as the isothermal \cite{Liu2013} and thermal Navier-Stokes-Korteweg equations \cite{Liu2015}. We state the quadrature rules as Lemma \ref{lemma:perturbed-trapezoidal-rules}, whose proof can be found in Appendix A of \cite{Gomez2011}.
\begin{lemma}
\label{lemma:perturbed-trapezoidal-rules}
For a function $f\in C^3([a,b])$, there exists $\xi_1,\xi_2 \in (a,b)$ such that the following quadrature rules hold,
\begin{align}
\label{eq:perturbed-trapezoidal-a}
\int_a^b f(x)dx &= \frac{(b-a)}{2} \left( f(a) + f(b) \right) - \frac{(b-a)^3}{12}f''(b) + \frac{(b-a)^4}{24}f'''(\xi_1), \\
\label{eq:perturbed-trapezoidal-b}
\int_a^b f(x)dx &= \frac{(b-a)}{2} \left( f(a) + f(b) \right) - \frac{(b-a)^3}{12}f''(a) - \frac{(b-a)^4}{24}f'''(\xi_2).
\end{align}
\end{lemma}

With the above lemma, we give the main results of the temporal scheme \eqref{eq:super-convex-concave-trapezoidal-scheme-q}-\eqref{eq:super-convex-concave-trapezoidal-scheme-p} in the following proposition.
\begin{proposition}
\label{proposition:trapezoidal-properties}
The scheme \eqref{eq:super-convex-concave-trapezoidal-scheme-q}-\eqref{eq:super-convex-concave-trapezoidal-scheme-p} is momentum conserving and energy decaying,
\begin{align}
\bm J(\bm q_{n+1}, \bm p_{n+1}) &= \bm J(\bm q_{n}, \bm p_{n}), \\
H(\bm q_{n+1}, \bm p_{n+1}) - H(\bm q_n, \bm p_n) &= \frac{\left(\|\bm q_{n+1}\| - \|\bm q_n\|\right)^4 }{2} \left( \hat{V}^{(4)}_{+}(\|\bm q\|_{n+\xi_1}) -  \hat{V}^{(4)}_{-}(\|\bm q\|_{n+\xi_2}) \right) \leq 0,
\end{align}
with $\xi_1,\xi_2 \in (0,1)$.
\end{proposition}
\begin{proof}
The momentum-conserving property holds because the right-hand side of \eqref{eq:perturbed-trapezoidal-b} is parallel to  $\bm q_{n+\frac12}$. Applying the trapezoidal rule \eqref{eq:perturbed-trapezoidal-b} to the function $\hat{V}'_{+}$ and \eqref{eq:perturbed-trapezoidal-a} to the function $\hat{V}'_{-}$ leads to the following,
\begin{align*}
 & \hat{V}(\|\bm q_{n+1}\|) - \hat{V}(\|\bm q_{n}\|) = \int_{\|\bm q_n\|}^{\|\bm q_{n+1}\|} \hat{V}'(r)dr = \int_{\|\bm q_n\|}^{\|\bm q_{n+1}\|} \hat{V}_{+}'(r) + \hat{V}_{-}'(r) dr \\
=& \frac{\|\bm q_{n+1}\| - \|\bm q_n\|}{2} \left( \hat{V}'_{+}(\|\bm q_{n}\|) + \hat{V}'_{+}(\|\bm q_{n+1}\|) \right) - \frac{\left(\|\bm q_{n+1}\| - \|\bm q_n\|\right)^3}{12} \hat{V}'''_{+}(\|\bm q_{n}\|) \\
&+ \frac{\|\bm q_{n+1}\| - \|\bm q_n\|}{2} \left( \hat{V}'_{-}(\|\bm q_{n}\|) + \hat{V}'_{-}(\|\bm q_{n+1}\|) \right) - \frac{\left(\|\bm q_{n+1}\| - \|\bm q_n\|\right)^3}{12} \hat{V}'''_{-}(\|\bm q_{n+1}\|) \\
& - \frac{\left(\|\bm q_{n+1}\| - \|\bm q_n\|\right)^4}{24} \left( \hat{V}^{(4)}_{+}(\|\bm q\|_{n+\xi_1}) - \hat{V}^{(4)}_{+}(\|\bm q\|_{n+\xi_2}) \right) \\
=& \frac{\|\bm q_{n+1}\| - \|\bm q_n\|}{2} \left( \hat{V}'(\|\bm q_{n}\|) + \hat{V}'(\|\bm q_{n+1}\|) \right) - \frac{\left(\|\bm q_{n+1}\| - \|\bm q_n\|\right)^3}{12} \left( \hat{V}'''_{+}(\|\bm q_{n}\|) + \hat{V}'''_{-}(\|\bm q_{n+1}\|) \right) \\
&-\frac{\left(\|\bm q_{n+1}\| - \|\bm q_n\|\right)^4}{24} \left( \hat{V}^{(4)}_{+}(\|\bm q\|_{n+\xi_1}) - \hat{V}^{(4)}_{+}(\|\bm q\|_{n+\xi_2}) \right).
\end{align*}
With the aid of the above relation, one may obtain the following from \eqref{eq:super-convex-concave-trapezoidal-scheme-q}-\eqref{eq:super-convex-concave-trapezoidal-scheme-p},
\begin{align*}
& \frac{\Delta t_n}{2m}\left( \|\bm p_{n+1}\|^2 - \|\bm p_n\|^2 \right) = \left( \bm p_{n+1} - \bm p_n \right) \cdot \left( \bm q_{n+1} - \bm q_n \right) \displaybreak[2] \\
=& -\Delta t_n \left( \frac{ \hat{V}'(\|\bm q_{n}\|) + \hat{V}'(\|\bm q_{n+1}\|)}{2} - \frac{\left(\| \bm q_{n+1}\| - \|\bm q_n \| \right)^2}{12} \left( \hat{V}_{+}'''(\|\bm q_{n}\|) + \hat{V}_{-}'''(\|\bm q_{n+1}\|) \right) \right) \left( \|\bm q_{n+1}\| - \|\bm q_n\|  \right) \displaybreak[2] \\
=& -\Delta t_n \left( \hat{V}(\|\bm q_{n+1}\|) - \hat{V}(\|\bm q_{n}\|) + \frac{\left(\|\bm q_{n+1}\| - \|\bm q_n\|\right)^4}{24} \left( \hat{V}^{(4)}_{+}(\|\bm q\|_{n+\xi_1}) -  \hat{V}^{(4)}_{-}(\|\bm q\|_{n+\xi_2}) \right) \right).
\end{align*}
Reorganization of the above leads to
\begin{align*}
H(\bm q_{n+1}, \bm p_{n+1}) - H(\bm q_n, \bm p_n) = -\frac{\left(\|\bm q_{n+1}\| - \|\bm q_n\|\right)^4}{24} \left( \hat{V}^{(4)}_{+}(\|\bm q\|_{n+\xi_1}) -  \hat{V}^{(4)}_{-}(\|\bm q\|_{n+\xi_2}) \right) \leq 0,
\end{align*}
which completes the proof.
\end{proof}

In the last, we demonstrate that the integrator is second-order accurate by analyzing its local truncation error.
\begin{proposition}
\label{prop:perturbed-trapezoidal-accuracy}
The local truncation errors associated with \eqref{eq:super-convex-concave-trapezoidal-scheme-q} and \eqref{eq:super-convex-concave-trapezoidal-scheme-p} are denoted as $\bm \tau_{\bm q}(t)$ and $\bm \tau_{\bm p}(t)$, respectively. They are bounded by
\begin{align*}
\|\bm \tau_{\bm q}(t_n)\| \leq \tilde{K}_{\bm q} \Delta t_n^3, \qquad \mbox{and} \qquad \|\bm \tau_{\bm q}(t_n)\| \leq \tilde{K}_{\bm p} \Delta t^3_n,
\end{align*}
for all $t_n \in (0,T)$.
\end{proposition}
\begin{proof}
From the proof of Proposition \ref{prop:perturbed-mid-point-accuracy}, we have
\begin{align*}
\|\bm q(t_{n+1})\| - \|\bm q(t_n)\| = \tilde{W}'(t_{n+\frac12}) \Delta t_n + \mathcal O(\Delta t_n^3),
\end{align*}
which indicates that 
\begin{align*}
\Delta t_n \frac{\left( \|\bm q(t_{n+1})\| - \|\bm q(t_n)\| \right)^2}{12} \left( \hat{V}_{+}'''(\|\bm q(t_n)\|) + \hat{V}_{-}'''(\|\bm q(t_{n+1})\|) \right) = \mathcal O(\Delta t_n^3).
\end{align*}
Therefore, the integrator \eqref{eq:super-convex-concave-trapezoidal-scheme-q}-\eqref{eq:super-convex-concave-trapezoidal-scheme-p} is the trapezoidal rule with a second-order perturbation and maintains the second-order time accuracy.
\end{proof}

In this section, three integrators have been presented and analyzed. They are all provably energy-decaying and momentum-conserving. The latter two (i.e., the perturbed mid-point and perturbed trapezoidal integrators) maintain the second-order accuracy. Importantly, the three proposed integrators do not involve the quotient formula and thus do not suffer from numerical instabilities when $\|\bm q_{n+1}\|$ gets close to $\|\bm q_n\|$. This attribute makes them attractive as alternate options for the existing conserving integrators.

\subsection{Algorithm implementation}
\label{sec:algorithm_implementation}
Here, we first present a predictor multi-corrector algorithm for the proposed integrators. In the energy-conserving and energy-decaying integrators, the difference of the residuals resides in the formula for the conservative force, and thus we may state the residuals as follows,
\begin{align*}
& \boldsymbol{\mathrm R}_{\bm q}\left( \bm q_{n+1}, \bm p_{n+1}, \bm q_{n}, \bm p_n \right) := \bm q_{n+1} - \bm q_n - \frac{\Delta t_n}{m} \bm p_{n+\frac12},\\
& \boldsymbol{\mathrm R}^{\star}_{\bm p}\left( \bm q_{n+1}, \bm p_{n+1}, \bm q_{n}, \bm p_n \right) := \bm p_{n+1} - \bm p_n + \Delta t_n \Lambda^{\star}\left( \bm q_{n+1}, \bm p_{n+1}, \bm q_{n}, \bm p_n \right) \frac{\bm q_{n+\frac12}}{\|\bm q\|_{n+\frac12}},
\end{align*}
in which the discrete conservative force is denoted by $\Lambda^{\star}\left( \bm q_{n+1}, \bm p_{n+1}, \bm q_{n}, \bm p_n \right)$. For example, if the residual is based on the LaBudde-Greenspan integrator, we have
\begin{align*}
\Lambda^{\star}\left( \bm q_{n+1}, \bm p_{n+1}, \bm q_{n}, \bm p_n \right) := \frac{\hat{V}(\|\bm q_{n+1}\|) - \hat{V}(\|\bm q_n\|)}{\|\bm q_{n+1}\| - \|\bm q_n\|};
\end{align*}
if the residual is based on the perturbed mid-point integrator, we have,
\begin{align*}
\Lambda^{\star}\left( \bm q_{n+1}, \bm p_{n+1}, \bm q_{n}, \bm p_n \right) := \hat{V}'(\|\bm q\|_{n+\frac12}) + \frac{\left(\| \bm q_{n+1}\| - \|\bm q_n \| \right)^2}{24} \left( \hat{V}_{+}'''(\|\bm q_{n+1}\|) + \hat{V}_{-}'''(\|\bm q_{n}\|) \right).
\end{align*}
The consistent tangent of the nonlinear system is written as
\begin{align*}
& \boldsymbol{\mathrm A} := \frac{\partial \boldsymbol{\mathrm R}_{\bm q}\left( \bm q_{n+1}, \bm p_{n+1}, \bm q_{n}, \bm p_n \right)}{\partial \bm q_{n+1}} = \boldsymbol{\mathrm I}_3, \quad \boldsymbol{\mathrm B} := \frac{\partial \boldsymbol{\mathrm R}_{\bm q}\left( \bm q_{n+1}, \bm p_{n+1}, \bm q_{n}, \bm p_n \right)}{\partial \bm p_{n+1}} = -\frac{\Delta t_n}{2m} \boldsymbol{\mathrm I}_3, \displaybreak[2] \\
& \boldsymbol{\mathrm C}^{\star} := \frac{\partial \boldsymbol{\mathrm R}^{\star}_{\bm p}\left( \bm q_{n+1}, \bm p_{n+1}, \bm q_{n}, \bm p_n \right)}{\partial \bm q_{n+1}} = \Delta t_n \left( \frac{\bm q_{n+\frac12}}{\|\bm q\|_{n+\frac12}} \otimes \frac{\partial \Lambda^{\star}}{\partial \bm q_{n+1}} + \frac{\Lambda^{\star}}{2\|\bm q\|_{n+\frac12}} \boldsymbol{\mathrm I}_3 - \frac{\Lambda^{\star} \bm q_{n+\frac12} \otimes \bm q_{n+1}}{2\|\bm q\|^{2}_{n+\frac12}\|\bm q_{n+1}\|} \right), \displaybreak[2] \\
& \boldsymbol{\mathrm D} := \frac{\partial \boldsymbol{\mathrm R}_{\bm p}\left( \bm q_{n+1}, \bm p_{n+1}, \bm q_{n}, \bm p_n \right)}{\partial \bm p_{n+1}} = \boldsymbol{\mathrm I}_3.
\end{align*}
The nonlinear equations are solved in each time step iteratively by the Newton-Raphson method. We use $\bm q_{n+1,(l)}$ and $\bm p_{n+1,(l)}$ to denote the solutions obtained at the subinterval $\mathbb I_n$ and the iteration step $l$, with $0 \leq l \leq l_{\mathrm{max}}$. If we further introduce the following notation,
\begin{align*}
\boldsymbol{\mathrm R}_{\bm q,n+1,(l)} :=& \boldsymbol{\mathrm R}_{\bm q}\left( \bm q_{n+1,(l)}, \bm p_{n+1,(l)}, \bm q_{n}, \bm p_n \right), \\
\boldsymbol{\mathrm R}^{\star}_{\bm p,n+1,(l)} :=& \boldsymbol{\mathrm R}^{\star}_{\bm p}\left( \bm q_{n+1}, \bm p_{n+1}, \bm q_{n}, \bm p_n \right), \\
\boldsymbol{\mathrm C}^{\star}_{n+1,(l)} :=& \frac{\partial \boldsymbol{\mathrm R}_{\bm p}\left( \bm q_{n+1,(l)}, \bm p_{n+1,(l)}, \bm q_{n}, \bm p_n \right)}{\partial \bm q_{n+1}},
\end{align*}
in each Newton-Raphson iteration, we need to determine $\Delta \bm q_{n+1,(l)}$ and $\Delta \bm p_{n+1,(l)}$ from the following system,
\begin{align*}
\begin{bmatrix}
\boldsymbol{\mathrm I}_3 & -\frac{\Delta t_n}{2m}\boldsymbol{\mathrm I}_3 \\[0.6mm]
\boldsymbol{\mathrm C}^{\star}_{n+1,(l)} & \boldsymbol{\mathrm I}_3
\end{bmatrix}
\begin{bmatrix}
\Delta \bm q_{n+1,(l)} \\[1.2mm]
\Delta \bm p_{n+1,(l)}
\end{bmatrix}
= -
\begin{bmatrix}
\boldsymbol{\mathrm R}_{\bm q,n+1,(l)}  \\[1.2mm]
\boldsymbol{\mathrm R}^{\star}_{\bm p,n+1,(l)} 
\end{bmatrix}.
\end{align*}
Due to the particular structure of the above block system, the incrementals can be solved in a segregated manner. First, the incremental $\Delta \bm q_{n+1,(l)}$ can be obtained by solving the system,
\begin{align*}
& \left( \boldsymbol{\mathrm I}_3 + \frac{\Delta t_n}{2m} \boldsymbol{\mathrm C}^{\star}_{n+1,(l)}  \right) \Delta \bm q_{n+1,(l)} = -\boldsymbol{\mathrm R}_{\bm q,n+1,(l)} - \frac{\Delta t_n}{2m} \boldsymbol{\mathrm R}^{\star}_{\bm p,n+1,(l)}.
\end{align*}
Second, the incremental $\Delta \bm p_{n+1,(l)}$ can be determined by
\begin{align*}
\Delta \bm p_{n+1,(l)} = \frac{2m}{\Delta t_n} \left( \boldsymbol{\mathrm R}_{\bm q,n+1,(l)}  + \Delta \bm q_{n+1,(l)} \right).
\end{align*}
The above discussion is summarized in the following algorithm, in which we use 
\begin{align*}
\|\left( \boldsymbol{\mathrm R}_{\bm q,n+1,(l-1)}; \boldsymbol{\mathrm R}^{\star}_{\bm p,n+1,(l-1)} \right)\|
\end{align*}
as the residual norm. We assign two tolerances $\mathrm{tol}_{\mathrm R}$ and $\mathrm{tol}_{\mathrm A}$ and the maximum number of iterations $l_{\mathrm{max}}$ to define the stopping criteria.

\noindent \textbf{Predictor:}
\begin{enumerate}
\item Set $\bm p_{n+1,(0)} = \bm p_n$ and $\bm q_{n+1,(0)} = \bm q_{n}$.

\item Assemble the matrix $\boldsymbol{\mathrm C}^{\star}_{n+1,(0)}$ and the residual vectors $\boldsymbol{\mathrm R}_{\bm q,n+1,(0)}$, $\boldsymbol{\mathrm R}^{\star}_{\bm p,n+1,(0)}$.

\item Record the initial residual norm $\|( \boldsymbol{\mathrm R}_{\bm q,n+1,(0)}; \boldsymbol{\mathrm R}^{\star}_{\bm p,n+1,(0)} ) \|$.
\end{enumerate}

\noindent \textbf{Multi-corrector:} Repeat the following steps for $l=1, \cdots, l_{\mathrm{max}}$.
\begin{enumerate}
\item Solve the linear system
\begin{align*}
& \left( \boldsymbol{\mathrm I}_3 + \frac{\Delta t_n}{2m} \boldsymbol{\mathrm C}^{\star}_{n+1,(l-1)}  \right) \Delta \bm q_{n+1,(l-1)} = -\boldsymbol{\mathrm R}_{\bm q,n+1,(l-1)} - \frac{\Delta t_n}{2m} \boldsymbol{\mathrm R}^{\star}_{\bm p,n+1,(l-1)}.
\end{align*}

\item Update the iterates as
\begin{align*}
\bm q_{n+1,(l)} = \bm q_{n+1,(l-1)} + \Delta \bm q_{n+1,(l-1)}, \quad
\bm p_{n+1,(l)} = \bm p_{n+1,(l-1)} + \frac{2m}{\Delta t_n} \left( \boldsymbol{\mathrm R}_{\bm q,n+1,(l-1)}  + \Delta \bm q_{n+1,(l-1)} \right).
\end{align*}

\item Assemble the matrix $\boldsymbol{\mathrm C}^{\star}_{n+1,(l)}$ and the residual vectors $\boldsymbol{\mathrm R}_{\bm q,n+1,(l)}$, $\boldsymbol{\mathrm R}^{\star}_{\bm p,n+1,(l)}$ using $\bm q_{n+1,(l)}$ and $\bm p_{n+1,(l)}$.

\item If either one of the stopping criteria
\begin{align*}
\frac{\|( \boldsymbol{\mathrm R}_{\bm q,n+1,(l-1)}; \boldsymbol{\mathrm R}^{\star}_{\bm p,n+1,(l-1)} ) \|}{\|( \boldsymbol{\mathrm R}_{\bm q,n+1,(0)}; \boldsymbol{\mathrm R}^{\star}_{\bm p,n+1,(0)} ) \|} \leq \mathrm{tol}_{\mathrm R}, \qquad \|( \boldsymbol{\mathrm R}_{\bm q,n+1,(l-1)}; \boldsymbol{\mathrm R}^{\star}_{\bm p,n+1,(l-1)} ) \| \leq \mathrm{tol}_{\mathrm A}
\end{align*}
is satisfied, set the solution at the time step $t_{n+1}$ as $\bm q_{n+1} = \bm q_{n+1,(l)}$ and $\bm p_{n+1} = \bm p_{n+1,(l)}$ and exit the multi-corrector stage.
\end{enumerate}

In the above algorithm, an integrator can be implemented to integrate the dynamics independently. For the LaBudde-Greenspan integrator, we need to be cautious because the predictor will inevitably cause the denominator to be zero in its quotient formula. Therefore, numerical instability will arise in each Newton-Raphson iteration. To address this issue, a prescribed tolerance $\mathrm{tol}_{\mathrm Q}$ is given, and an alternate integrator will be invoked when the denominator is detected to be smaller than $\mathrm{tol}_{\mathrm Q}$. In this work, we follow the approach adopted in \cite{Janz2019} and invoke the definition \eqref{eq:janz-mid-point-rescue} for $\Lambda^{\star}$, that is,
\begin{align}
\label{eq:labudde-greenspan-default-switch}
\Lambda^{\star} = \hat{V}'(\|\bm q\|_{n+\frac12})
\end{align}
as the \textit{default} option in the LaBudde-Greenspan integrator. Also, we set $\mathrm{tol}_{\mathrm Q} = 10^{-8}$, unless otherwise specified. This choice of $\Lambda^{\star}$ is convenient and intuitive. Yet, there is no guarantee in the conservation property when this option is invoked. In fact, the above definition of $\Lambda^{\star}$ may trigger energy accumulation and break the conservation property, as will be shown in the numerical examples. Based on the new integrators, we propose three more options for $\Lambda^{\star}$ when $|\|\bm q_{n+1,(l)}\| - \|\bm q_{n}\|| \leq \mathrm{tol}_{\mathrm Q}$. They include the definition for $\Lambda^{\star}$ based on the generalized Eyre, perturbed mid-point, or perturbed trapezoidal schemes. Combining with either one of the three formulas, the LaBudde-Greenspan integrator can guarantee no growth for the Hamiltonian in the discrete formulation. We cautiously mention that one may still observe energy growth in practice, even when the LaBudde-Greenspan is paired with one provably energy-decaying integrator. This phenomenon is attributed by the error from the iterative solver for the nonlinear system. Readers may refer to \cite[Sec.~4.2]{Betsch2000} for an analysis on this point.

\subsection{Numerical results}
\label{sec:numerical_results}
In this section, we perform numerical studies of the proposed integrators using a stiff neo-Hookean spring model \cite[Sec.~4.8.2]{Gros2004}. The setup of the problem involves a single particle with mass $m=10$ located at $\bm q_0 = (2,1,1)$ initially, whose initial momentum is $\bm p_0 = (-30, 15, 45)$. The potential energy is given by 
\begin{align*}
\hat{V}(r) = c\frac{\bar{r}^2}{6} \left[ \left( \frac{r}{\bar{r}} \right)^2 + 2 \frac{\bar{r}}{r} -3 \right],
\end{align*}
with the stiffness coefficient $c=10^3$ and the reference radius $\bar{r} = 4$. With the above settings, the values of the Hamiltonian and angular momentum at the initial time are given by $H(\bm q_0, \bm p_0) = 1866.8$ and $\bm J(\bm q_0, \bm p_0) = (30, -120, 60)$. In terms of numerical methods, we make the potential energy split in the following manner,
\begin{align*}
\hat{V}_{\mathrm{c}} = \hat{V}, \qquad \hat{V}_{\mathrm e} = 0, \qquad \hat{V}_{+} = \hat{V}, \qquad \hat{V}_{-} = 0. 
\end{align*}
In the predictor multi-corrector algorithm, the stopping criteria are given by $\mathrm{tol}_{\mathrm R} = 10^{-10}$, $\mathrm{tol}_{\mathrm A} = 10^{-15}$, and $l_{\mathrm{max}} = 20$.

\begin{figure}
	\begin{center}
	\begin{tabular}{ccc}
\includegraphics[angle=0, trim=120 100 120 110, clip=true, scale = 0.32]{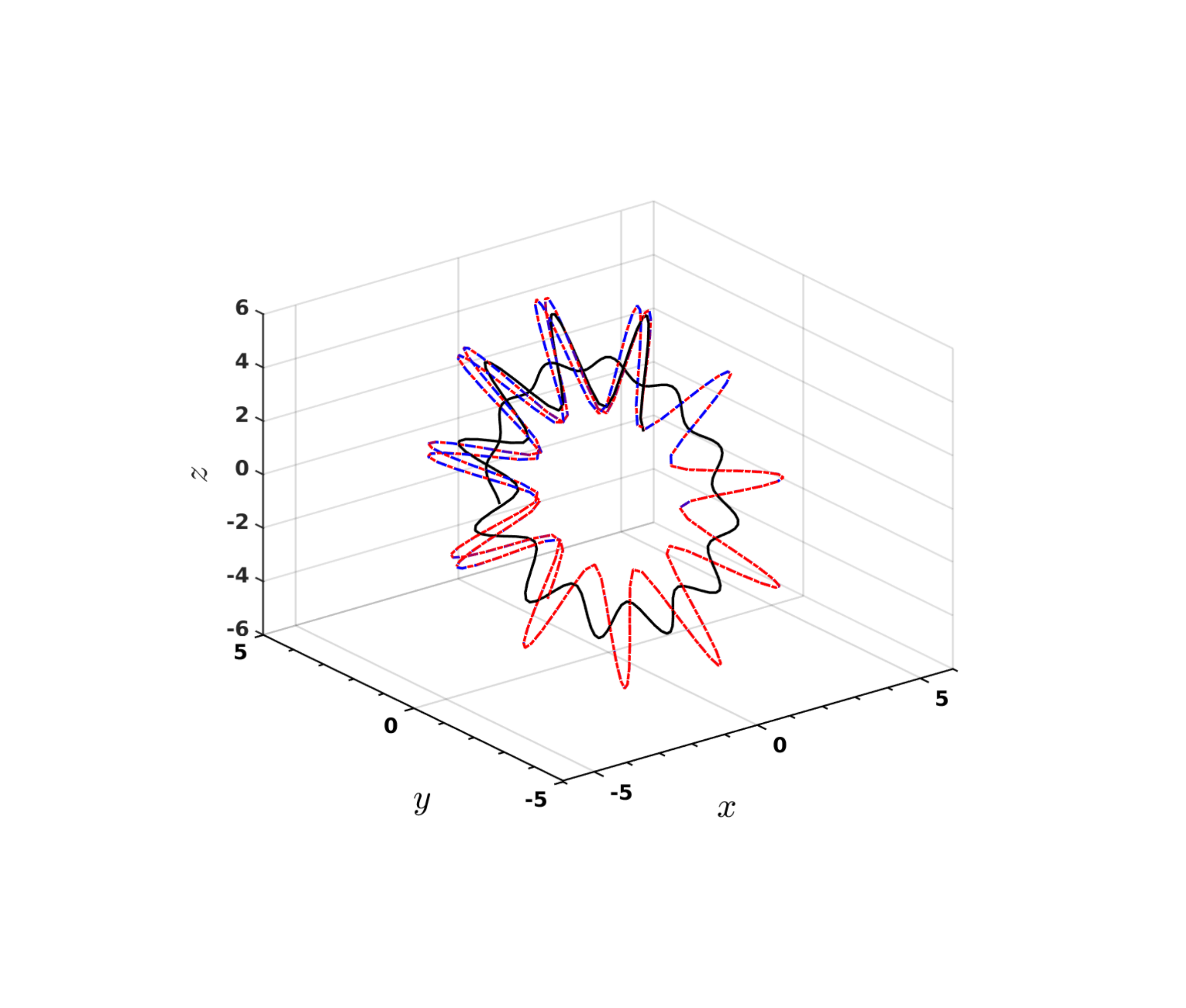} &
\includegraphics[angle=0, trim=120 100 120 110, clip=true, scale = 0.32]{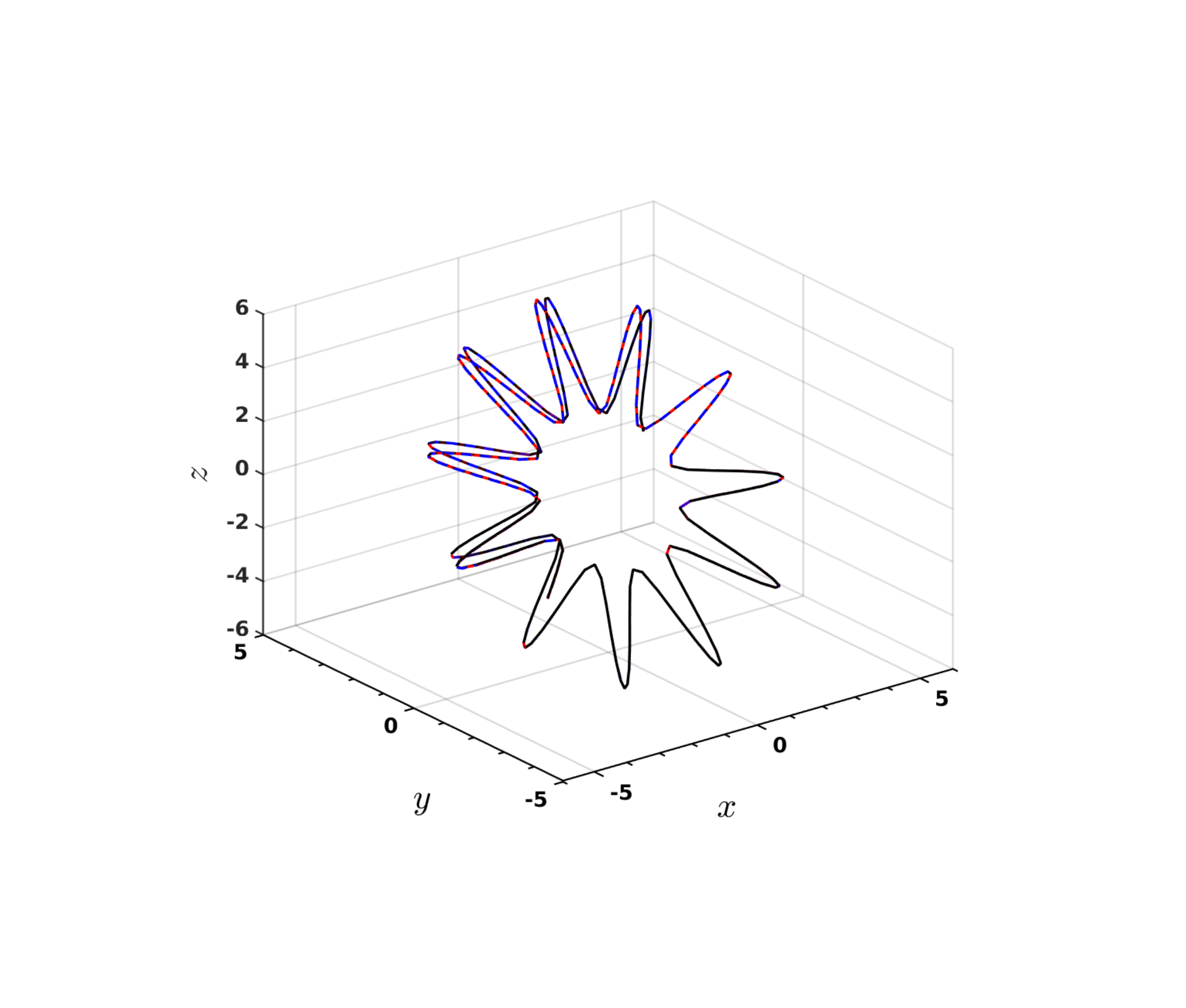} &
\includegraphics[angle=0, trim=120 100 120 110, clip=true, scale = 0.32]{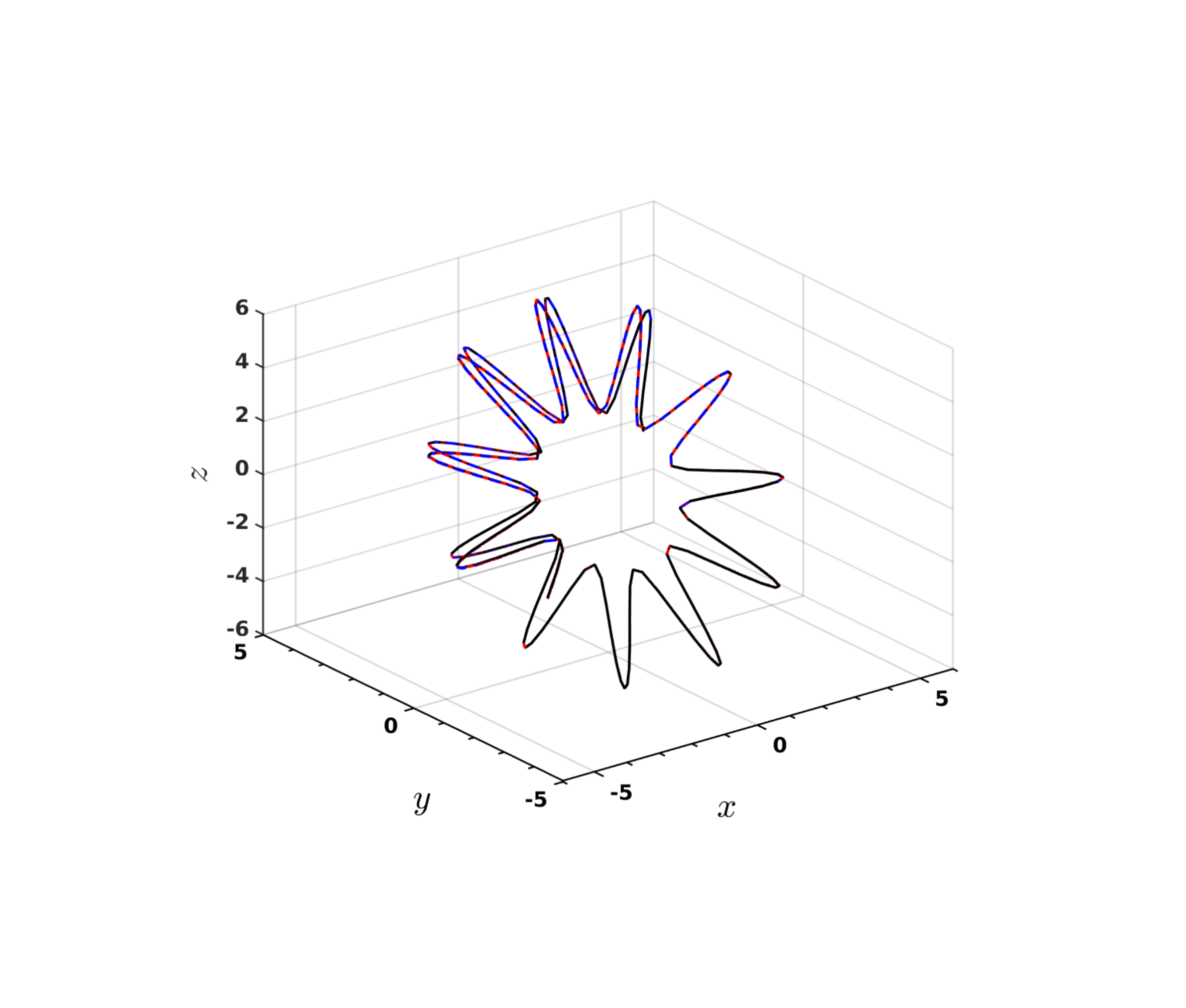} \\
(a) & (b) & (c)
\end{tabular}
\caption{The orbits of the single particle computed by the generalized Eyre integrator (a), the perturbed mid-point integrator (b), and the perturbed trapezoidal integrator (c) are illustrated as black solid lines, respectively. For comparison purposes, the orbits computed by the mid-point (blue dashed lines) and LaBudde-Greenspan integrators (red dash-dot lines) are depicted.} 
\label{fig:single_particle_orbits}
\end{center}
\end{figure}

\paragraph{Accuracy}
We performed computation using the mid-point integrator with $\Delta t_n = 10^{-6}$ to obtain a reference solution at $T=10$ denoted by $\bm q^{\mathrm{ref}}$ and $\bm p^{\mathrm{ref}}$. We then calculated the relative error by comparing numerical solutions with the reference solution and the results are reported in Table \ref{table:accuracy-stiff-new-Hookean-model}. The convergence rates corroborate our error analysis made in Section \ref{sec:integrator-split-perturbation}. Also, from the last two columns of Table \ref{table:accuracy-stiff-new-Hookean-model}, we observe that the errors of the perturbed mid-point and perturbed trapezoidal integrators are no larger than those of the LaBudde-Greenspan or mid-point integrators after more than $\mathcal O(10^5)$ time steps. In Figure \ref{fig:single_particle_orbits}, the orbits calculated with the largest time step size, i.e., $\Delta t_n= 10^{-2}$, are depicted. It can be seen that the trajectories calculated by the mid-point, LaBudde-Greenspan, perturbed mid-point, and perturbed trapezoidal schemes are indistinguishable. The strong dissipative character of the generalized Eyre integrator is observed in Figure \ref{fig:single_particle_orbits} (a), in which the oscillations of the spring are quickly damped.

\begin{table}[htbp]
\footnotesize
\begin{center}
\tabcolsep=0.21cm
\renewcommand{\arraystretch}{1.2}
\begin{tabular}{c | c c c c c}
\hline
\hline
$T / \Delta t$ & $1 \times 10^3$ & $2\times 10^3$ & $1 \times 10^4$ & $2\times 10^4$ & $1 \times 10^5$ \\
\hline 
Mid-point & & & & &  \\
$\|\bm q_{n_{\mathrm{ts}}} - \bm q^{\mathrm{ref}}\|/\|\bm q^{\mathrm{ref}}\|$  & $4.32 \times 10^{-1}$ & $1.08 \times 10^{-2}$ & $4.31 \times 10^{-4}$ & $1.08 \times 10^{-4}$ & $4.31 \times 10^{-6}$  \\
order & - & 2.00 & 2.00 & 2.00 & 2.00 \\
$\|\bm p_{n_{\mathrm{ts}}} - \bm p^{\mathrm{ref}}\|/\|\bm p^{\mathrm{ref}}\|$ & $2.47 \times 10^{-2}$ & $6.74 \times 10^{-3}$ & $2.77 \times 10^{-4}$ & $6.92 \times 10^{-5}$ & $2.77 \times 10^{-6}$ \\
order & - & 1.88 & 1.98 & 2.00 & 2.00 \\
\hline 
LaBudde-Greenspan & & & & &  \\
$\|\bm q_{n_{\mathrm{ts}}} - \bm q^{\mathrm{ref}}\|/\|\bm q^{\mathrm{ref}}\|$  & $4.30 \times 10^{-1}$ & $1.07 \times 10^{-2}$ & $4.29 \times 10^{-4}$ & $1.07 \times 10^{-4}$ & $4.29 \times 10^{-6}$  \\
order & - & 2.00 & 2.00 & 2.00 & 2.00 \\
$\|\bm p_{n_{\mathrm{ts}}} - \bm p^{\mathrm{ref}}\|/\|\bm p^{\mathrm{ref}}\|$ & $2.45 \times 10^{-2}$ & $6.71 \times 10^{-3}$ & $2.76 \times 10^{-4}$ & $6.90 \times 10^{-5}$ & $2.76 \times 10^{-6}$ \\
order & - & 1.87 & 1.98 & 2.00 & 2.00 \\
\hline 
Generalized Eyre & & & & &\\
$\|\bm q_{n_{\mathrm{ts}}} - \bm q^{\mathrm{ref}}\|/\|\bm q^{\mathrm{ref}}\|$  & $7.18 \times 10^{-1}$ & $5.88 \times 10^{-1}$ &  $2.52 \times 10^{-1}$ & $1.45 \times 10^{-1}$ & $3.27 \times 10^{-2}$\\
order & - & 0.29 & 0.53 & 0.80 & 0.93  \\
$\|\bm p_{n_{\mathrm{ts}}} - \bm p^{\mathrm{ref}}\|/\|\bm p^{\mathrm{ref}}\|$ & $8.51 \times 10^{-1}$ & $7.67 \times 10^{-1}$ & $2.39 \times 10^{-1}$ & $1.19 \times 10^{-1}$ & $2.36 \times 10^{-2}$ \\
order & - & 0.15 & 0.72 & 1.00 & 1.00 \\
\hline 
Perturbed mid-point & & & & &\\
$\|\bm q_{n_{\mathrm{ts}}} - \bm q^{\mathrm{ref}}\|/\|\bm q^{\mathrm{ref}}\|$  & $4.43 \times 10^{-2}$ & $1.09 \times 10^{-2}$ & $4.30 \times 10^{-4}$ & $1.07 \times 10^{-4}$ & $4.29 \times 10^{-6}$  \\
order & - & 2.02 & 2.01 & 2.00 & 2.00 \\
$\|\bm p_{n_{\mathrm{ts}}} - \bm p^{\mathrm{ref}}\|/\|\bm p^{\mathrm{ref}}\|$ & $2.26 \times 10^{-2}$ & $6.50 \times 10^{-3}$ & $2.74 \times 10^{-4}$ & $6.89 \times 10^{-5}$ & $2.76 \times 10^{-6}$ \\
order & - & 1.80 & 1.97 & 1.99 & 2.00 \\
\hline 
Perturbed trapezoid & & & & &\\
$\|\bm q_{n_{\mathrm{ts}}} - \bm q^{\mathrm{ref}}\|/\|\bm q^{\mathrm{ref}}\|$  & $4.56 \times 10^{-2}$ & $1.10 \times 10^{-2}$ & $4.32 \times 10^{-4}$ & $1.07 \times 10^{-4}$ & $4.29 \times 10^{-6}$  \\
order & - & 2.04 & 2.01 & 2.00 & 2.00 \\
$\|\bm p_{n_{\mathrm{ts}}} - \bm p^{\mathrm{ref}}\|/\|\bm p^{\mathrm{ref}}\|$ & $2.07 \times 10^{-2}$ & $6.30 \times 10^{-3}$ & $2.73 \times 10^{-4}$ & $6.87 \times 10^{-5}$ & $2.76 \times 10^{-6}$ \\
order & - & 1.72 & 1.95 & 1.99 & 2.00 \\
\hline 
\hline
\end{tabular}
\end{center}
\caption{Relative errors and convergence rates of the mid-point, LaBudde-Greenspan, generalized Eyre, perturbed mid-point, and perturbed trapezoidal integrators using the stiff new-Hookean spring model.}
\label{table:accuracy-stiff-new-Hookean-model}
\end{table}

\paragraph{The conserving and decaying properties}
We investigate the conservation and dissipation properties of the integrators using a fixed time step size $\Delta t_n = 10^{-3}$. The differences between the calculated Hamiltonian $H(\bm q_{n}, \bm p_{n})$ and its value at the initial time are plotted in Figure \ref{fig:single_particle_energy}. It can be observed that the LaBudde-Greenspan integrator preserves the energy up to $\mathcal O(10^{-10})$, which corroborates the analysis made previously \cite{Betsch2000}. In the meantime, its energy error also shows a periodical oscillatory pattern, following the pattern of the mid-point integrator. The similarity of the two indicates that the default alternate option in the LaBudde-Greenspan integrator indeed has a tendency to destroy the conservation property. The two dissipative second-order integrators clearly exhibit dissipation effects on the Hamiltonian, and the amount of dissipation is less than the energy error of the mid-point integrator. The first-order generalized Eyre integrator shows the strongest damping effect, with around $40\%$ energy dissipated over the simulated time period. In terms of the angular momentum, we have the three components depicted in Figure \ref{fig:single_particle_momentum}, and all integrators conserve the angular momentum with the relative error less than $\mathcal O(10^{-11})$, which is proportional to the nonlinear solver tolerance.

\begin{figure}
\begin{center}
\begin{tabular}{cc}
\includegraphics[angle=0, trim=80 80 125 100, clip=true, scale = 0.3]{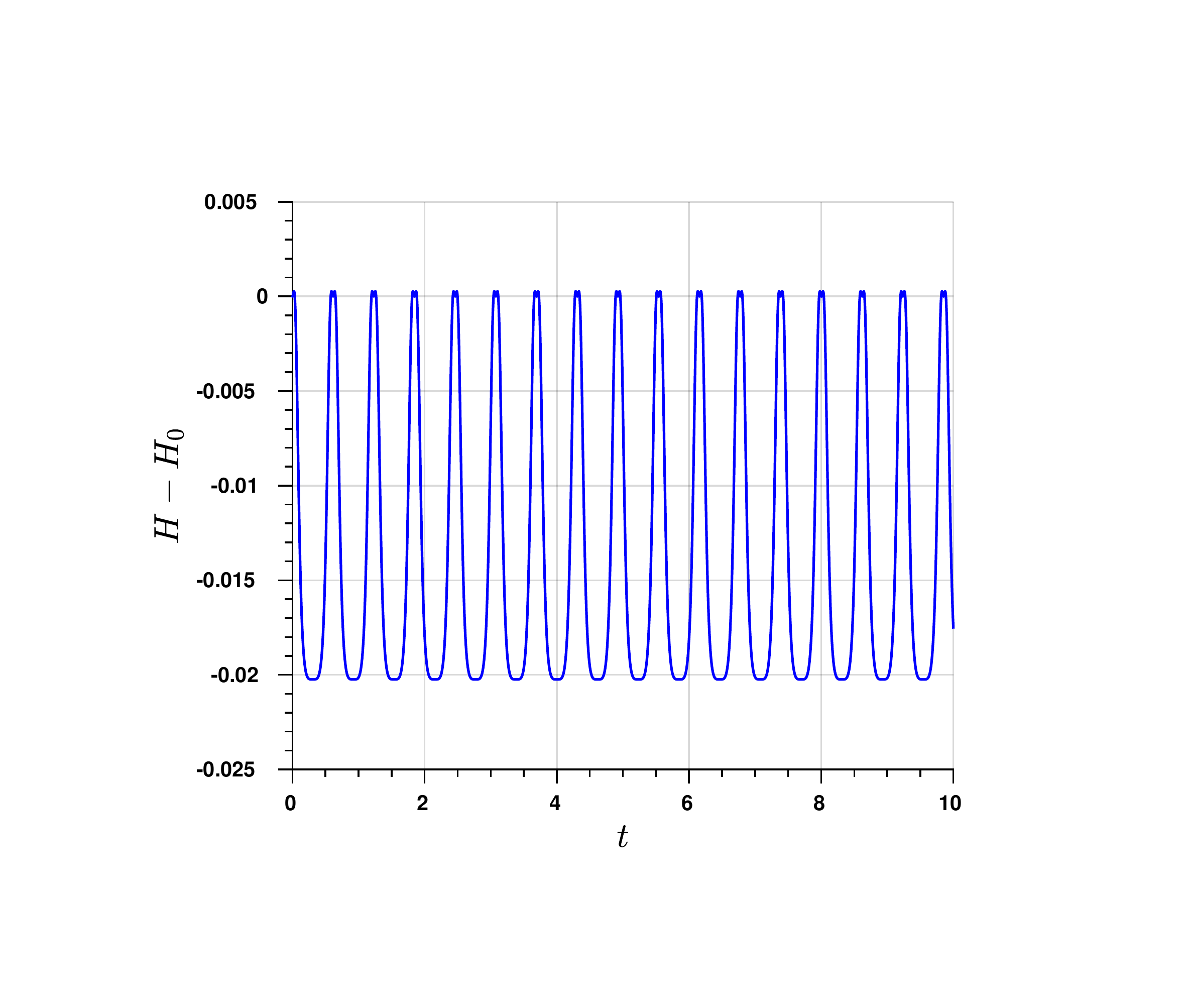} & \includegraphics[angle=0, trim=80 80 125 100, clip=true, scale = 0.3]{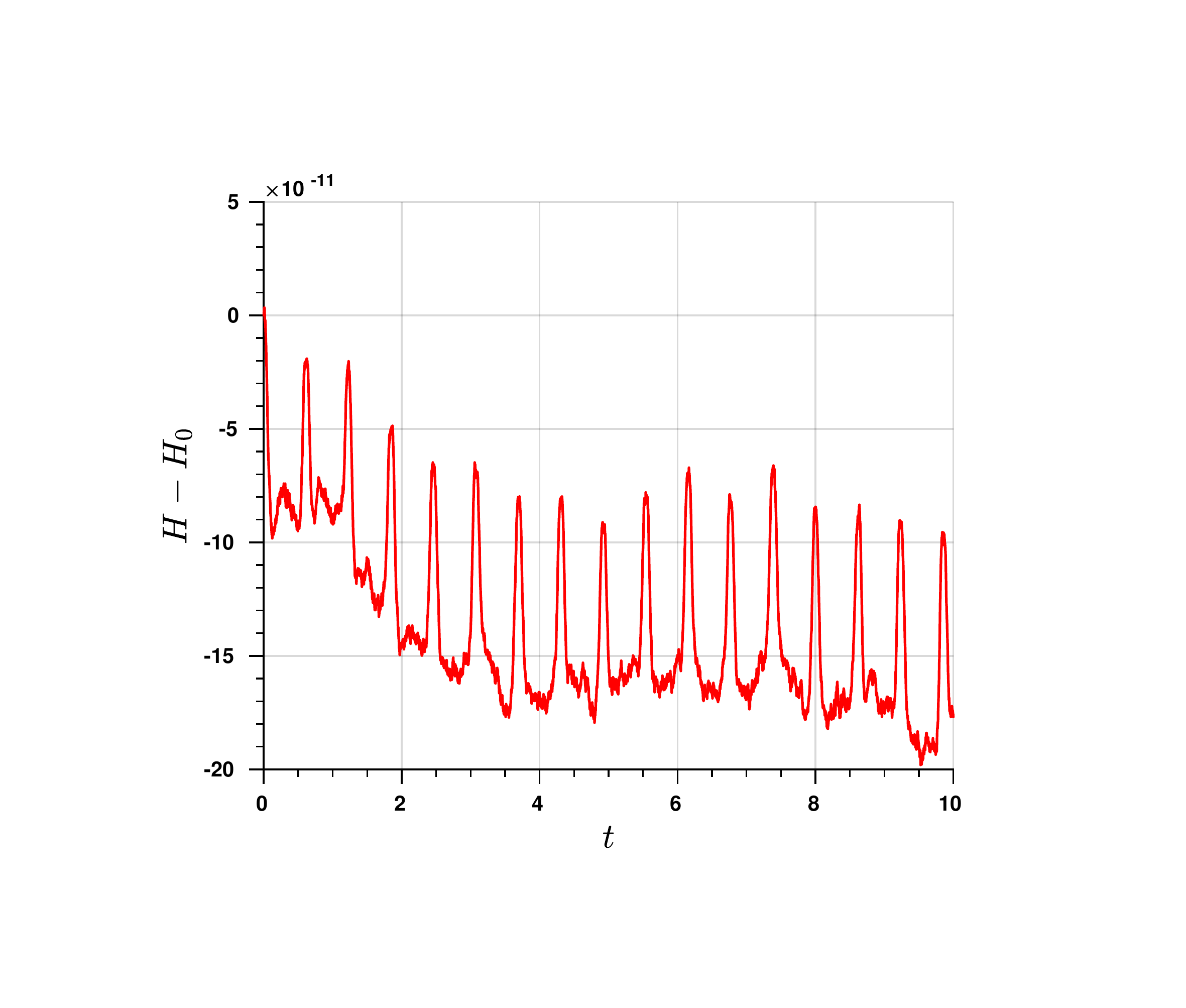} \\
(a) & (b)
\end{tabular}
\begin{tabular}{ccc}
\includegraphics[angle=0, trim=80 80 125 100, clip=true, scale = 0.3]{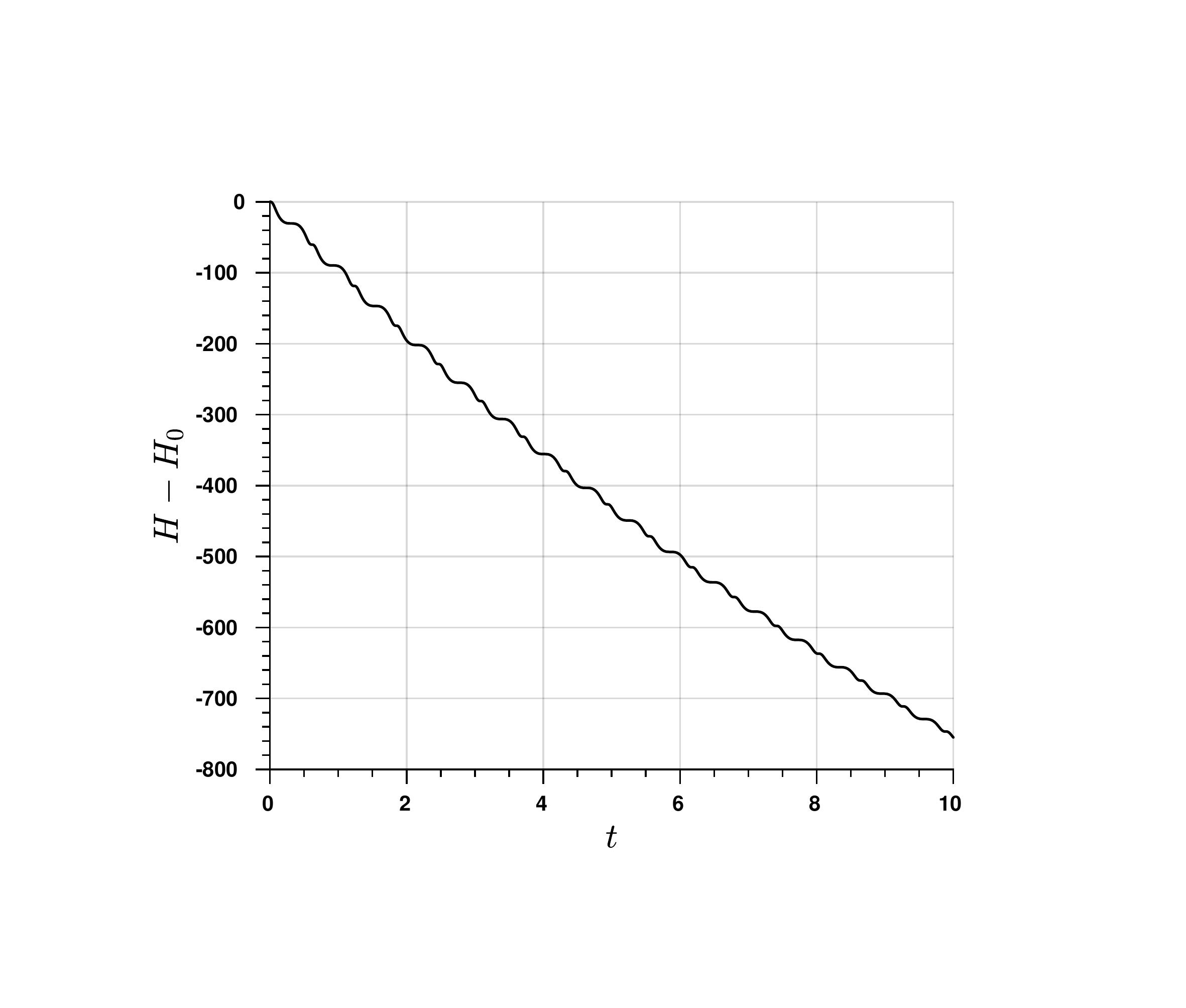} &
\includegraphics[angle=0, trim=80 80 125 100, clip=true, scale = 0.3]{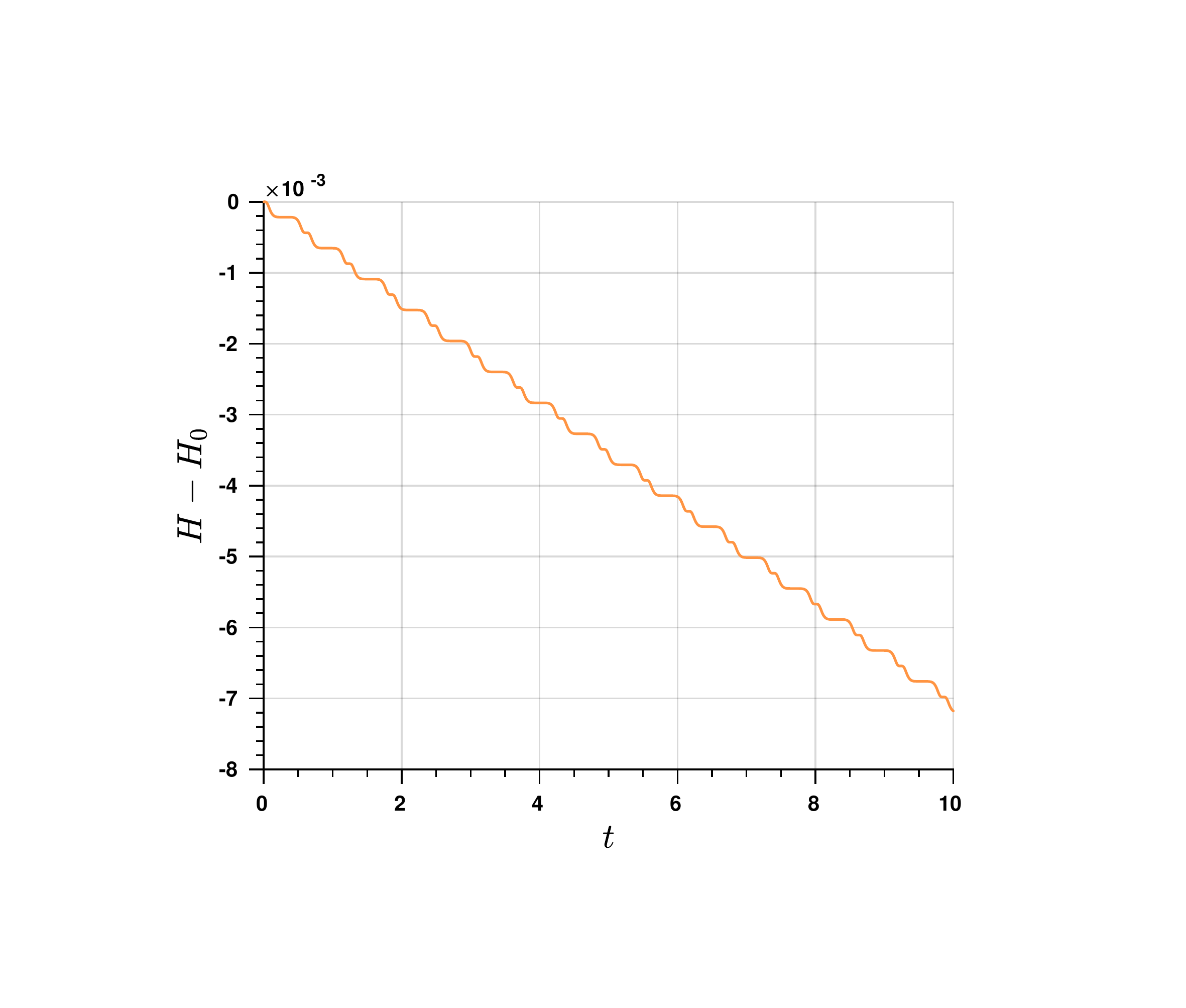} & \includegraphics[angle=0, trim=80 80 125 100, clip=true, scale = 0.3]{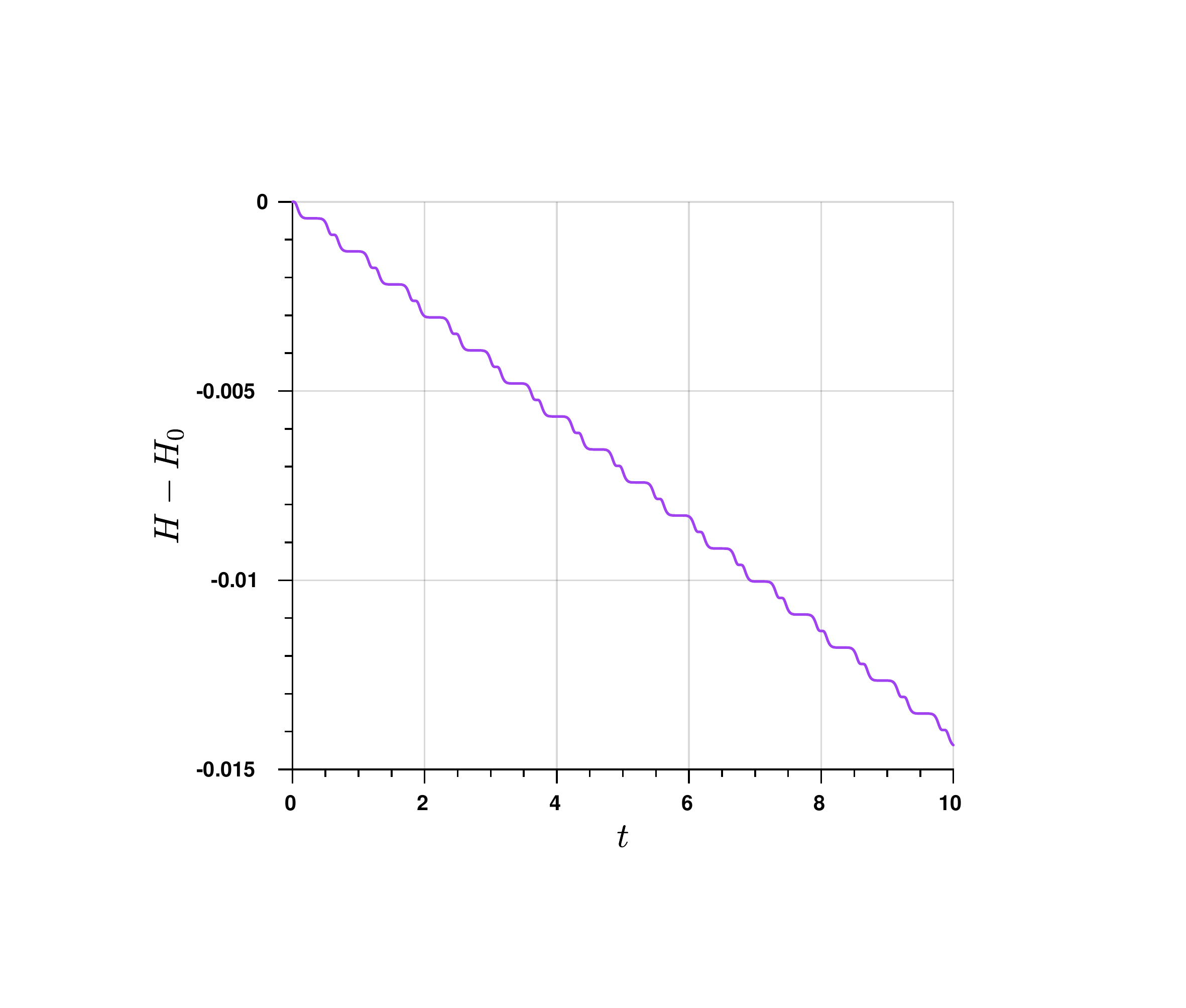} \\
(c) & (d) & (e)
\end{tabular}
\caption{The error of the total energy $H(\bm q_n, \bm p_n) - H(\bm q_0, \bm p_0)$ over time calculated by (a) the mid-point, (b) LaBudde-Greenspan, (c) generalized Eyre, (d) perturbed mid-point, and (e) perturbed trapezoidal integrators.} 
\label{fig:single_particle_energy}
\end{center}
\end{figure}

\begin{figure}
\begin{center}
\begin{tabular}{ccc}
\multicolumn{3}{c}{ \includegraphics[angle=0, trim=780 265 690 1730, clip=true, scale = 0.45]{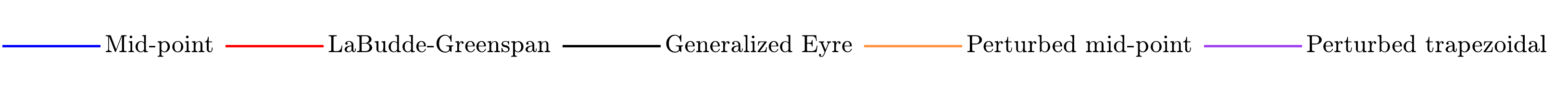} }\\
\includegraphics[angle=0, trim=80 80 125 110, clip=true, scale = 0.3]{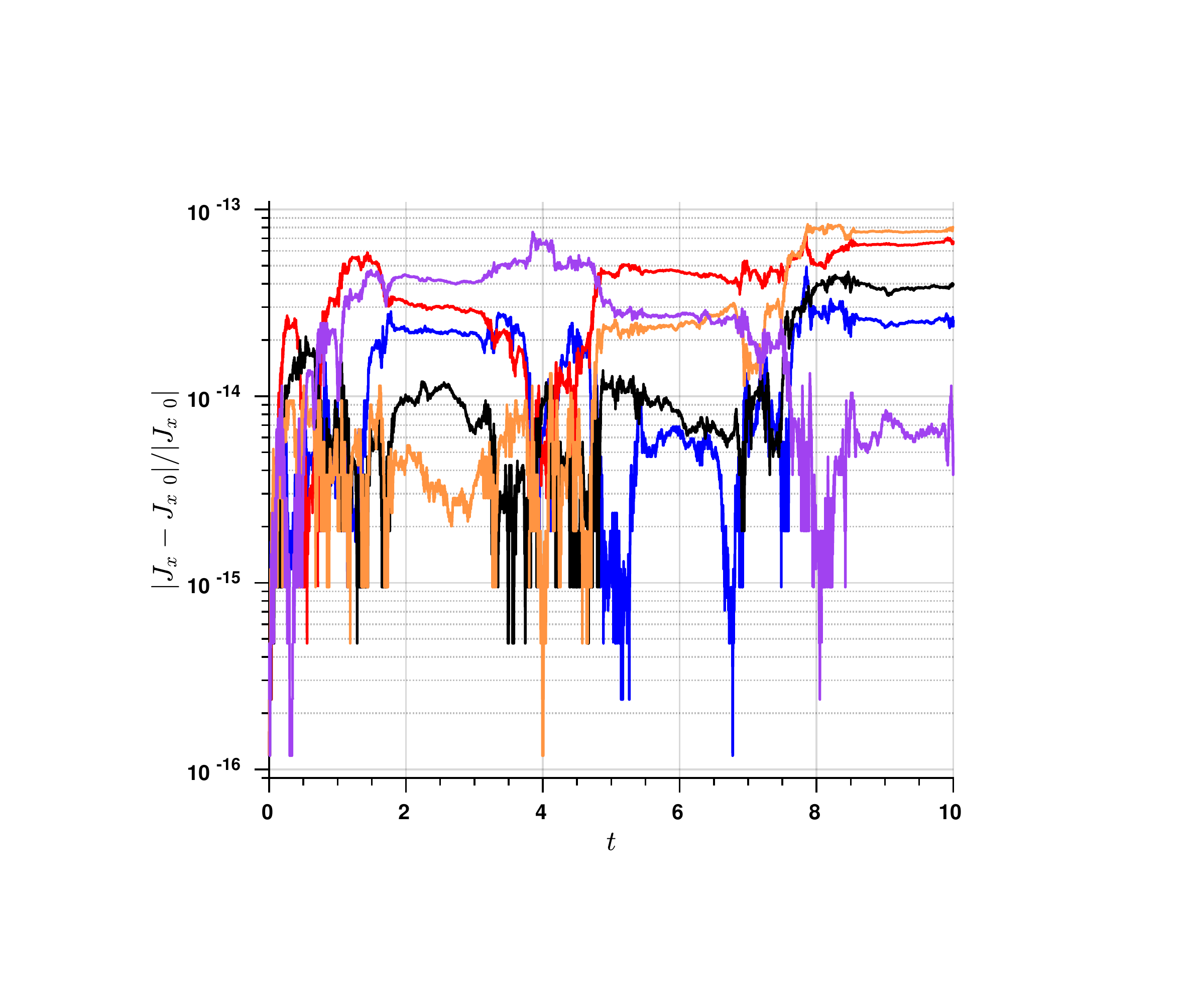} &
\includegraphics[angle=0, trim=80 80 125 110, clip=true, scale = 0.3]{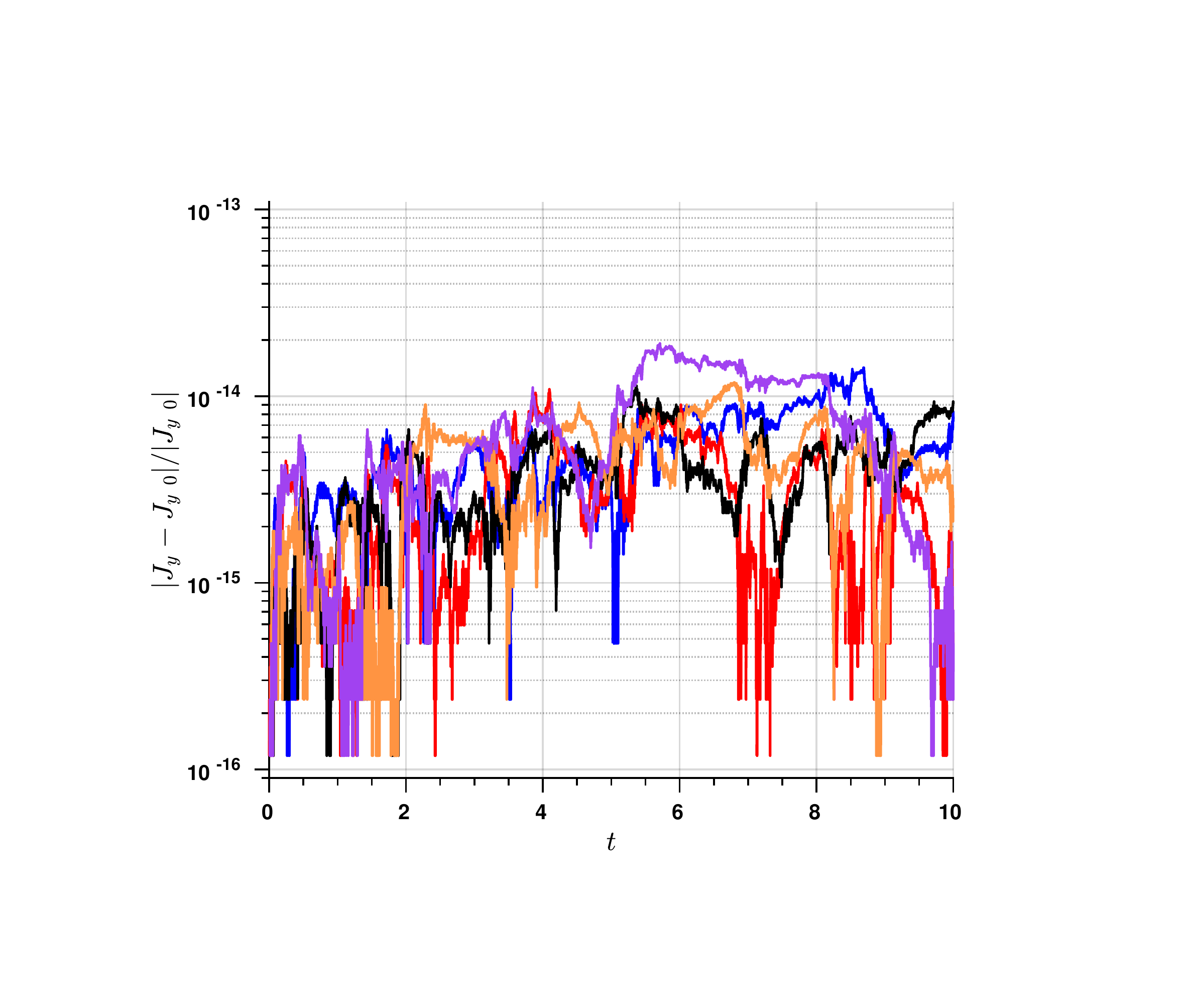} &
\includegraphics[angle=0, trim=80 80 125 110, clip=true, scale = 0.3]{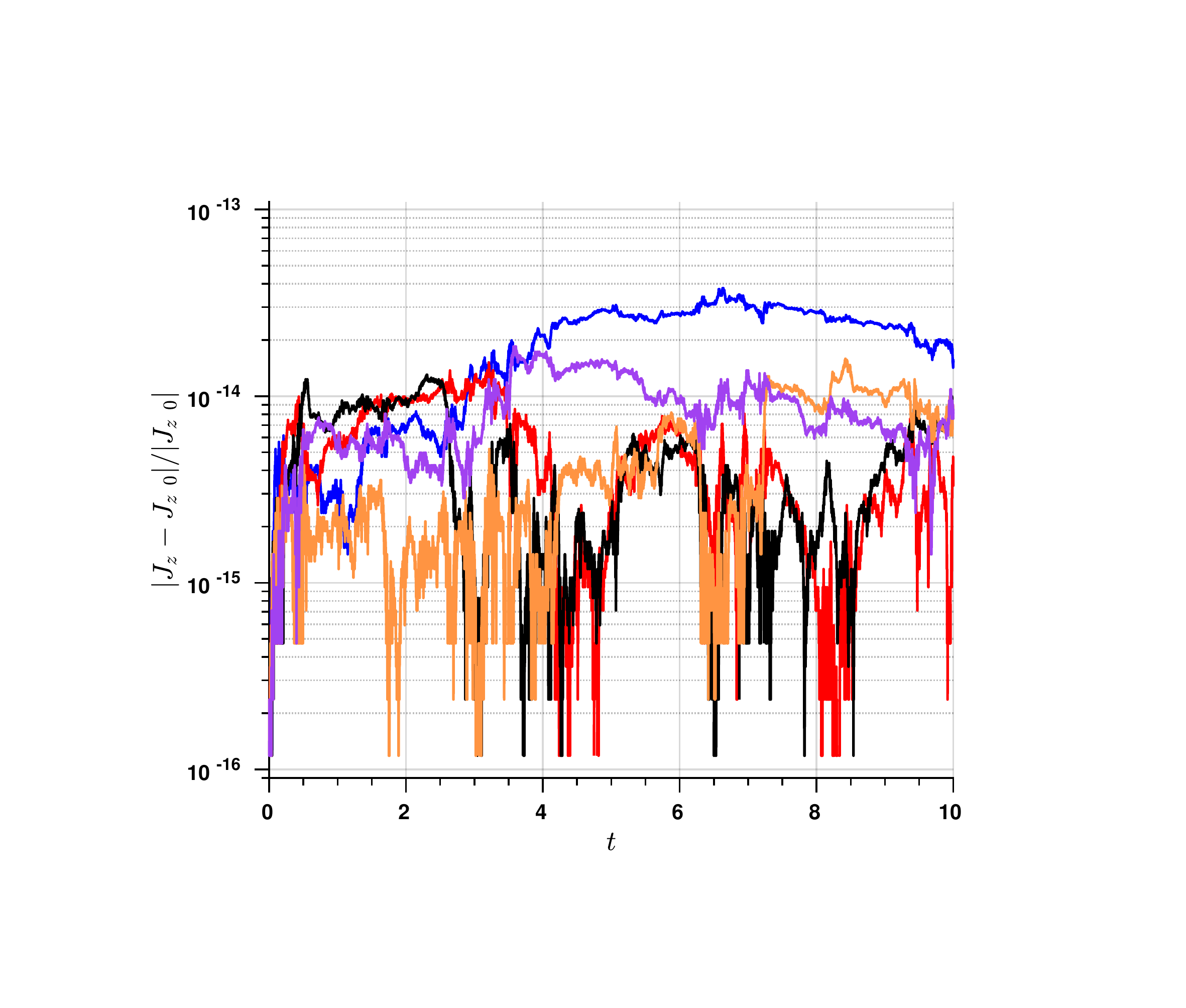} 
\end{tabular}
\caption{The angular momentum relative errors of the mid-point, LaBudde-Greenspan, generalized Eyre, perturbed mid-point, and perturbed trapezoidal integrators.} 
\label{fig:single_particle_momentum}
\end{center}
\end{figure}

\begin{figure}
	\begin{center}
	\begin{tabular}{cc}
\includegraphics[angle=0, trim=80 80 120 110, clip=true, scale = 0.45]{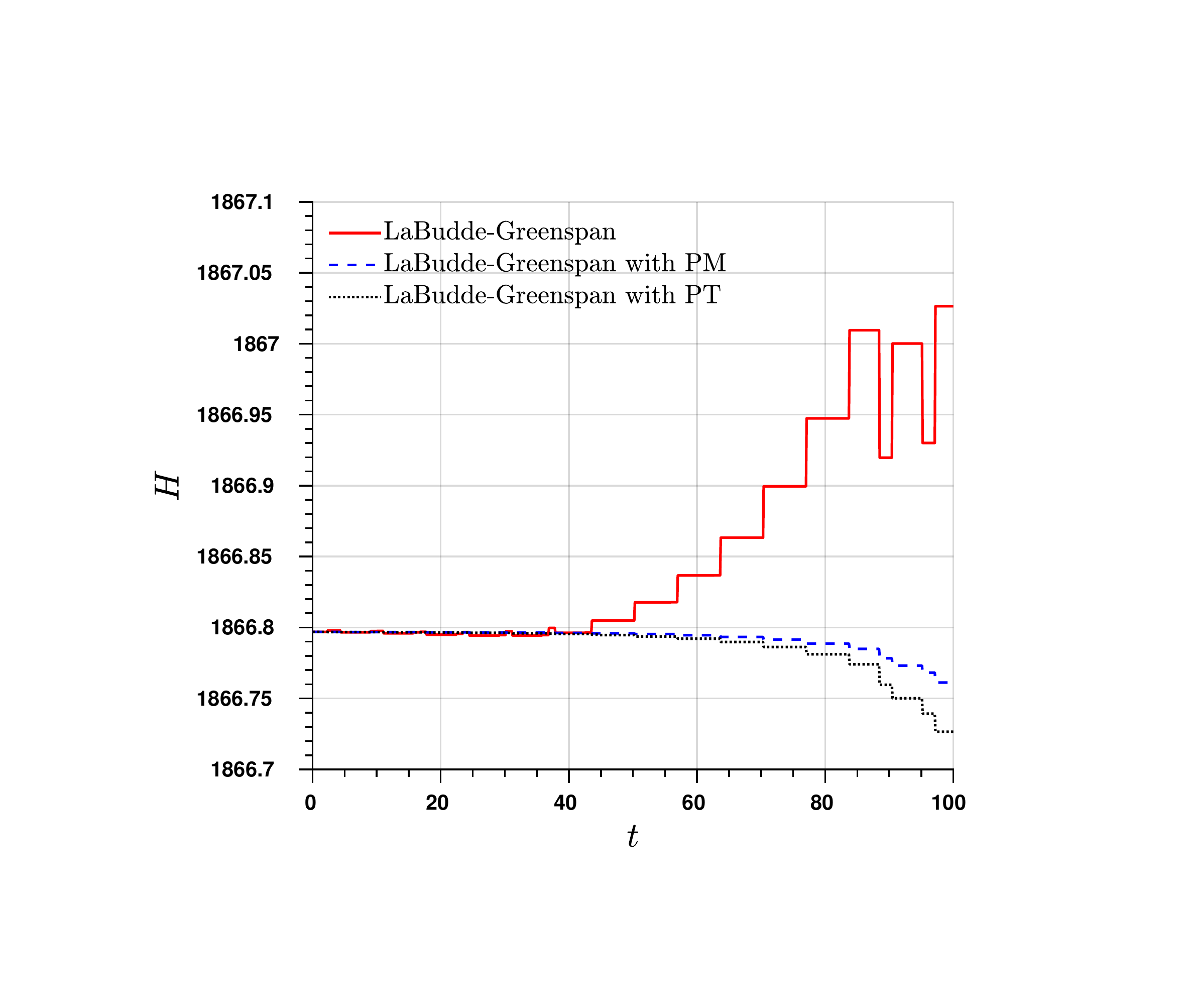} &\includegraphics[angle=0, trim=80 80 120 110, clip=true, scale = 0.45]{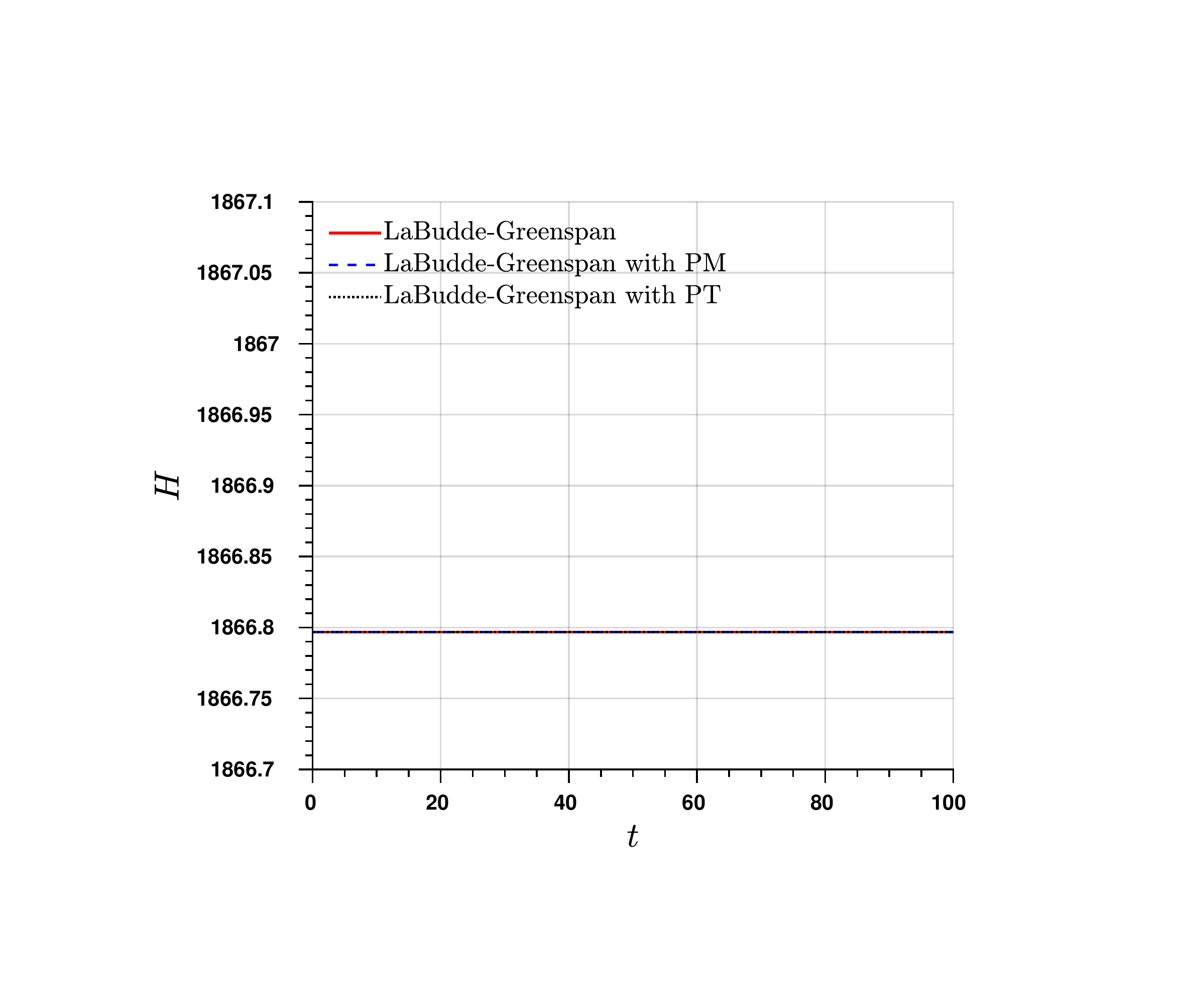} \\
$\mathrm{tol}_{\mathrm Q} = 10^{-1}$ & $\mathrm{tol}_{\mathrm Q} = 10^{-8}$
\end{tabular}
\caption{The Hamiltionian of the particle calculated by the the LaBudde-Greenspan with different alternate options. The red solid line is the LaBudde-Greenspan integrator, which invokes \eqref{eq:labudde-greenspan-default-switch} for small denominator. The blue dashed and black dotted lines represent the results calculated by the LaBudde-Greenspan paired with the perturbed mid-point (PM) and perturbed trapezoidal (PT) schemes, respectively.} 
\label{fig:hybird_algorithm_energy}
\end{center}
\end{figure}

\paragraph{The robustness of the LaBudde-Greenspan integrator} As was analyzed previously, the quotient formula in the LaBudde-Greenspan integrator is by default switched to \eqref{eq:labudde-greenspan-default-switch} when the criterion $|\|\bm q_{n+1,(l)}\| - \|\bm q_{n}\|| \leq \mathrm{tol}_{\mathrm Q}$ is met. However, this option is likely to induce the growth of the energy. In this study, the perturbed mid-point and perturbed trapezoidal integrators are studied as the alternate option for $\Lambda^{\star}$. We performed numerical investigation with $\Delta t_n=10^{-1}$ and $T=10^{2}$, and the results are illustrated in Figure \ref{fig:hybird_algorithm_energy}. It can be seen that there indeed exists an uncontrolled energy growth in the LaBudde-Greenspan integrator when the tolerance $\mathrm{tol}_{\mathrm Q}$ is relatively large (i.e., $10^{-1}$). Invoking the energy-decaying integrators ensures the boundedness of the energy and thus enjoys better numerical stability. When the tolerance $\mathrm{tol}_{\mathrm Q}$ is small (i.e., $10^{-8}$), the three integrators behave similarly in terms of energy preservation.

\section{Many-body dynamics}
\label{sec:many-body-dynamics}
In this section, we generalize our proposed algorithm to many-body dynamics. The linear momentum and center of mass are two additional invariants of many-body dynamics, which will be respected as well in the proposed integrators.

\subsection{The model problem and integrators}
\label{sec:many-body-model-integrators}
In many-body dynamics, the kinetic energy $K(\bm p)$ adopts the form given in \eqref{eq:Hamiltonian_K}, while the potential energy takes the following form,
\begin{align*}
V(\bm q) = \frac12 \sum_{A=1}^{N} \sum_{\substack{B=1, \\ B\neq A}}^{N} \hat{V}_{AB}(d_{AB}), \quad d_{AB} := \|\bm q_A - \bm q_B\|,
\end{align*}
wherein $\hat{V}_{AB}: \mathbb R_{+} \rightarrow \mathbb R$ is a scalar function. Following the notations introduced in the previous section, we use $\hat{V}_{AB~\mathrm c}$, $\hat{V}_{AB~\mathrm e}$, $\hat{V}_{AB~+}$ and $\hat{V}_{AB~-}$ to represent the convex, concave, super-convex, and super-concave parts of $\hat{V}_{AB}$, respectively. With the above form of the potential energy, the force $\bm F^A$ can be written as
\begin{align*}
\bm F^A := \frac{\partial V}{\partial \bm q_A} = \sum^{N}_{\substack{B=1\\B \neq A}} \bm f^{AB}, \quad \bm f^{AB} := \hat{V}'_{AB}(d_{AB}) \frac{\bm q_A - \bm q_B}{d_{AB}}.
\end{align*}
Different from the single-particle problem, the linear momentum $\bm L(\bm q, \bm p)$ is conserved in many-body dynamics due to the translational invariance in the Hamiltonian $H$. The center of mass for the multiple bodies can be defined as
\begin{align*}
\bm C(\bm q, \bm p, t) := \left( \sum_{A=1}^{N} \sum_{B=1}^{N} \bm m^{AB} \bm q_{B} - t \bm L \right) / \left( \sum_{A=1}^{N} \sum_{B=1}^{N} \bm m^{AB} \right),
\end{align*}
and it can be conveniently shown that $\bm C$ is also an invariant for many-body dynamics \cite{Wan2022}.

Following the approach outlined in Sections \ref{sec:summary_classcial_integrators} and \ref{sec:integrator-split-perturbation}, the discrete formulation for the many-body problem can be stated as follows. In each time subinterval $\mathbb I_n$, given $\bm q_{A~n}$ and $\bm p^{A~n}$, find $\bm q_{A~n+1}$ and $\bm p^{A~n+1}$, for $A=1,\cdots, N$, such that
\begin{align}
\label{eq:many-body-space-time-q}
\bm q_{A~n+1} - \bm q_{A~n} =& \Delta t_n \sum_{B=1}^{N} \bm m^{-1}_{AB}  \bm p^{B}_{n+\frac12}, \\
\label{eq:many-body-space-time-p}
\bm p^A_{n+1} - \bm p^A_{n} =& - \sum^{N}_{\substack{B=1\\B \neq A}} \bm f^{AB}_{\square},
\end{align}
and $\bm f^{AB}_{\square}$ is the discrete approximation to the conservative force $\bm f^{AB}$ over the subinterval $\mathbb I_n$. Using linear finite element in time, the discrete force is given by
\begin{align}
\label{eq:many-body-time-fe-time-integral}
\bm f^{AB}_{\square} = \bm f^{AB}_{\mathrm{fe}} := \int_{t_n}^{t_{n+1}}  \hat{V}'(d_{AB}(t)) \frac{\bm q_A(t) - \bm q_B(t)}{d_{AB}(t)} dt,
\end{align}
wherein the quantities are linearly interpolated as
\begin{align*}
\bm q_{A}(t) = \left( \bm q_{A~n+1} - \bm q_{A~n} \right) \frac{t-t_n}{\Delta t_n} + \bm q_{A~n}, \quad \bm p^{A}(t) = \left( \bm p^{A}_{n+1} - \bm p^{A}_n \right) \frac{t-t_n}{\Delta t_n} + \bm p^{A}_{n}, \quad d_{AB}(t) = \|\bm q_{A}(t) - \bm q_{B}(t)\|.
\end{align*}
In the time finite element formulation, one needs to address the integral appearing in \eqref{eq:many-body-time-fe-time-integral} to complete its definition, and the Gaussian quadrature can be adopted \cite{Betsch2000}. If the one-point Gaussian quadrature is used, the time finite element formulation reduces to the mid-point scheme, whose formulation can be obtained by representing $\bm f^{AB}_{\square}$ as
\begin{align*}
\bm f^{AB}_{\square} = \bm f^{AB}_{\mathrm{mp}} := \Delta t_n \hat{V}'_{AB}(d_{AB~n+\frac12}) \frac{\bm q_{A~n+\frac12} - \bm q_{B~n+\frac12}}{d_{AB~m}},
\end{align*}
wherein $\bm q_{A~n+\frac12} := (\bm q_{A~n} + \bm q_{A~n+1})/2$ and $d_{AB~m}:= \|\bm q_{A~n+\frac12} - \bm q_{B~n+\frac12}\|$. The LaBudde-Greenspan scheme can be defined by
\begin{align*}
\bm f^{AB}_{\square} = \bm f^{AB}_{\mathrm{lg}} := 2\Delta t_n \frac{\hat{V}_{AB}(d_{AB~n+1}) - \hat{V}_{AB}(d_{AB~n})}{d_{AB~n+1} - d_{AB~n}} \frac{\bm q_{A~n+\frac12} - \bm q_{B~n+\frac12}}{d_{AB~n+1} + d_{AB~n}}.
\end{align*}
The three energy-decaying integrators can be stated as follows. The generalized Eyre integrator is given by
\begin{align*}
\bm f^{AB}_{\square} = \bm f^{AB}_{\mathrm{ge}} := 2\Delta t_n \left( \hat{V}'_{AB~\mathrm c}(d_{AB~n+1}) + \hat{V}'_{AB~\mathrm e}(d_{AB~n}) \right) \frac{\bm q_{A~n+\frac12} - \bm q_{B~n+\frac12}}{d_{AB~n+1} + d_{AB~n}};
\end{align*}
the perturbed mid-point integrator is given by
\begin{align*}
\bm f^{AB}_{\square} = \bm f^{AB}_{\mathrm{pm}} :=& 2 \Delta t_n  \Bigg( \hat{V}'_{AB}\left(\frac{d_{AB~n+1} + d_{AB~n}}{2}\right) \displaybreak[2] \nonumber \\
& + \frac{\left( d_{AB~n+1} - d_{AB~n} \right)^2}{24} \left( \hat{V}'''_{AB~+}(d_{AB~n+1}) + \hat{V}'''_{AB~-}(d_{AB~n}) \right) \Bigg) \frac{\bm q_{A~n+\frac12} - \bm q_{B~n+\frac12}}{d_{AB~n+1} + d_{AB~n}};
\end{align*}
the perturbed trapezoidal integrator can be obtained by choosing the algorithmic force as
\begin{align*}
\bm f^{AB}_{\square} = \bm f^{AB}_{\mathrm{pt}} := 2 \Delta t_n & \Bigg( \frac12 \left( \hat{V}'_{AB}(d_{AB~n+1}) + \hat{V}'_{AB}(d_{AB~n}) \right) - \frac{\left( d_{AB~n+1} - d_{AB~n} \right)^2}{12} \left( \hat{V}'''_{AB~+}(d_{AB~n}) \right. \displaybreak[2] \nonumber \\
& \left. + \hat{V}'''_{AB~-}(d_{AB~n+1}) \right) \Bigg) \frac{\bm q_{A~n+\frac12} - \bm q_{B~n+\frac12}}{d_{AB~n+1} + d_{AB~n}}.
\end{align*}
With the above five different definitions of the algorithmic force $\bm f^{AB}_{\square}$, the integrators for many-body dynamics are presented in a unified manner in \eqref{eq:many-body-space-time-q}-\eqref{eq:many-body-space-time-p}.

\subsection{The structure-preserving properties}
Here, we investigate the behavior of the five integrators in terms of the structure-preserving properties. In particular, it will be shown that they preserve the linear momentum, angular momentum, and center of mass. The total energy will be dissipated in the three proposed integrators. First, to demonstrate the angular momentum conservation, the following lemma is useful.
\begin{lemma}
\label{proposition:many-body-momentum}
If the algorithmic force can be represented as $\bm f^{AB}_{\square} = \beta_{AB}\left( \bm q_{A~n+\frac12} - \bm q_{B~n+\frac12} \right)$, in which $\beta_{AB} = \beta_{BA}$, the integrator \eqref{eq:many-body-space-time-q}-\eqref{eq:many-body-space-time-p} conserves the angular momentum.
\end{lemma}
\begin{proof}
The identity \eqref{eq:momentum_identity} can be generalized to the following form,
\begin{align*}
\bm J(\bm q_{n+1}, \bm p_{n+1}) - \bm J(\bm q_n, \bm p_n) &= \sum_{A=1}^{N}\left( \bm q_{A~n+1} \times \bm p^{A}_{n+1} \right) - \sum_{A=1}^{N}\left( \bm q_{A~n} \times \bm p^{A}_{n}  \right) \nonumber \\
&= \sum_{A=1}^{N} \left( \bm q_{A~n+\frac12} \times \left( \bm p^A_{n+1} - \bm p^{A}_n \right) - \bm p^{A}_{n+\frac12} \times \left( \bm q_{A~n+1} - \bm q_{A~n} \right) \right).
\end{align*}
Inserting \eqref{eq:many-body-space-time-q} and \eqref{eq:many-body-space-time-p} into the above identity results in
\begin{align*}
\bm J(\bm q_{n+1}, \bm p_{n+1}) - \bm J(\bm q_n, \bm p_n) = - \sum_{A=1}^{N} \sum^{N}_{\substack{B=1\\B \neq A}} \bm q_{A~n+\frac12} \times \bm f^{AB}_{\square} - \sum_{A=1}^{N} \sum_{B=1}^{N} \bm p^{A}_{n+\frac12} \times \bm m^{-1}_{AB} \bm p^{B}_{n+\frac12}.
\end{align*}
Using the fact that $\bm f^{AB}_{\square} = \beta_{AB} \left( \bm q_{A~n+\frac12} - \bm q_{B~n+\frac12} \right)$, one has the following,
\begin{align*}
& \sum_{A=1}^{N} \sum^{N}_{\substack{B=1\\B \neq A}} \bm q_{A~n+\frac12} \times \bm f^{AB}_{\square} = \sum_{A=1}^{N} \sum^{N}_{\substack{B=1\\B \neq A}} \bm q_{A~n+\frac12} \times \beta_{AB} \left( \bm q_{A~n+\frac12} - \bm q_{B~n+\frac12} \right) \displaybreak[2] \nonumber \\
&= -\sum_{A=1}^{N} \sum^{N}_{\substack{B=1\\B \neq A}} \beta_{AB} \left( \bm q_{A~n+\frac12} \times \bm q_{B~n+\frac12}\right) = -\sum_{A=1}^{N} \sum^{N}_{B=1} \beta_{AB} \left( \bm q_{A~n+\frac12} \times \bm q_{B~n+\frac12}\right) = 0.
\end{align*}
Using a similar argument, it can be shown that
\begin{align*}
\sum_{A=1}^{N} \sum_{B=1}^{N} \bm p^{A}_{n+\frac12} \times \bm m^{-1}_{AB} \bm p^{B}_{n+\frac12} = 0.
\end{align*}
Therefore, one has $\bm J(\bm q_{n+1}, \bm p_{n+1}) = \bm J(\bm q_n, \bm p_n)$, which completes the proof.
\end{proof}

With the aid of Lemma \ref{proposition:many-body-momentum}, we may conveniently have the following results.

\begin{proposition}
The mid-point, LaBudde-Greenspan, generalized Eyre, perturbed mid-point, and perturbed trapezoidal integrators are all momentum conserving.
\end{proposition}
\begin{proof}
It suffices to show that all of the algorithmic forces can be written in the form of 
\begin{align*}
\bm f^{AB}_{\square} = \beta_{AB}\left( \bm q_{A~n+\frac12} - \bm q_{B~n+\frac12} \right)
\end{align*}
with $\beta_{AB} = \beta_{BA}$. This can be verified by inspecting the definitions of $\bm f^{AB}_{\mathrm{mp}}$, $\bm f^{AB}_{\mathrm{lg}}$, $\bm f^{AB}_{\mathrm{ge}}$, $\bm f^{AB}_{\mathrm{pm}}$, and $\bm f^{AB}_{\mathrm{pt}}$ stated in Section \ref{sec:many-body-model-integrators}.
\end{proof}

Considering the total energy $H(\bm q, \bm p) = K(\bm p) + V(\bm q)$ of the model problem, the behavior of the integrators can be summarized in the following proposition.
\begin{proposition}
Let $H(\bm q_{n}, \bm p_{n})$ be the discrete Hamiltonian at time step $t_n$. For the LaBudde-Greenspan integrator, the discrete energy is conserved, i.e., $H(\bm q_{n+1}, \bm p_{n+1}) = H(\bm q_{n}, \bm p_{n})$;
for the generalized Eyre integrator, the discrete energy satisifes $H(\bm q_{n+1}, \bm p_{n+1}) = H(\bm q_{n}, \bm p_{n}) - \mathcal D_{\mathrm{ge}}$, with the dissipation $\mathcal D_{\mathrm{ge}}$ being
\begin{align*}
\mathcal D_{\mathrm{ge}} = \frac{1}{4} \sum_{A=1}^{N} \sum_{\substack{B=1\\B \neq A}}^{N} \left( d_{AB~n+1} - d_{AB~n} \right)^2 \left( \hat{V}''_{AB~\mathrm c}(d_{AB~n+\xi_1}) - \hat{V}''_{AB~\mathrm e}(d_{AB~n+\xi_2}) \right) \geq 0,
\end{align*}
for $\xi_1, \xi_2 \in (0,1)$; for the perturbed mid-point integrator, the discrete energy satisfies $H(\bm q_{n+1}, \bm p_{n+1}) = H(\bm q_{n}, \bm p_{n}) - \mathcal D_{\mathrm{pm}}$, with the dissipation $\mathcal D_{\mathrm{pm}}$ being
\begin{align*}
\mathcal D_{\mathrm{pm}} = \frac{1}{96} \sum_{A=1}^{N} \sum_{\substack{B=1\\B \neq A}}^{N} \left( d_{AB~n+1} - d_{AB~n} \right)^4 \left( \hat{V}^{(4)}_{AB~+}(d_{AB~n+\xi_3}) - \hat{V}^{(4)}_{AB~-}(d_{AB~n+\xi_4}) \right) \geq 0,
\end{align*}
for $\xi_3, \xi_4 \in (0,1)$; for the perturbed trapezoidal integrator, the discrete energy satisfies $H(\bm q_{n+1}, \bm p_{n+1}) = H(\bm q_{n}, \bm p_{n}) - \mathcal D_{\mathrm{pt}}$, with the dissipation $\mathcal D_{\mathrm{pt}}$ being
\begin{align*}
\mathcal D_{\mathrm{pt}} = \frac{1}{48} \sum_{A=1}^{N} \sum_{\substack{B=1\\B \neq A}}^{N} \left( d_{AB~n+1} - d_{AB~n} \right)^4 \left( \hat{V}^{(4)}_{AB~+}(d_{AB~n+\xi_5}) - \hat{V}^{(4)}_{AB~-}(d_{AB~n+\xi_6}) \right) \geq 0,
\end{align*}
for $\xi_5, \xi_6 \in (0,1)$. In the above, $d_{AB~n+\alpha}$ is defined as $(1-\alpha) d_{AB~n} + \alpha d_{AB~n+1}$ for a parameter $\alpha \in [0,1]$.
\end{proposition}

\begin{proof}
First, invoking the equations \eqref{eq:many-body-space-time-q}, it can be shown that
\begin{align}
& \sum_{A=1}^{N} \left( \bm q_{A~n+1} - \bm q_{A~n} \right) \cdot \left( \bm p^{A}_{n+1} - \bm p^{A}_{n} \right) = \Delta t_n \sum_{A=1}^{N} \sum_{B=1}^{N} \left( \bm p^{A}_{n+1} - \bm p^{A}_{n} \right) \bm m^{-1}_{AB} \bm p^{B}_{n+\frac12} \displaybreak[2] \nonumber \\
=& \frac{\Delta t_n}{2} \sum_{A=1}^{N} \sum_{B=1}^{N} \left( \bm p^{A}_{n+1} \bm m^{-1}_{AB} \bm p^{B}_{n+1} + \bm p^{A}_{n+1} \bm m^{-1}_{AB} \bm p^{B}_{n} - \bm p^{A}_{n} \bm m^{-1}_{AB} \bm p^{B}_{n+1} - \bm p^{A}_{n} \bm m^{-1}_{AB} \bm p^{B}_{n} \right) \displaybreak[2] \nonumber \\
=& \frac{\Delta t_n}{2} \sum_{A=1}^{N} \sum_{B=1}^{N} \left( \bm p^{A}_{n+1} \bm m^{-1}_{AB} \bm p^{B}_{n+1} - \bm p^{A}_{n} \bm m^{-1}_{AB} \bm p^{B}_{n} \right)  \displaybreak[2] \nonumber \\
\label{eq:proof-many-body-kinetic-energy}
=& \Delta t_n \left( K(\bm p_{n+1}) - K(\bm p_n) \right).
\end{align}
Second, invoking the equations \eqref{eq:many-body-space-time-p}, it can be revealed that
\begin{align}
& \sum_{A=1}^{N} \left( \bm q_{A~n+1} - \bm q_{A~n} \right) \cdot \left( \bm p^{A}_{n+1} - \bm p^{A}_{n} \right) = - \sum_{A=1}^{N} \sum_{\substack{B=1\\B \neq A}}^{N} \left( \bm q^{A}_{n+1} - \bm q^{A}_{n} \right) \cdot \bm f^{AB}_{\square} \displaybreak[2] \nonumber \\
\label{eq:proof-many-body-potential-energy}
& = - \sum_{A=1}^{N} \sum_{B=1}^{A-1} \left( \bm q^{A}_{n+1} - \bm q^{A}_{n} \right) \cdot \bm f^{AB}_{\square} - \sum_{A=1}^{N} \sum_{B=A+1}^{N} \left( \bm q^{A}_{n+1} - \bm q^{A}_{n} \right) \cdot \bm f^{AB}_{\square}
\end{align}
Regarding the last term of \eqref{eq:proof-many-body-potential-energy}, algebraic manipulation leads to
\begin{align*}
\sum_{A=1}^{N} \sum_{B=A+1}^{N} \left( \bm q^{A}_{n+1} - \bm q^{A}_{n} \right) \cdot \bm f^{AB}_{\square} = - \sum_{B=1}^{N} \sum_{A=B+1}^{N} \left( \bm q^{B}_{n+1} - \bm q^{B}_{n} \right) \cdot \bm f^{AB}_{\square} = - \sum_{A=1}^{N} \sum_{B=1}^{A-1} \left( \bm q^{B}_{n+1} - \bm q^{B}_{n} \right) \cdot \bm f^{AB}_{\square}.
\end{align*}
Therefore, the relation \eqref{eq:proof-many-body-potential-energy} can be rewritten as
\begin{align*}
\sum_{A=1}^{N} \left( \bm q_{A~n+1} - \bm q_{A~n} \right) \cdot \left( \bm p^{A}_{n+1} - \bm p^{A}_{n} \right) &= - \sum_{A=1}^{N} \sum_{B=1}^{A-1} \left( \bm q^{A}_{n+1} - \bm q^{A}_{n} - \bm q^B_{n+1} + \bm q^{B}_n \right) \cdot \bm f^{AB}_{\square} \displaybreak[2] \nonumber \\
&= - \frac12 \sum_{A=1}^{N} \sum_{\substack{B=1\\B \neq A}}^{N} \left( \bm q^{A}_{n+1} - \bm q^{A}_{n} - \bm q^B_{n+1} + \bm q^{B}_n \right) \cdot \bm f^{AB}_{\square}.
\end{align*}
Combining the above with \eqref{eq:proof-many-body-kinetic-energy}, the following relation can be obtained,
\begin{align}
\label{eq:proof-many-body-kinetic-potential-energy}
\Delta t_n \left( K(\bm p_{n+1}) - K(\bm p_n) \right) = - \frac12 \sum_{A=1}^{N} \sum_{\substack{B=1\\B \neq A}}^{N} \left( \bm q^{A}_{n+1} - \bm q^{A}_{n} - \bm q^B_{n+1} + \bm q^{B}_n \right) \cdot \bm f^{AB}_{\square},
\end{align}
and we will analyze the right-hand side of \eqref{eq:proof-many-body-kinetic-potential-energy} based on the specific forms of $\bm f^{AB}_{\square}$ from now on.

For the LaBudde-Greenspan integrator, one has
\begin{align*}
- \frac12 \sum_{A=1}^{N} \sum_{\substack{B=1\\B \neq A}}^{N} \left( \bm q^{A}_{n+1} - \bm q^{A}_{n} - \bm q^B_{n+1} + \bm q^{B}_n \right) \cdot \bm f^{AB}_{\mathrm{lg}} = - \frac{\Delta t_n}{2} \sum_{A=1}^{N} \sum_{\substack{B=1\\B \neq A}}^{N} \left( \hat{V}_{AB}(d_{AB~n+1}) - \hat{V}_{AB}(d_{AB~n}) \right).
\end{align*}
Combining the above with \eqref{eq:proof-many-body-kinetic-potential-energy}, we have $H(\bm q_{n+1}, \bm p_{n+1}) = H(\bm q_{n}, \bm p_{n})$ for this integrator. 

For the generalized Eyre integrator, one has the following,
\begin{align*}
& - \frac12 \sum_{A=1}^{N} \sum_{\substack{B=1\\B \neq A}}^{N} \left( \bm q^{A}_{n+1} - \bm q^{A}_{n} - \bm q^B_{n+1} + \bm q^{B}_n \right) \cdot \bm f^{AB}_{\mathrm{ge}} \displaybreak[2] \\
 =& - \frac{\Delta t_n}{2} \sum_{A=1}^{N} \sum_{\substack{B=1\\B \neq A}}^{N} \left( d_{AB~n+1} - d_{AB~n} \right) \left( \hat{V}'_{AB~\mathrm c}(d_{AB~n+1}) + \hat{V}'_{AB~\mathrm e}(d_{AB~n}) \right) \displaybreak[2] \\
 =& - \frac{\Delta t_n}{2} \sum_{A=1}^{N} \sum_{\substack{B=1\\B \neq A}}^{N} \left( \hat{V}_{AB}(d_{AB~n+1}) - \hat{V}_{AB}(d_{AB~n}) \right) \displaybreak[2] \\
& + \frac{\Delta t_n}{4} \sum_{A=1}^{N} \sum_{\substack{B=1\\B \neq A}}^{N} \left( d_{AB~n+1} - d_{AB~n} \right)^2 \left( -\hat{V}''_{AB~\mathrm c}(d_{AB~n+\xi_1}) + \hat{V}''_{AB~\mathrm e}(d_{AB~n+\xi_2}) \right),
\end{align*}
for $\xi_1,\xi_2 \in (0,1)$. The last equality in the above results from the rectangular quadrature rules \eqref{eq:rectangular-quad-rule-1}-\eqref{eq:rectangular-quad-rule-2}. Consequently, combining with \eqref{eq:proof-many-body-kinetic-energy}, one has 
\begin{align*}
H(\bm q_{n}, \bm p_{n}) - H(\bm q_{n+1}, \bm p_{n+1}) &= \mathcal D_{\mathrm{ge}} \nonumber \displaybreak[2] \\
&= \frac{1}{4} \sum_{A=1}^{N} \sum_{\substack{B=1\\B \neq A}}^{N} \left( d_{AB~n+1} - d_{AB~n} \right)^2 \left( \hat{V}''_{AB~\mathrm c}(d_{AB~n+\xi_1}) - \hat{V}''_{AB~\mathrm e}(d_{AB~n+\xi_2}) \right).
\end{align*}
The dissipation $\mathcal D_{\mathrm{ge}}$ is non-negative, thanks to the convex-concave split.

For the perturbed mid-point scheme, after applying the definition of $\bm f^{AB}_{\mathrm{pm}}$, one has,
\begin{align*}
& - \frac12 \sum_{A=1}^{N} \sum_{\substack{B=1\\B \neq A}}^{N} \left( \bm q^{A}_{n+1} - \bm q^{A}_{n} - \bm q^B_{n+1} + \bm q^{B}_n \right) \cdot \bm f^{AB}_{\mathrm{pm}}  \displaybreak[2] \\
=& -\frac{\Delta t_n}{2} \sum_{A=1}^{N} \sum_{\substack{B=1\\B \neq A}}^{N} \Bigg( \left( d_{AB~n+1} - d_{AB~n} \right) \hat{V}'_{AB}\left(\frac{d_{AB~n+1} + d_{AB~n}}{2}\right) \displaybreak[2] \\
& + \frac{\left( d_{AB~n+1} - d_{AB~n} \right)^3}{24} \left( \hat{V}'''_{AB~+}(d_{AB~n+1}) + \hat{V}'''_{AB~-}(d_{AB~n}) \right) \Bigg) \displaybreak[2] \\
=& - \frac{\Delta t_n}{2} \sum_{A=1}^{N} \sum_{\substack{B=1\\B \neq A}}^{N} \left( \hat{V}_{AB}(d_{AB~n+1}) - \hat{V}_{AB}(d_{AB~n}) \right) \displaybreak[2] \\
& + \frac{\Delta t_n}{96} \sum_{A=1}^{N} \sum_{\substack{B=1\\B \neq A}}^{N} \left( d_{AB~n+1} - d_{AB~n} \right)^4 \left( -\hat{V}^{(4)}_{AB~+}(d_{AB~n+\xi_3}) + \hat{V}^{(4)}_{AB~-}(d_{AB~n+\xi_4}) \right),
\end{align*}
for $\xi_3,\xi_4 \in (0,1)$. In the above derivation, the perturbed mid-point quadrature rules \eqref{eq:perturbed-mid-point-a}-\eqref{eq:perturbed-mid-point-b} are applied to $\hat{V}'_{AB~+}$ and $\hat{V}'_{AB~-}$ separately. Together with \eqref{eq:proof-many-body-kinetic-potential-energy}, one has 
\begin{align*}
H(\bm q_{n}, \bm p_{n}) - H(\bm q_{n+1}, \bm p_{n+1}) &= \mathcal D_{\mathrm{pm}} \nonumber \displaybreak[2] \\
&= \frac{1}{96} \sum_{A=1}^{N} \sum_{\substack{B=1\\B \neq A}}^{N} \left( d_{AB~n+1} - d_{AB~n} \right)^4 \left( \hat{V}^{(4)}_{AB~+}(d_{AB~n+\xi_3}) - \hat{V}^{(4)}_{AB~-}(d_{AB~n+\xi_4}) \right).
\end{align*}
It is straightforward to verify that $\mathcal D_{\mathrm{pm}} \geq 0$ with the definitions of the super-convex and super-concave functions.

For the perturbed trapezoidal scheme, one has
\begin{align*}
& - \frac12 \sum_{A=1}^{N} \sum_{\substack{B=1\\B \neq A}}^{N} \left( \bm q^{A}_{n+1} - \bm q^{A}_{n} - \bm q^B_{n+1} + \bm q^{B}_n \right) \cdot \bm f^{AB}_{\mathrm{pt}}  \displaybreak[2] \\
=& -\frac{\Delta t_n}{2} \sum_{A=1}^{N} \sum_{\substack{B=1\\B \neq A}}^{N} \Bigg( \frac{\left( d_{AB~n+1} - d_{AB~n} \right)}{2} \left( \hat{V}'_{AB}(d_{AB~n+1}) + \hat{V}'_{AB}(d_{AB~n}) \right) \displaybreak[2] \\
& - \frac{\left( d_{AB~n+1} - d_{AB~n} \right)^3}{12} \left( \hat{V}'''_{AB~+}(d_{AB~n}) + \hat{V}'''_{AB~-}(d_{AB~n+1}) \right) \Bigg) \displaybreak[2] \\
=& - \frac{\Delta t_n}{2} \sum_{A=1}^{N} \sum_{\substack{B=1\\B \neq A}}^{N} \left( \hat{V}_{AB}(d_{AB~n+1}) - \hat{V}_{AB}(d_{AB~n}) \right) \displaybreak[2] \\
& - \frac{\Delta t_n}{48} \sum_{A=1}^{N} \sum_{\substack{B=1\\B \neq A}}^{N} \left( d_{AB~n+1} - d_{AB~n} \right)^4 \left( \hat{V}^{(4)}_{AB~+}(d_{AB~n+\xi_5}) - \hat{V}^{(4)}_{AB~-}(d_{AB~n+\xi_6}) \right),
\end{align*}
for $\xi_5,\xi_6 \in (0,1)$. Again, with \eqref{eq:proof-many-body-kinetic-energy}, it can be shown that
\begin{align*}
H(\bm q_{n}, \bm p_{n}) - H(\bm q_{n+1}, \bm p_{n+1}) &= \mathcal D_{\mathrm{pt}} \nonumber \displaybreak[2] \\
&= \frac{1}{48} \sum_{A=1}^{N} \sum_{\substack{B=1\\B \neq A}}^{N} \left( d_{AB~n+1} - d_{AB~n} \right)^4 \left( \hat{V}^{(4)}_{AB~+}(d_{AB~n+\xi_5}) - \hat{V}^{(4)}_{AB~-}(d_{AB~n+\xi_6}) \right),
\end{align*}
and $\mathcal D_{\mathrm{pt}} \geq 0$ due to the definitions of $\hat{V}_{AB~+}$ and $\hat{V}_{AB~-}$.
\end{proof}

\begin{remark}
If the potential energy is linear or quadratic, the third derivatives in the definition of $\bm f^{AB}_{\mathrm{pm}}$ and $\bm f^{AB}_{\mathrm{pt}}$ vanish, which leads to
\begin{align*}
\bm f^{AB}_{\mathrm{pm}} = 2 \Delta t_n \hat{V}'_{AB}\left(\frac{d_{AB~n+1} + d_{AB~n}}{2} \right), \quad \mbox{and} \quad \bm f^{AB}_{\mathrm{pt}} = \Delta t_n \left( \hat{V}'_{AB}\left(d_{AB~n+1}\right) + \hat{V}'_{AB}\left(d_{AB~n}\right) \right).
\end{align*}
The resulting integators are energy conserving. The definition of $\bm f^{AB}_{\mathrm{pm}}$ collocates $\hat{V}'_{AB}$ at $d_{AB~n+1/2}$ and is still different from the mid-point integrator. If the potential energy is cubic, the perturbation terms in $\bm f^{AB}_{\mathrm{pm}}$ and $\bm f^{AB}_{\mathrm{pt}}$ are nonzero in both integrators, while the dissipation terms remain to be zero. The dissipation in the perturbed mid-point and perturbed trapezoidal integrators arise when there is a non-vanishing fourth derivative of the potential energy.
\end{remark}

In addition to the angular momentum and total energy, the integrators also conserve the total linear momentum and mass center in the many-body dynamics, which is stated in the following proposition.
\begin{proposition}
The mid-point, LaBudde-Greenspan, generalized Eyre, perturbed mid-point, and perturbed trapezoidal integrators conserve the linear momentum, i.e., $\bm L(\bm q_{n+1}, \bm p_{n+1}) = \bm L(\bm q_n, \bm p_n)$ and the center of mass, i.e., $\bm C(\bm q_{n+1}, \bm p_{n+1}, t_{n+1}) = \bm C(\bm q_{n}, \bm p_{n}, t_{n})$.
\end{proposition}
\begin{proof}
Based on the definition of the linear momentum $\bm L$ given in \eqref{eq:def_angular_linear_momenta}, one has
\begin{align*}
\bm L(\bm q_{n+1}, \bm p_{n+1}) - \bm L(\bm q_n, \bm p_n) = - \sum_{A=1}^{N} \sum^{N}_{\substack{B=1\\B \neq A}} \bm f^{AB}_{\square}.
\end{align*}
Since $\bm f^{AB}_{\square} = \bm f^{BA}_{\square}$ for all integrators, the above relation suggests $\bm L(\bm q_{n+1}, \bm p_{n+1}) - \bm L(\bm q_n, \bm p_n) = \bm 0$. Next, the equations \eqref{eq:many-body-space-time-q} result in
\begin{align*}
\bm 0 =& \sum_{A=1}^{N} \left( \sum_{B=1}^{N} \bm m^{AB} \left( \bm q_{B~n+1} - \bm q_{B~n} \right) - \Delta t_n \bm p^{A}_{n+\frac12} \right) \nonumber \\
=& \sum_{A=1}^{N}\sum_{B=1}^{N} \bm m^{AB} \left( \bm q_{B~n+1} - \bm q_{B~n} \right) - \Delta t_n \bm L(\bm q_{n+1}, \bm p_{n+1}) \nonumber \\
=& \sum_{A=1}^{N}\sum_{B=1}^{N} \bm m^{AB} \left( \bm q_{B~n+1} - \bm q_{B~n} \right) - \left( t_{n+1} \bm L(\bm q_{n+1}, \bm p_{n+1}) - t_n \bm L(\bm q_{n}, \bm p_{n}) \right).
\end{align*}
In the last equality of the above derivation, the conservation of linear momentum is utilized. The above relation implies that the center of mass is conserved, i.e., $\bm C(\bm q_{n+1}, \bm p_{n+1}, t_{n+1}) = \bm C(\bm q_{n}, \bm p_{n}, t_{n})$.
\end{proof}

\begin{figure}
	\begin{center}
	\begin{tabular}{c}
\includegraphics[angle=0, trim=80 90 120 100, clip=true, scale = 0.45]{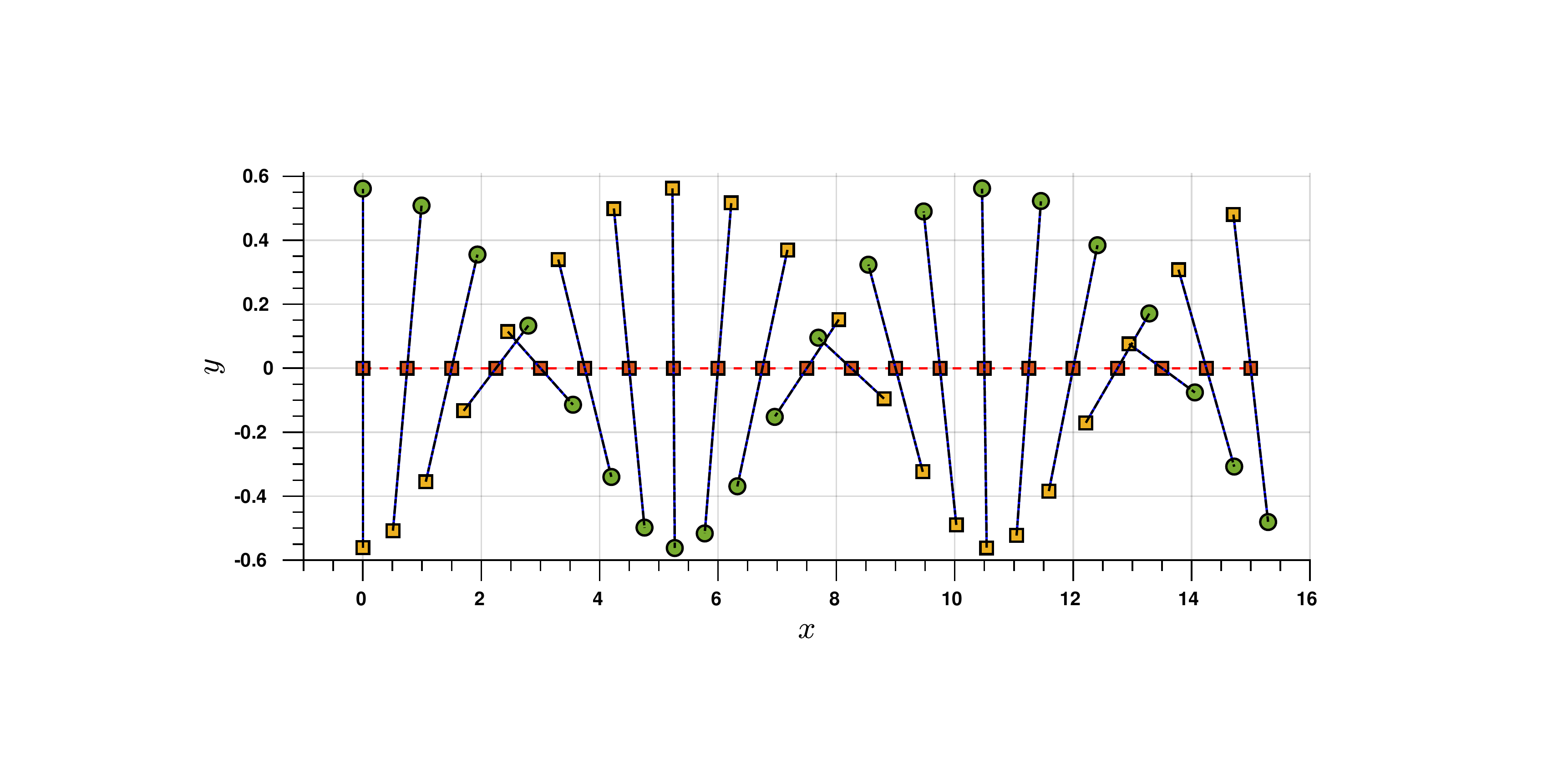}
\end{tabular}
\caption{The trajectories of the two particles with the Lennard-Jones potential over time $0 \leq t \leq 2$. The red square is the location of $(\bm m^{11} \bm q_1 + \bm m^{22} \bm q_2) / (\bm m^{11} + \bm m^{22})$.} 
\label{fig:two_particle_trajectory}
\end{center}
\end{figure}

\begin{figure}
	\begin{center}
	\begin{tabular}{cc}
\includegraphics[angle=0, trim=80 80 120 110, clip=true, scale = 0.42]{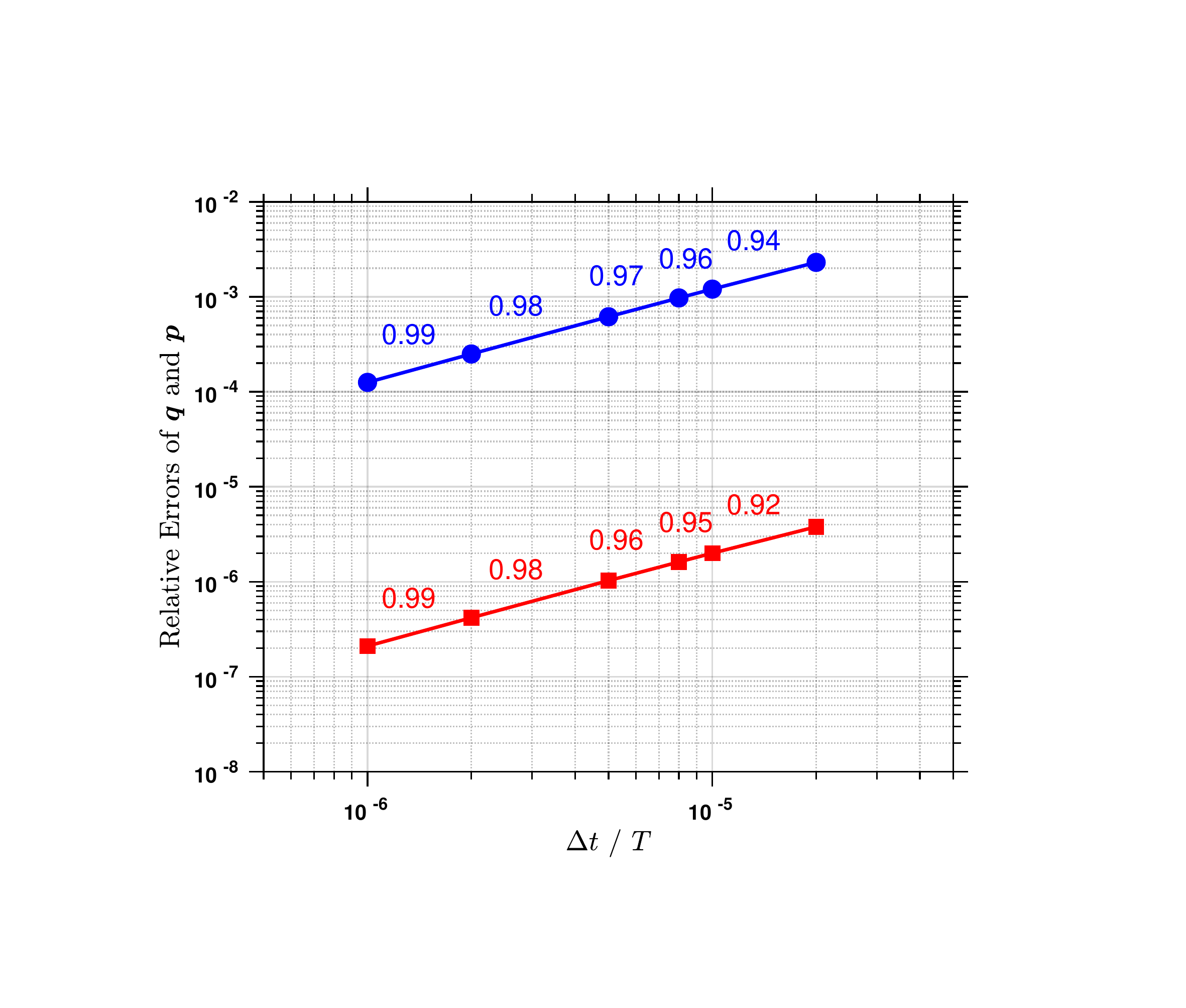} &\includegraphics[angle=0, trim=80 80 120 110, clip=true, scale = 0.42]{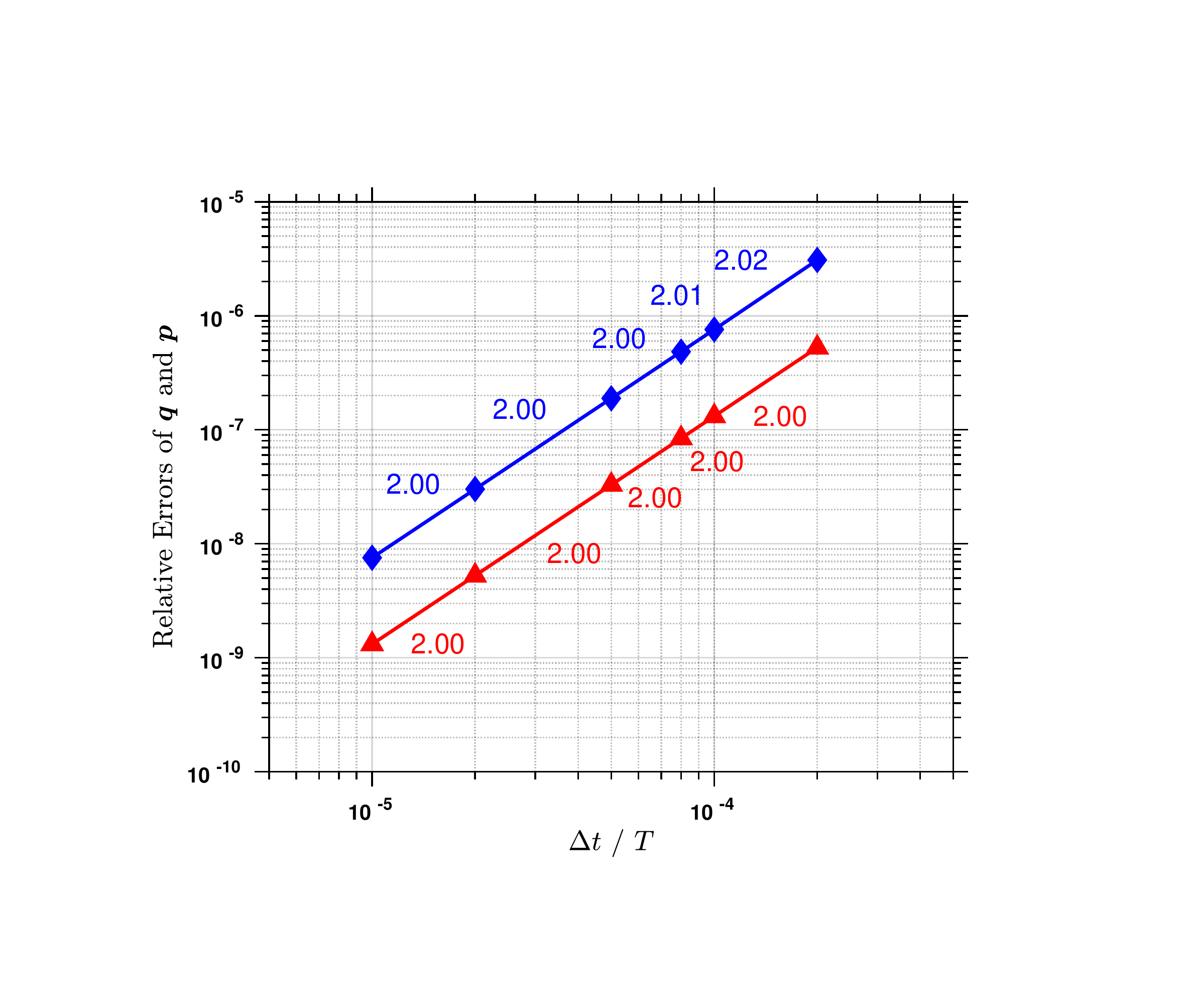} \\
(a) & (b)
\end{tabular}
\caption{The relative errors of (a) the generalized Eyre integrator and (b) the perturbed mid-point and perturbed trapezoidal integrators in the two-particle problem. The convergence rates calculated from adjacent errors are annotated. The relative errors of $\bm q$ and $\bm p$ are depicted in blue and red, respectively.} 
\label{fig:two_particle_conv_rates}
\end{center}
\end{figure}

\begin{figure}
\begin{center}
\begin{tabular}{cc}
\includegraphics[angle=0, trim=80 80 130 100, clip=true, scale = 0.3]{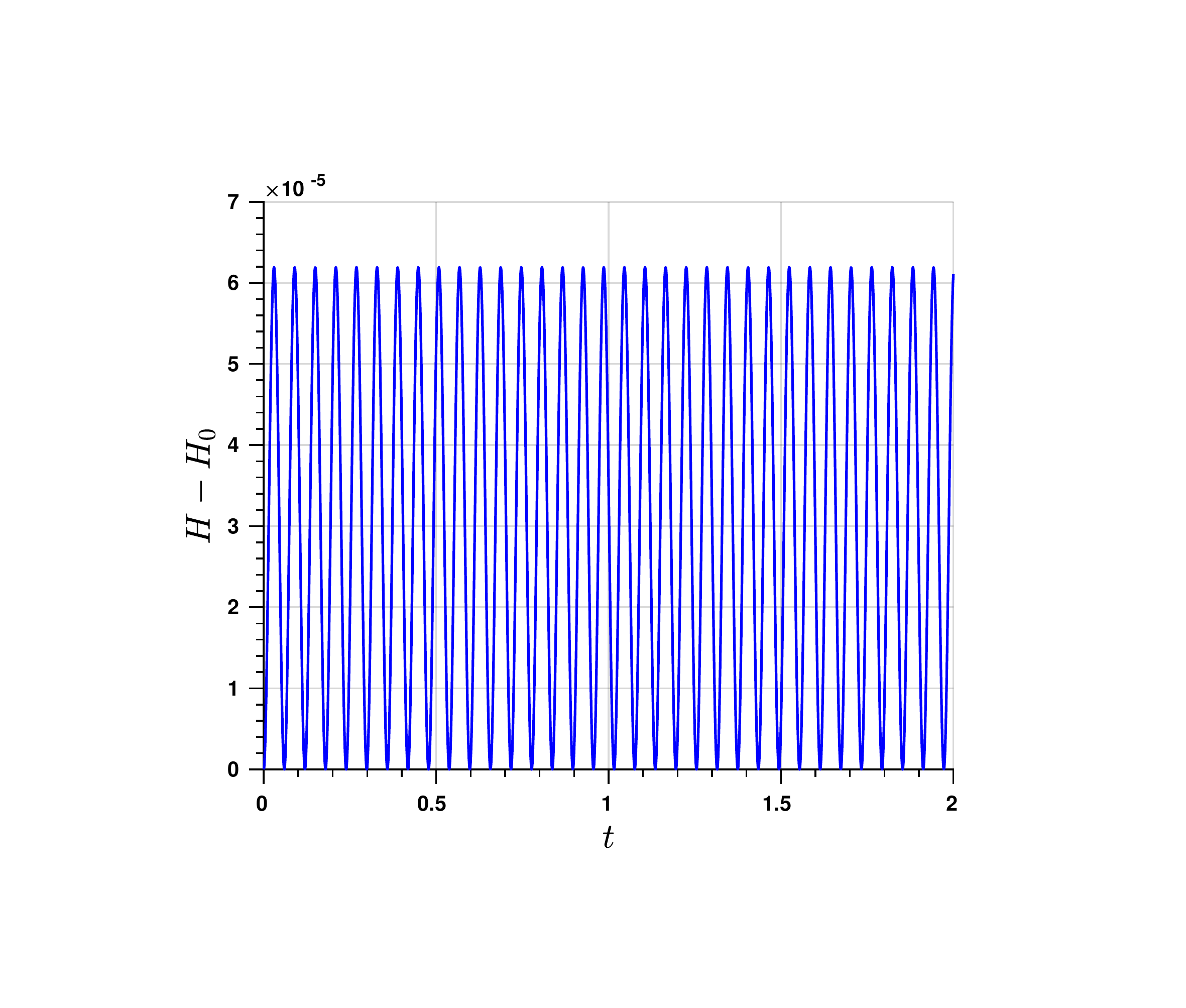} &\includegraphics[angle=0, trim=80 80 130 100, clip=true, scale = 0.3]{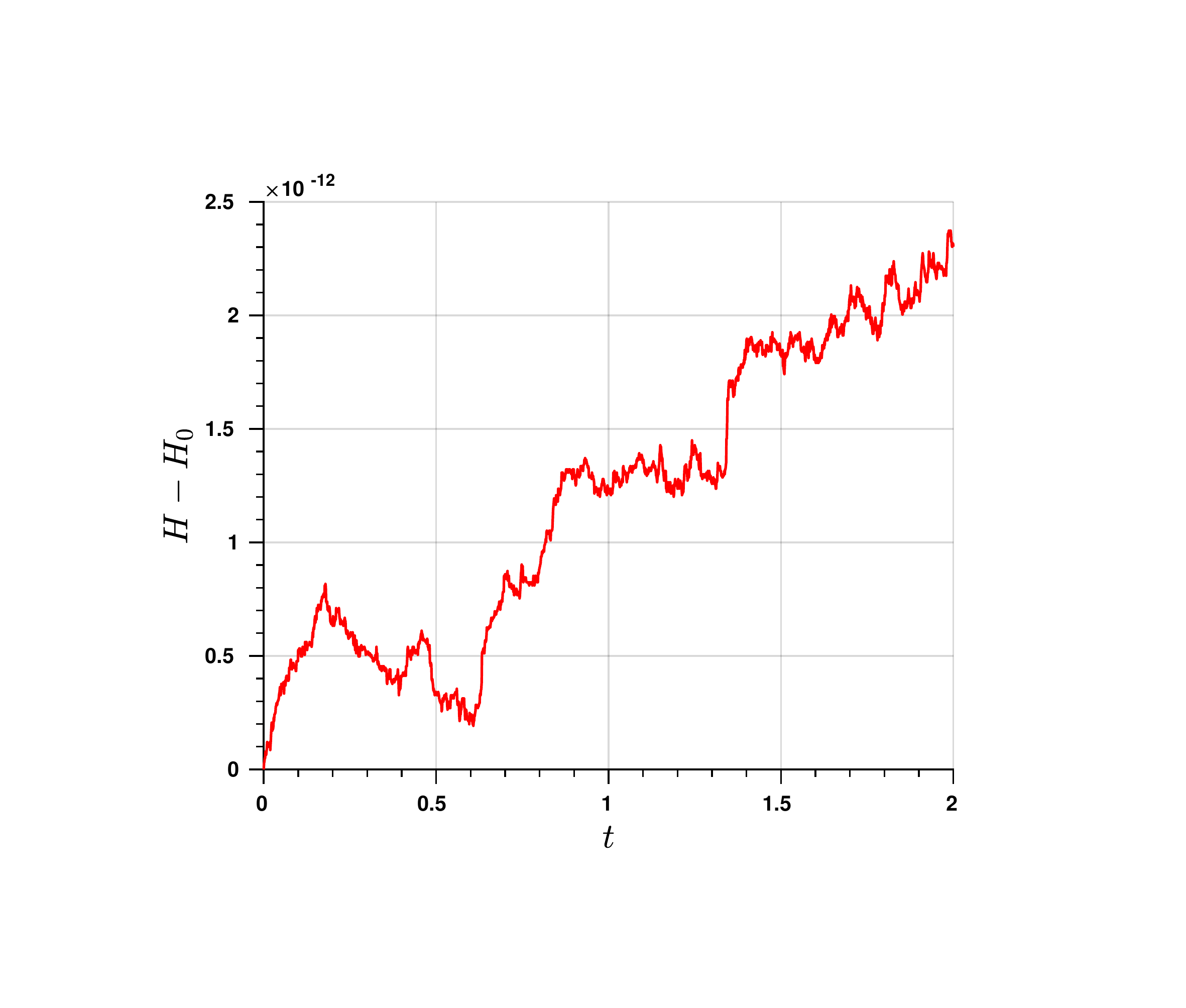} \\
(a) & (b) 
\end{tabular}
\begin{tabular}{ccc}
\includegraphics[angle=0, trim=80 80 130 100, clip=true, scale = 0.3]{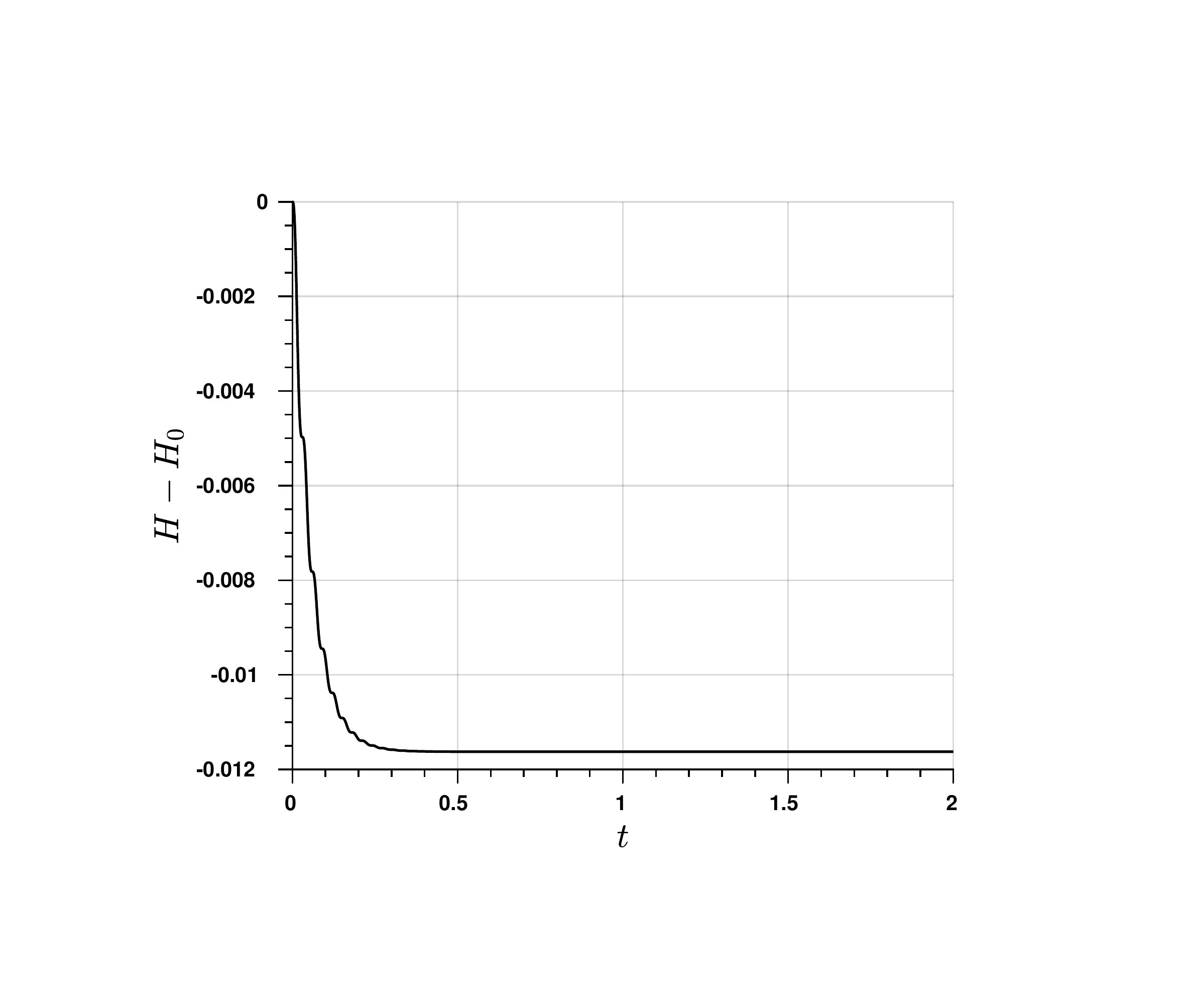} &
\includegraphics[angle=0, trim=80 80 130 100, clip=true, scale = 0.3]{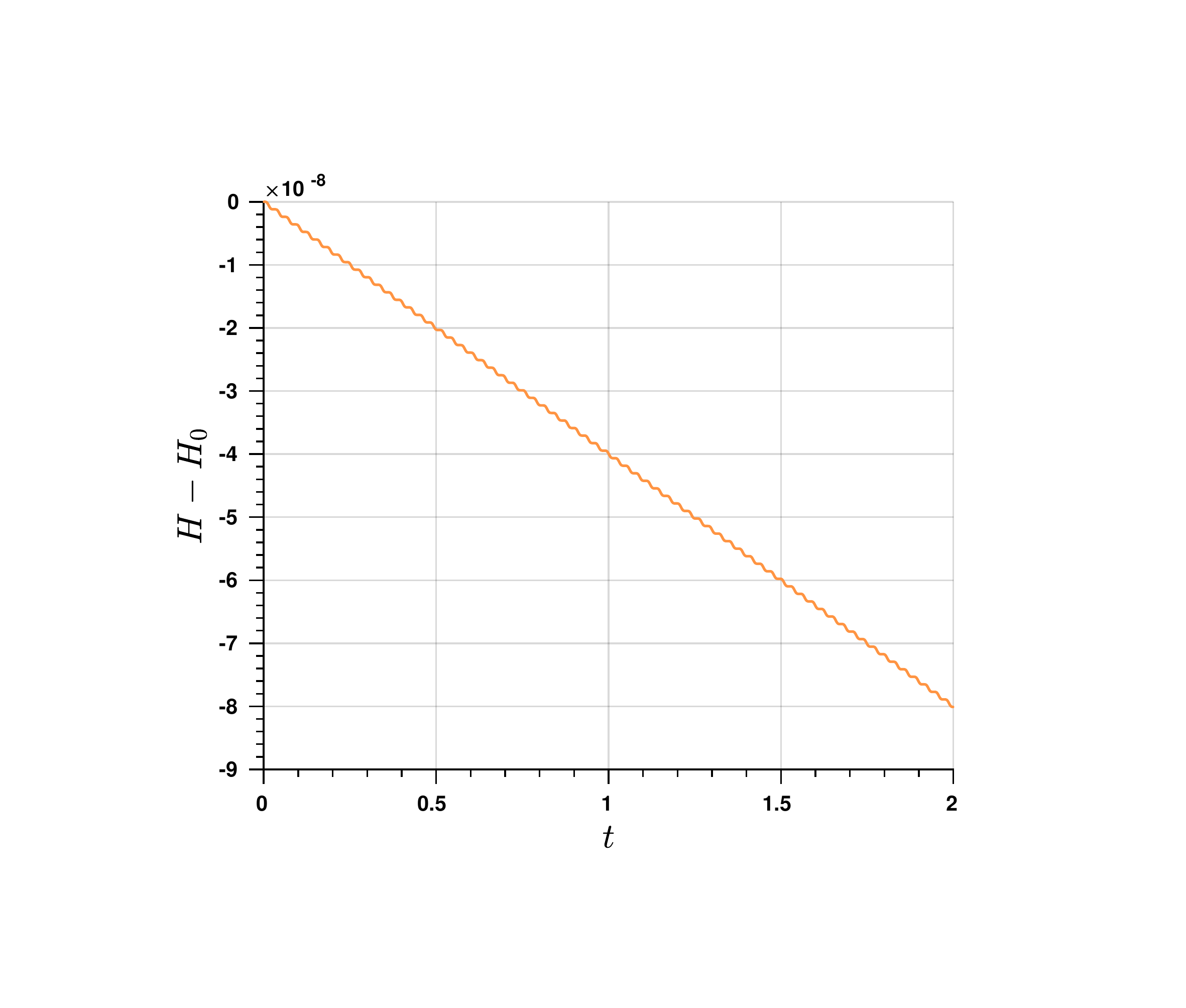} &
\includegraphics[angle=0, trim=80 80 130 100, clip=true, scale = 0.3]{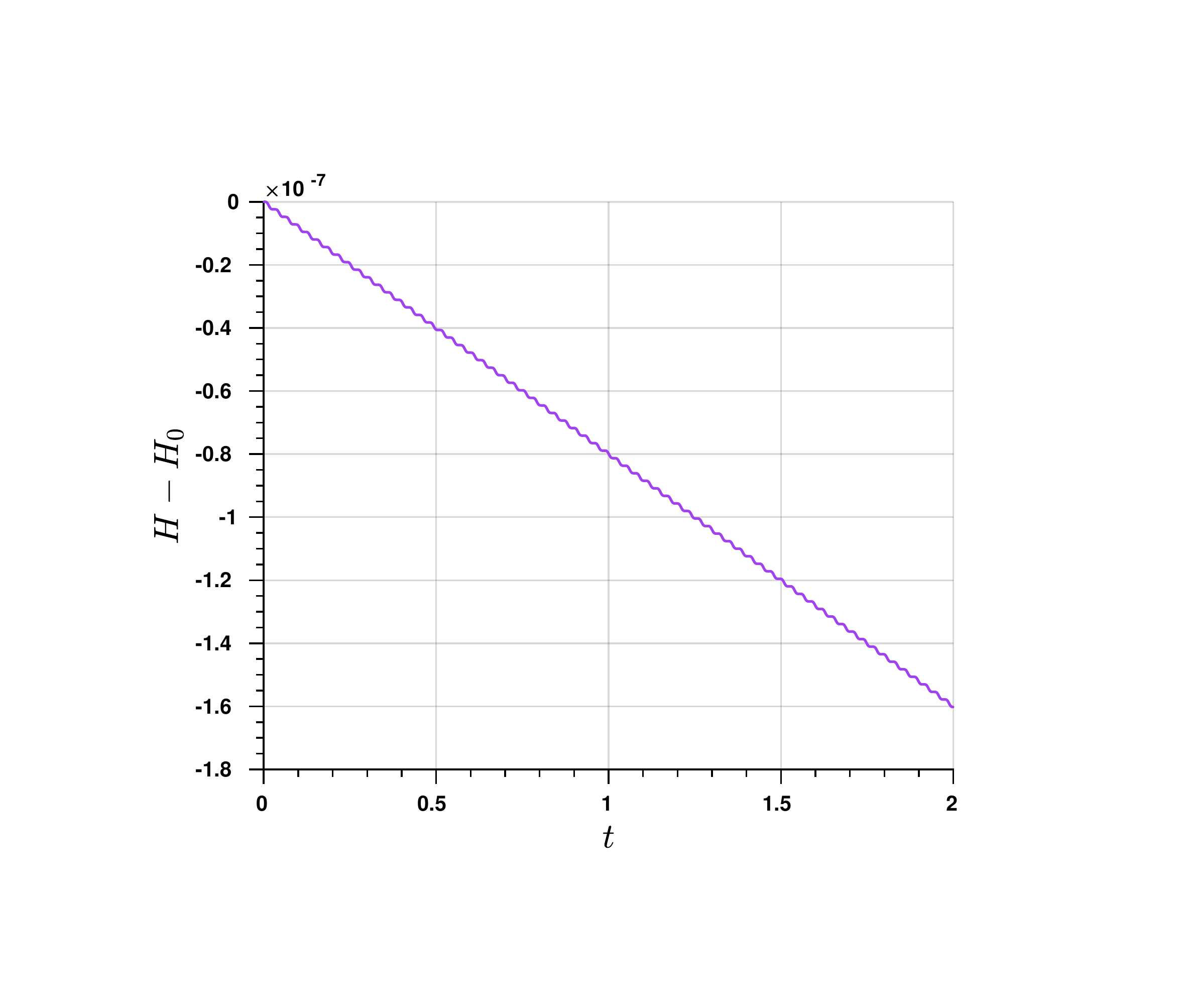} \\
(c) & (d) & (e)
\end{tabular}
\caption{The errors of the total energy $H(\bm q_n, \bm p_n) - H(\bm q_0, \bm p_0)$ of the two-particle problem calculated by (a) the mid-point, (b) LaBudde-Greenspan, (c) generalized Eyre, (d) perturbed mid-point, and (e) perturbed trapezoidal integrators. } 
\label{fig:two_particle_energy}
\end{center}
\end{figure}

\subsection{Numerical results}
The numerical implementation of the many-body dynamics is analogous to the implementation of single-particle dynamics presented in Section \ref{sec:algorithm_implementation}. The parameters of the algorithm are set to be $\mathrm{tol}_{\mathrm R} = 10^{-12}$, $\mathrm{tol}_{\mathrm A} = 10^{-15}$, $l_{\mathrm{max}} = 20$, and $\mathrm{tol}_{\mathrm Q} = 10^{-8}$, unless otherwise specified.

\subsubsection{Two bodies with the Lennard-Jones potential}
In the first example, we consider a two-particle system, in which both particles has unit mass (i.e. $\bm m^{11} = \bm m^{22} = 1$, and $\bm m^{12} = \bm m^{21} = 0$). Their initial locations and momenta are
\begin{align*}
\bm q_{A~0} = \begin{cases} (0, \: -0.5612, \: 0), & A=1 \\ (0, \quad 0.5612, \: 0), & A=2 \end{cases}, \qquad
\bm p^{A}_{0} = \begin{cases} (\: 5, \: 0, \: 0), & A=1 \\ (10, \: 0, \: 0), & A=2 \end{cases}.
\end{align*}
The potential energy takes the form of the Lennard-Jones 12-6 potential given in \eqref{eq:lennard_jones_12_6}, and the splits can be made conveniently in the following manner,
\begin{align*}
\hat{V}_{AB~\mathrm c}(\|\bm q\|) = \hat{V}_{AB~+}(\|\bm q\|) = 4\varepsilon \left( \frac{\sigma}{\|\bm q\|}\right)^{12}, \qquad \hat{V}_{AB~\mathrm e}(\|\bm q\|) = \hat{V}_{AB~-}(\|\bm q\|) = -4 \varepsilon \left( \frac{\sigma}{\|\bm q\|} \right)^6.
\end{align*}
In this study, the model parameters are taken as $\varepsilon=10^2$ and $\sigma=1.0$. To study the accuracy of the three proposed integrators, the mid-point integrator is utilized with $\Delta t_n = 10^{-8}$ to obtain an overkill solution at $T=1.0$. Then the calculations are repeated with larger time steps using the generalized Eyre, perturbed mid-point, and perturbed trapezoidal integrators. The relative errors and the corresponding convergence rates are illustrated in Figure \ref{fig:two_particle_conv_rates}. The errors between the perturbed mid-point and perturbed trapezoidal integrators coincide up to the first three significant digits, and therefore their convergence curves are not distinguished in Figure \ref{fig:two_particle_conv_rates} (b). From the results, it is confirmed that the generalized Eyre integrator is first-order accurate, while the rest two converge quadratically. 

To investigate the properties of the integrators with respect to the invariants, the calculations are performed with a uniform time step size $\Delta t_n=1\times 10^{-3}$ up to $T=2.0$. Again, it can be observed that the energy error of the mid-point integrator exhibits an oscillatory pattern. For the LaBudde-Greenspan integrator, the energy error mildly grows over time, which can be attributed to the nonlinear solver accuracy. In our test, the magnitude of the energy error will be amplified with a looser tolerance $\mathrm{tol}_{\mathrm R}$. For the three proposed integrators, the energy monotonically decays. Again, the dissipation errors from the two energy-decaying integrators are less than the energy oscillation magnitude in the mid-point integrator. In Figure \ref{fig:two_particle_momentum}, the errors of the momenta and the center of mass are depicted, which confirms that the momenta and the mass center are well preserved in all five considered integrators.

\begin{figure}
	\begin{center}
\begin{tabular}{ccc}
\multicolumn{3}{c}{ \includegraphics[angle=0, trim=780 265 690 1730, clip=true, scale = 0.45]{./momentum_legend} }\\
\includegraphics[angle=0, trim=80 80 130 100, clip=true, scale = 0.3]{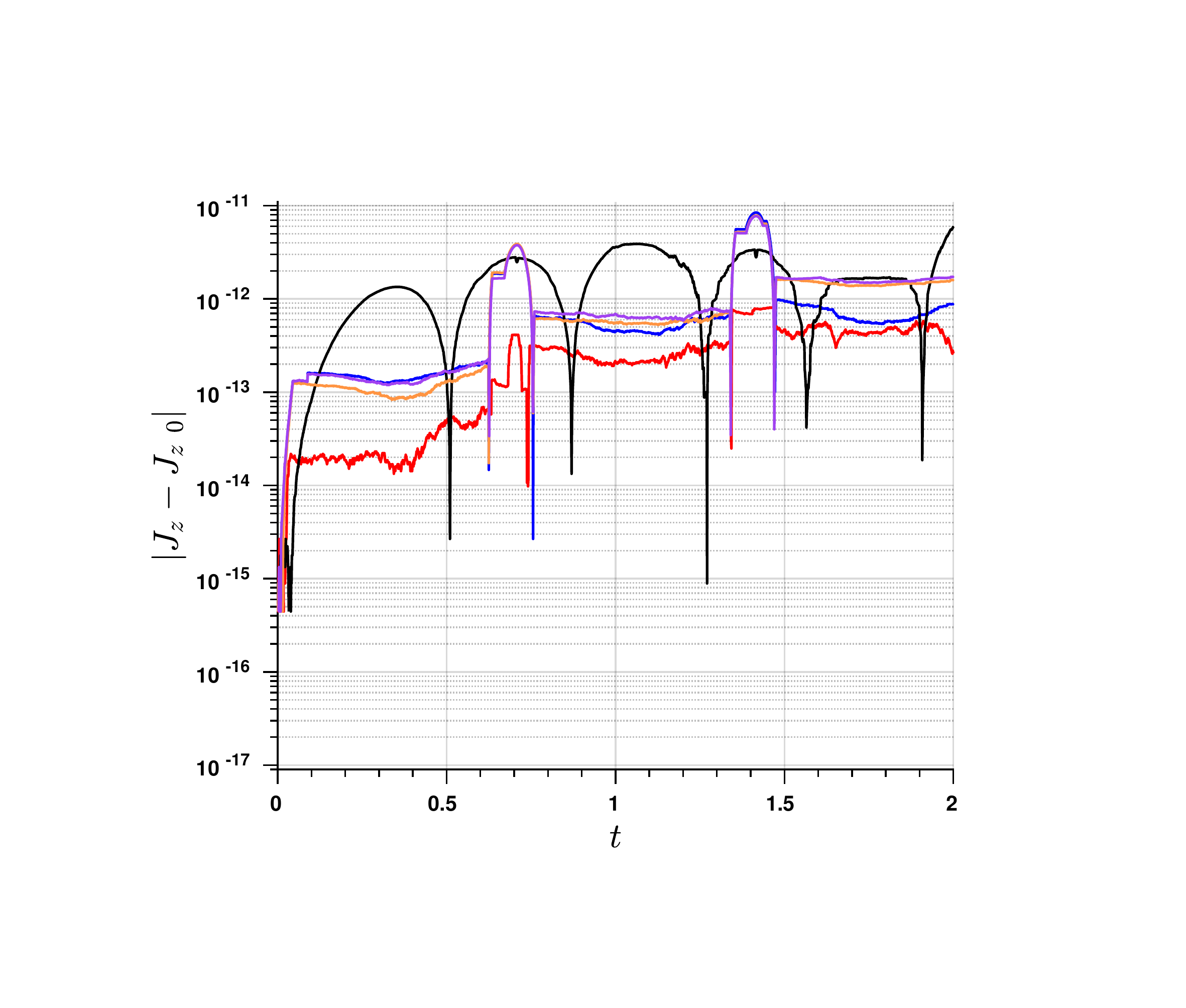} &
\includegraphics[angle=0, trim=80 80 130 100, clip=true, scale = 0.3]{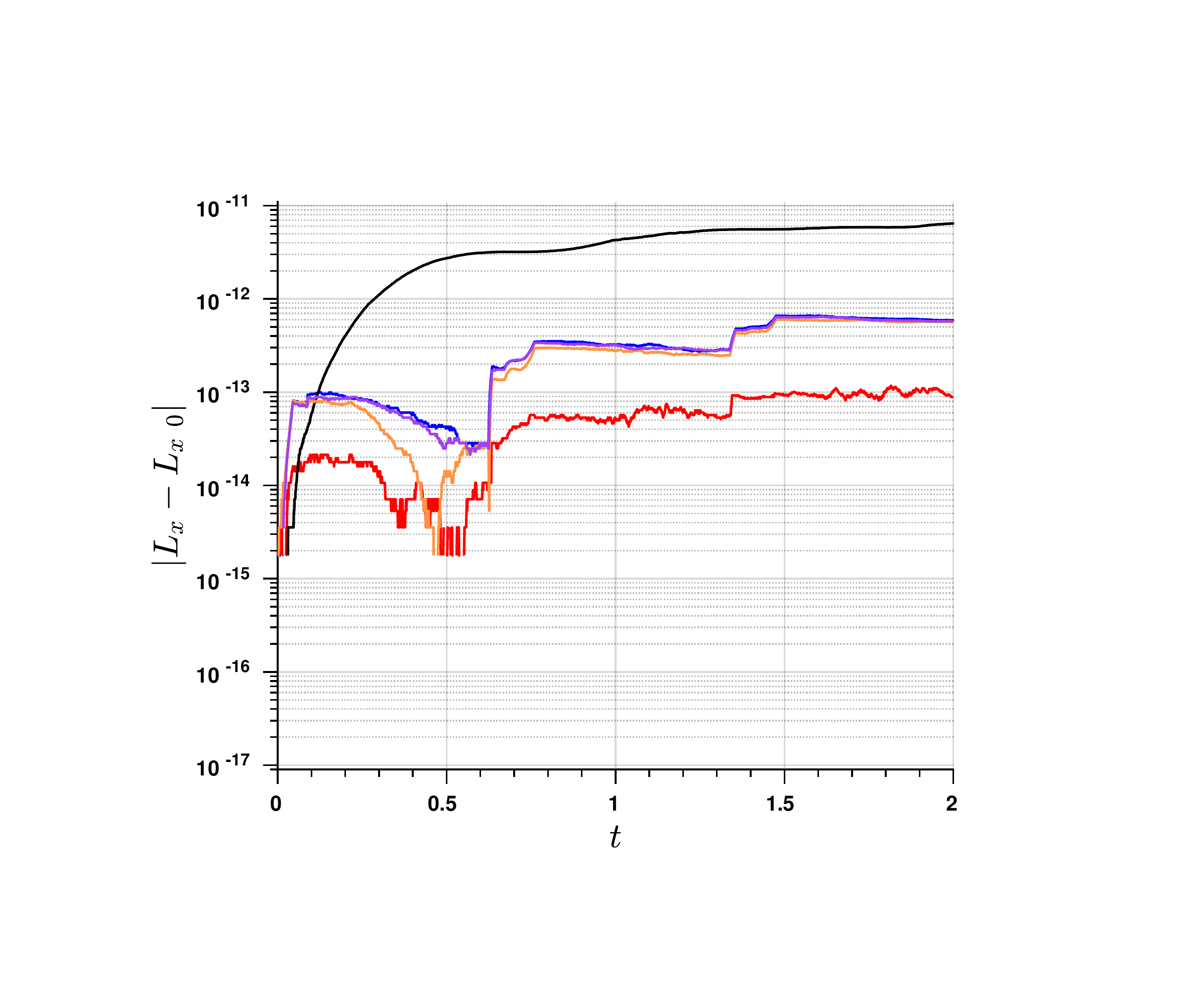} &
\includegraphics[angle=0, trim=80 80 130 100, clip=true, scale = 0.3]{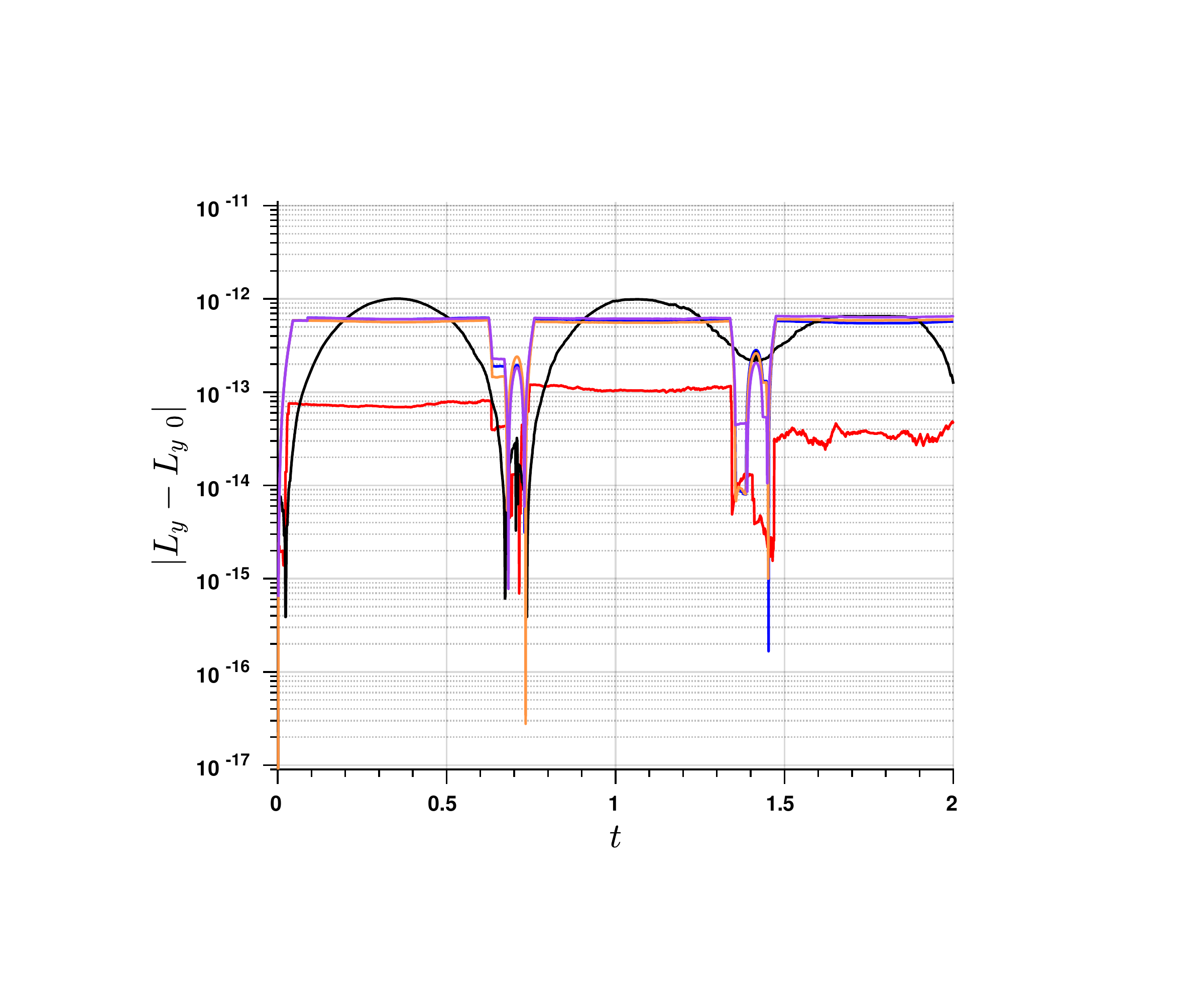} \\
(a) & (b) & (c)
\end{tabular}
\begin{tabular}{cc}
\includegraphics[angle=0, trim=80 80 130 100, clip=true, scale = 0.3]{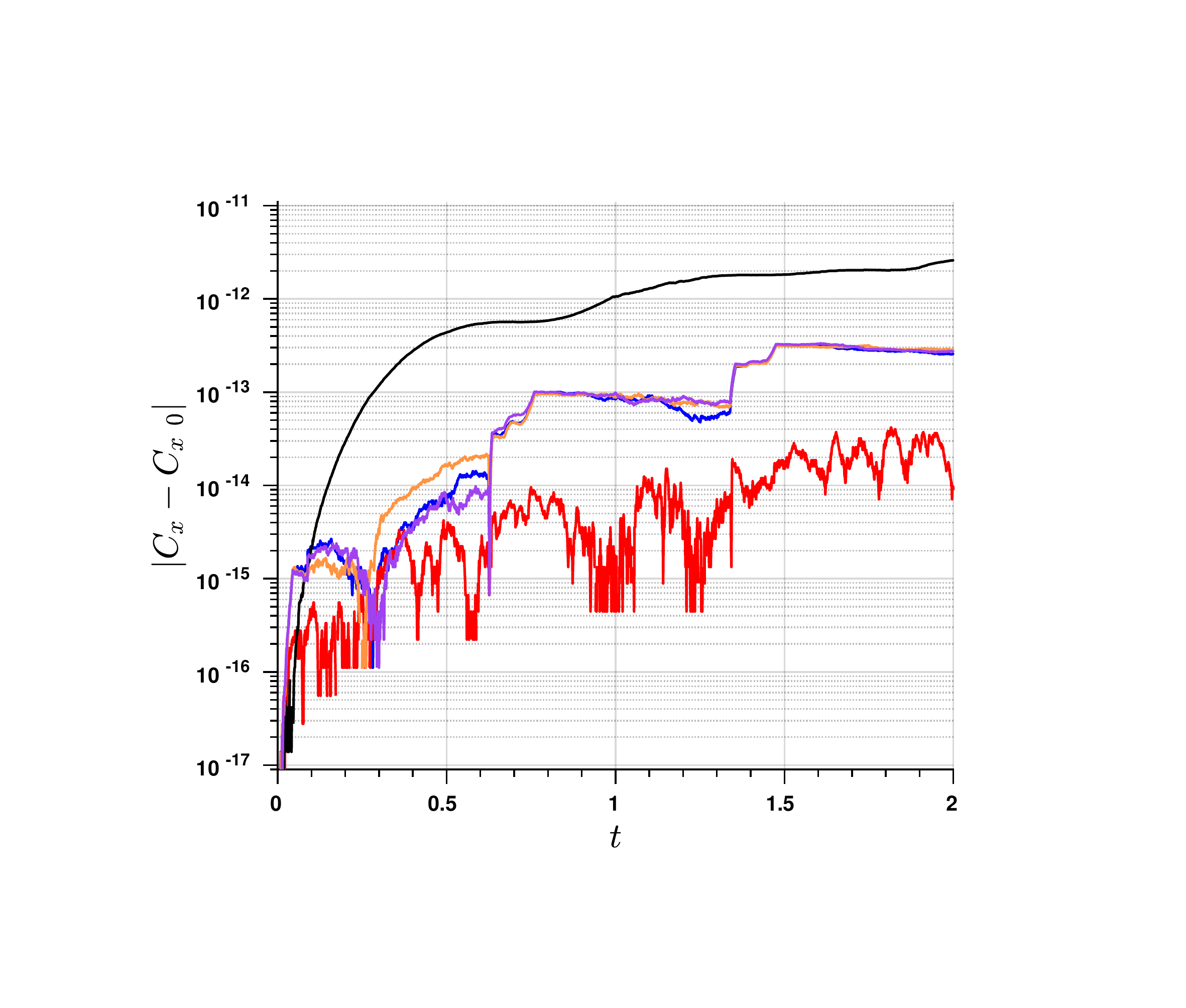} &
\includegraphics[angle=0, trim=80 80 130 100, clip=true, scale = 0.3]{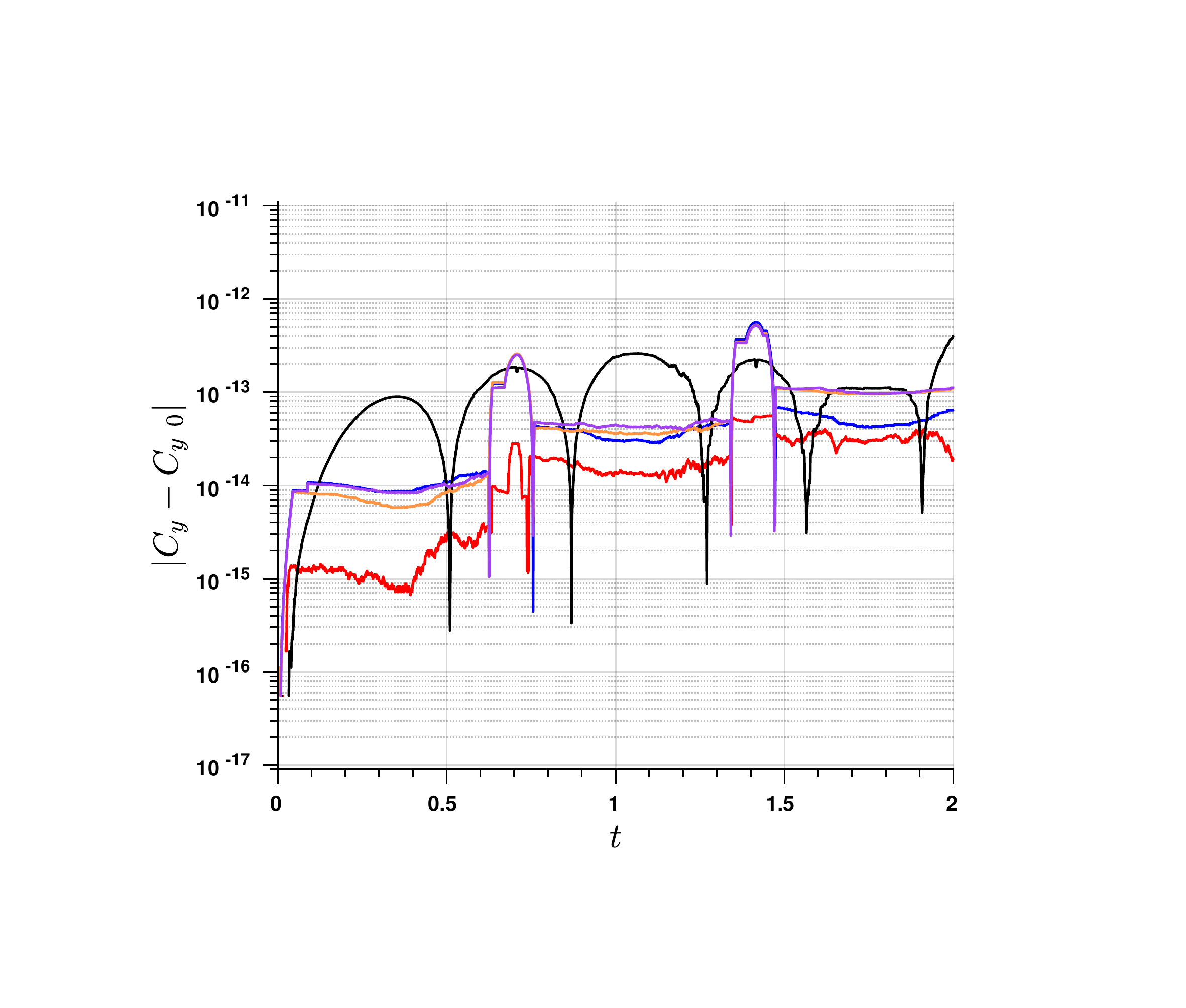} \\
(d) & (e)
\end{tabular}
\caption{The errors (a) $J_z(\bm q_n, \bm p_n) - J_z(\bm q_0, \bm p_0)$, (b) $L_x(\bm q_n, \bm p_n) - L_x(\bm q_0, \bm p_0)$, (c) $L_y(\bm q_n, \bm p_n) - L_y(\bm q_0, \bm p_0)$, (d) $C_x(\bm q_n, \bm p_n) - C_x(\bm q_0, \bm p_0)$, and (e) $C_y(\bm q_n, \bm p_n) - C_y(\bm q_0, \bm p_0)$ of the two-particle problem over time calculated the mid-point, LaBudde-Greenspan, generalized Eyre, perturbed mid-point, and perturbed trapezoidal integrators. The rest components of the invariants are identically zero from all integrators and are thus not illustrated.} 
\label{fig:two_particle_momentum}
\end{center}
\end{figure}

Regarding the robustness of the LaBudde-Greenspan integrator, we repeated the calculations using the same temporal settings (i.e., $\Delta t_n=1\times 10^{-3}$ and $T=2.0$), with varying values of $\mathrm{tol}_{\mathrm{Q}}$. The behavior of the energy conservation is depicted in Figure \ref{fig:two_particle_hybird_labudde_greenspan}. From the results, we may first observe that the default option renders the default LaBudde-Greenspan behave like the mid-point integrator when the tolerance is loose (i.e., $\mathrm{tol}_{\mathrm Q} = 10^{-4}$). With the same loose tolerance, the LaBudde-Greenspan combined with the generalized Eyre and the perturbed mid-point options produces energy-decaying results. When using the perturbed mid-point integrator as the alternate option, the energy is not monotonically decaying. An explanation is that the amount of dissipation is not strong enough to counterbalance the error from the nonlinear solver. When the tolerance $\mathrm{tol}_{\mathrm Q}$ is $10^{-6}$, we see opposite energy behavior at the \textit{same} time instance between the default option and the generalized Eyre option. This indicates that the energy growth and decaying are indeed triggered by the alternate options, rather than the solution of the LaBudde-Greenspan integrator. For the option based on the perturbed mid-point integrator, the energy error oscillates around zero with rather small magnitude, indicating the dissipation and the solver error are balanced and the energy is preserved well. In the third column of Figure \ref{fig:two_particle_hybird_labudde_greenspan}, we may observe similar patterns of the energy error between the default option and the option based on the perturbed mid-point integrator. This is likely attributed by the accuracy of the nonlinear solver, and the effect of the alternate options are negligible. Yet, there is still a clear damping effect when the generalized Eyre is adopted. The results here suggest that the proposed energy-decaying integrator can be beneficial in ensuring the energy stability when used together with the LaBudde-Greenspan integrator.

\begin{figure}
	\begin{center}
\begin{tabular}{|c|c|c|}
\hline
\includegraphics[angle=0, trim=80 80 120 80, clip=true, scale = 0.28]{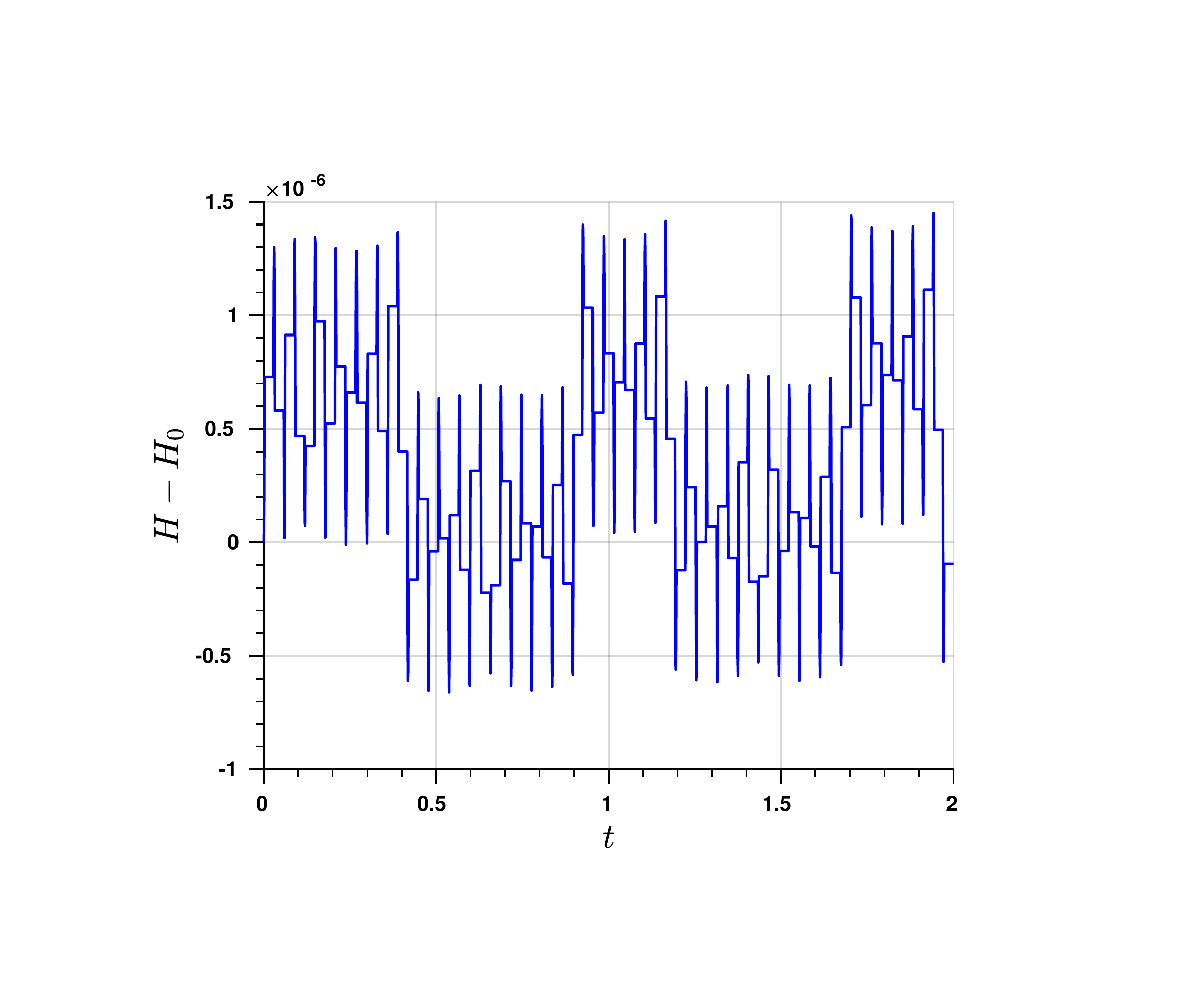} &
\includegraphics[angle=0, trim=80 80 120 80, clip=true, scale = 0.28]{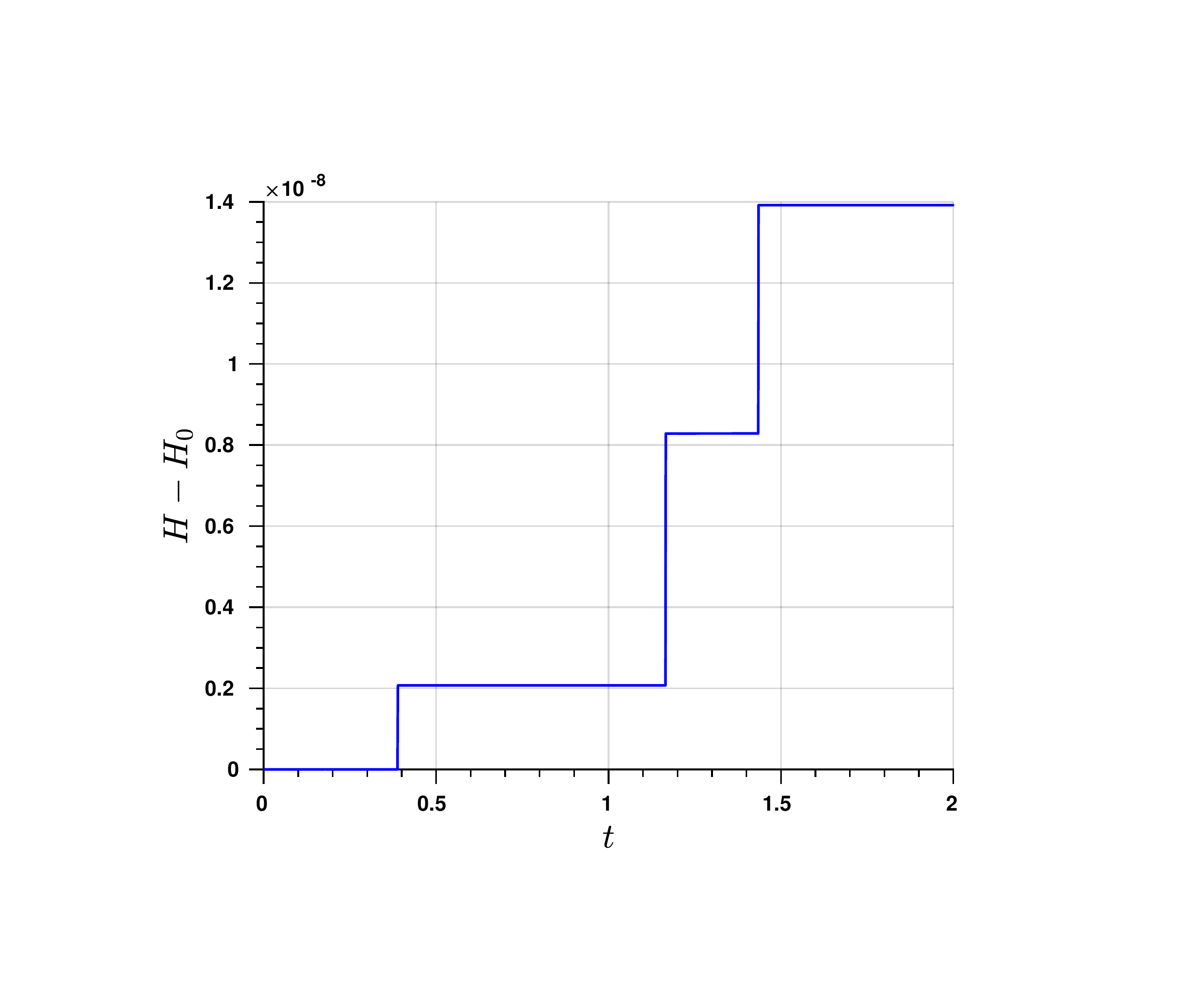} &
\includegraphics[angle=0, trim=80 80 120 80, clip=true, scale = 0.28]{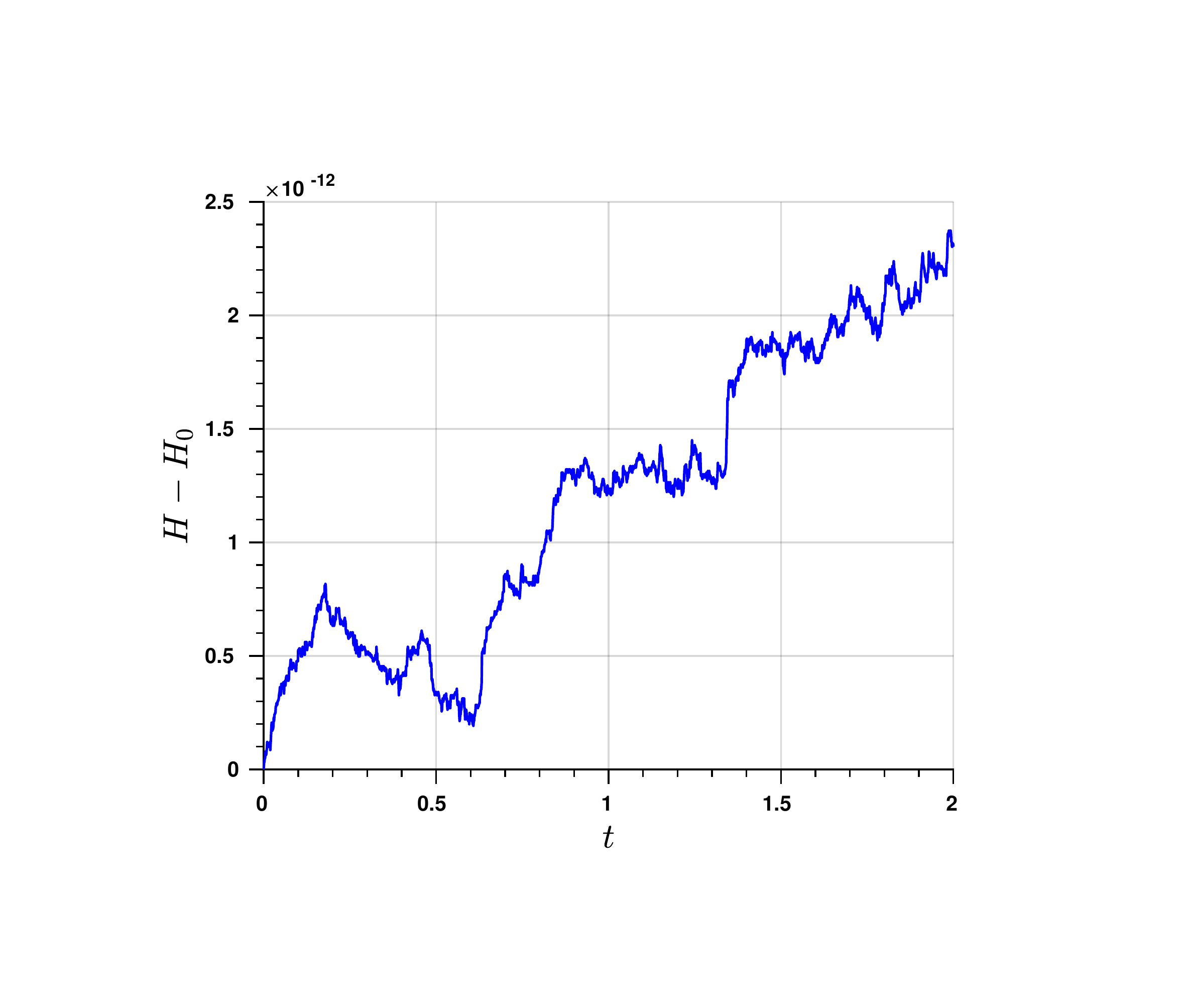} \\
\includegraphics[angle=0, trim=80 80 120 100, clip=true, scale = 0.28]{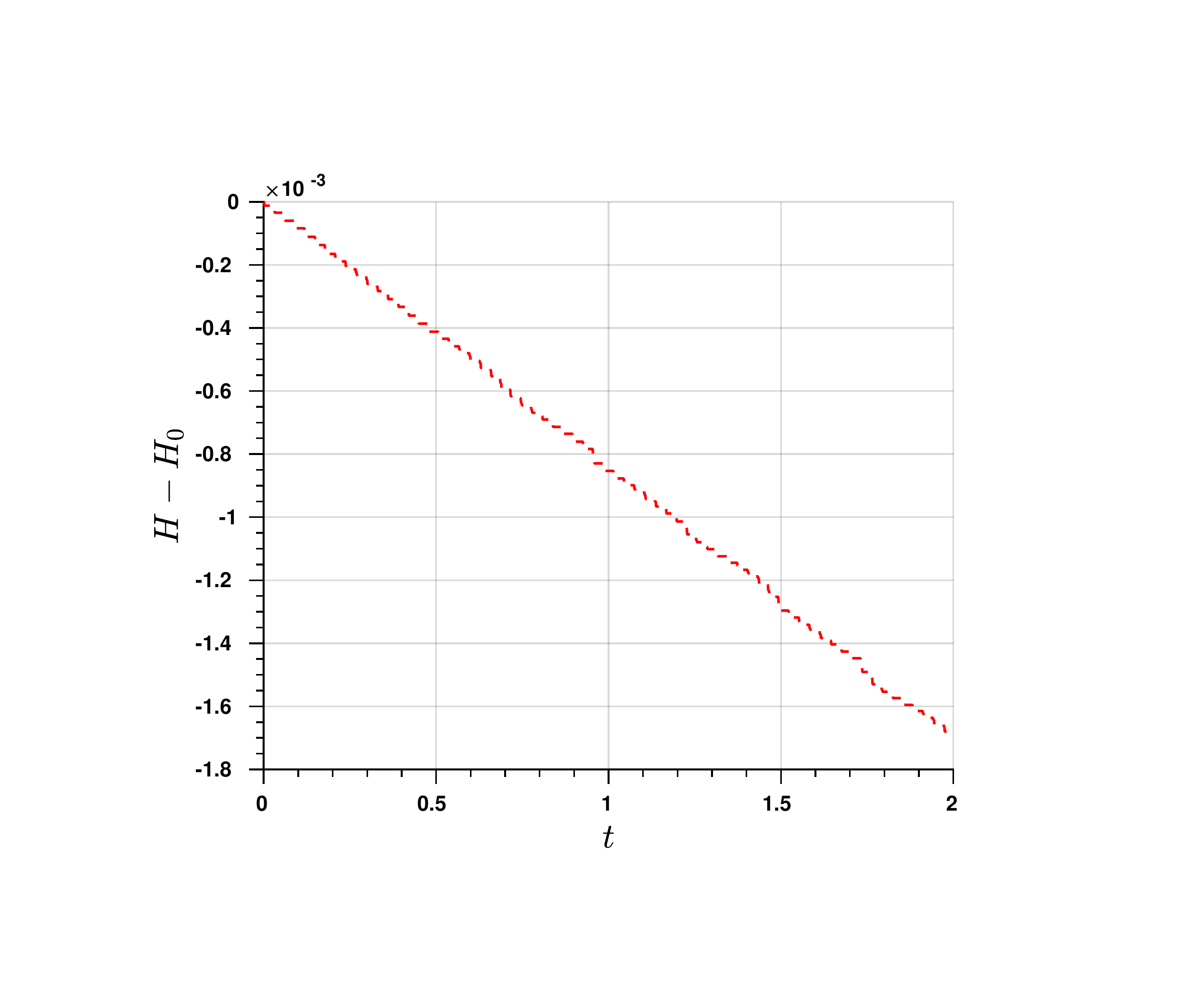} &
\includegraphics[angle=0, trim=80 80 120 100, clip=true, scale = 0.28]{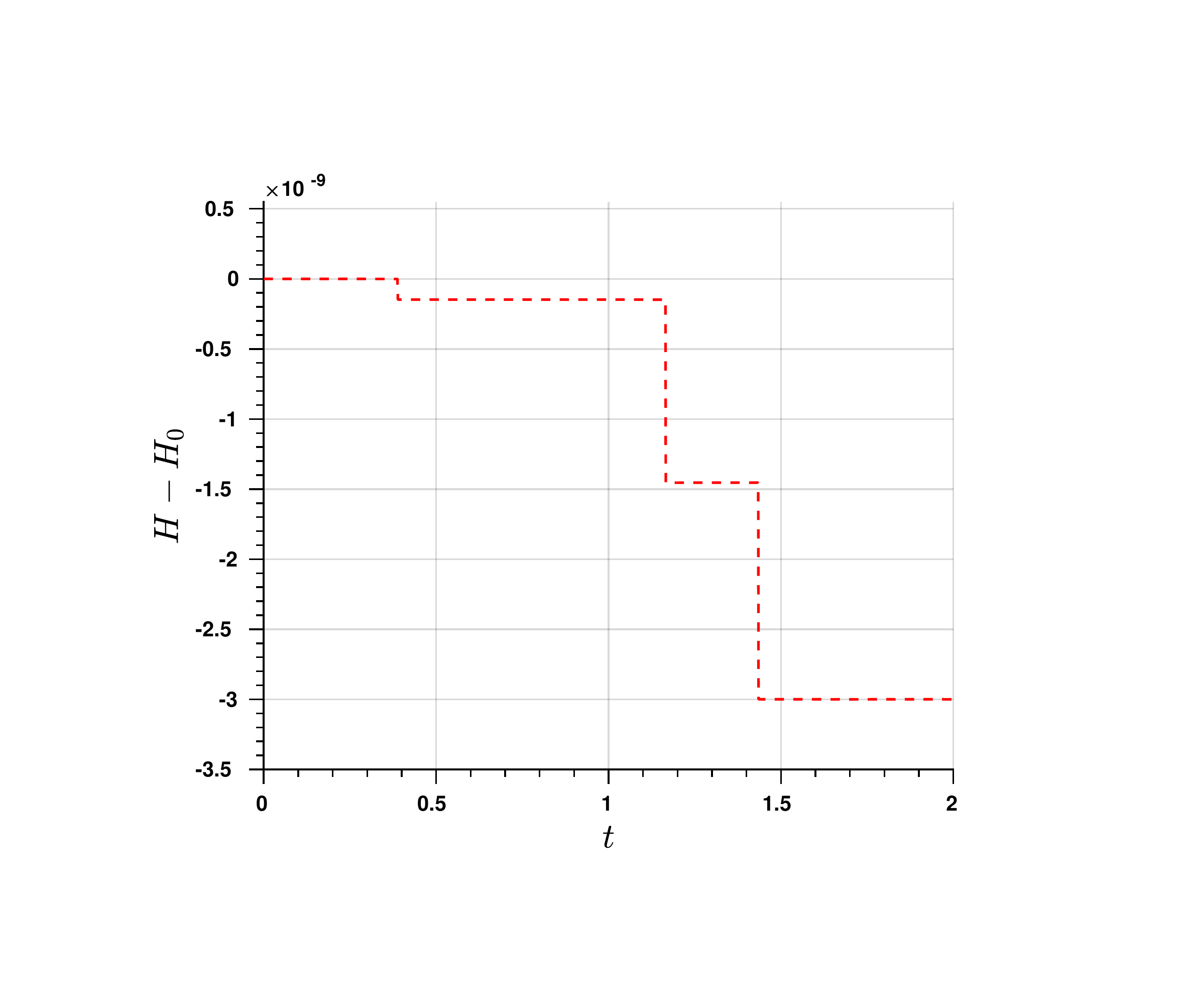} &
\includegraphics[angle=0, trim=80 80 120 100, clip=true, scale = 0.28]{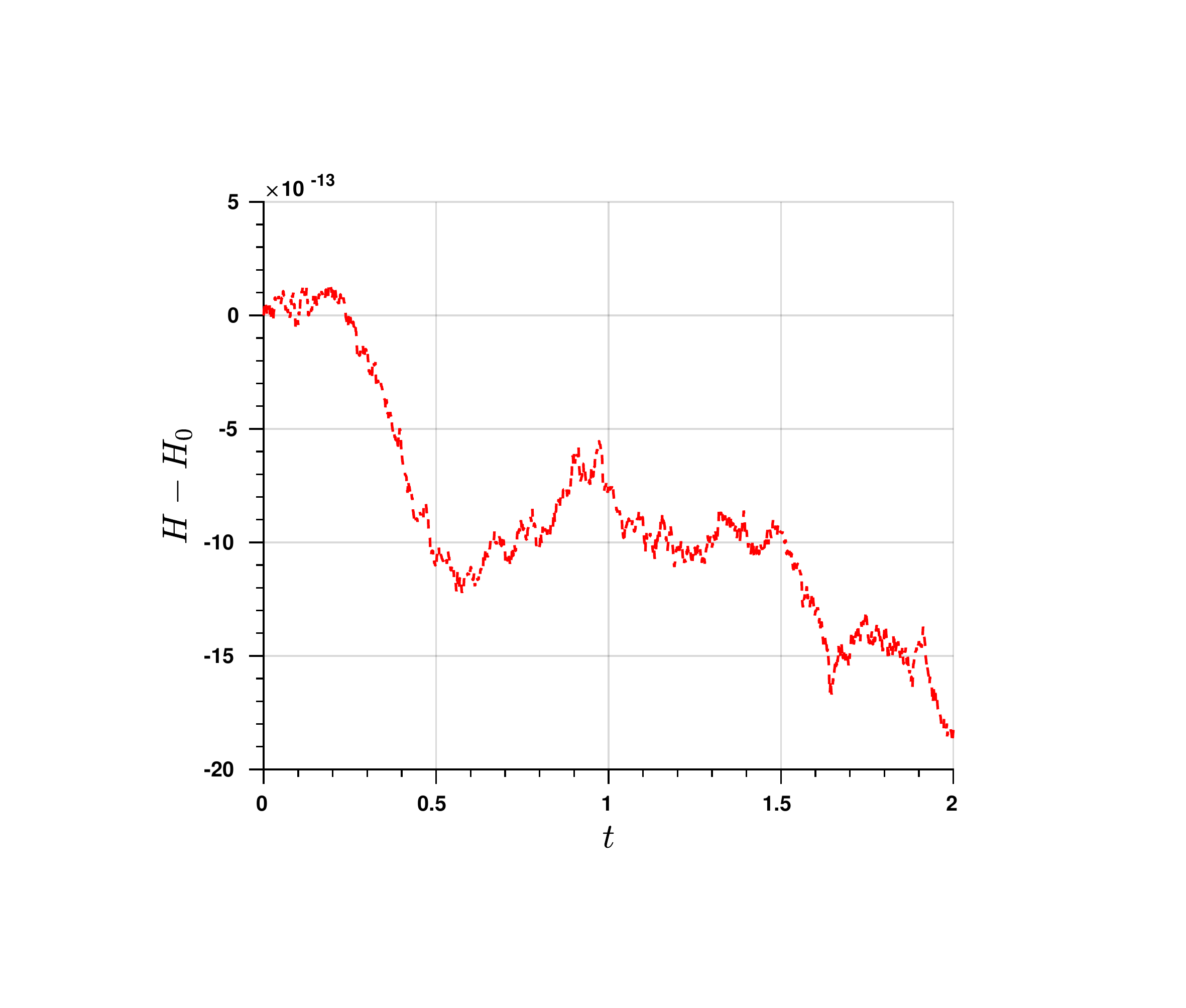} \\
\includegraphics[angle=0, trim=80 80 120 100, clip=true, scale = 0.28]{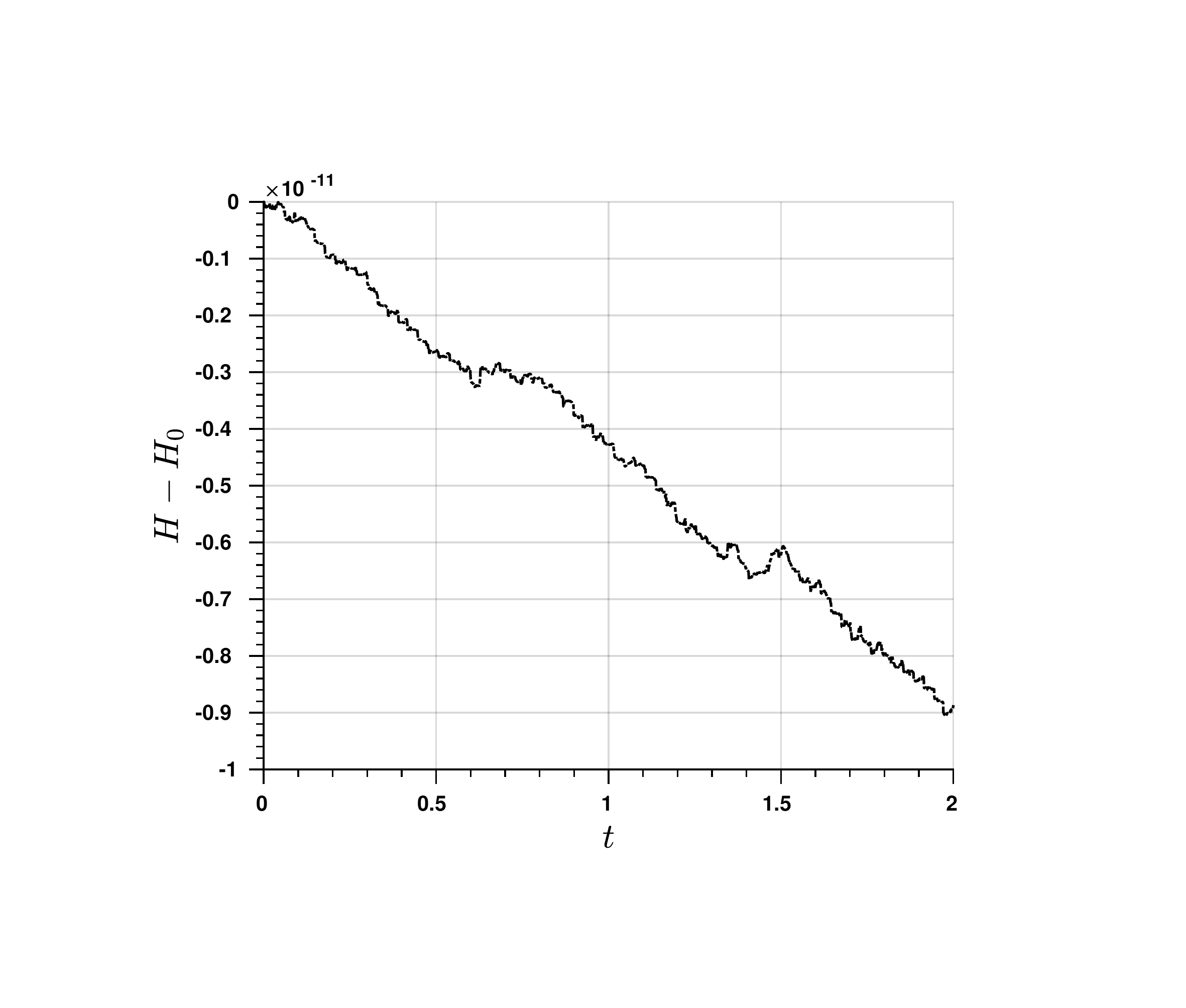} &
\includegraphics[angle=0, trim=80 80 120 100, clip=true, scale = 0.28]{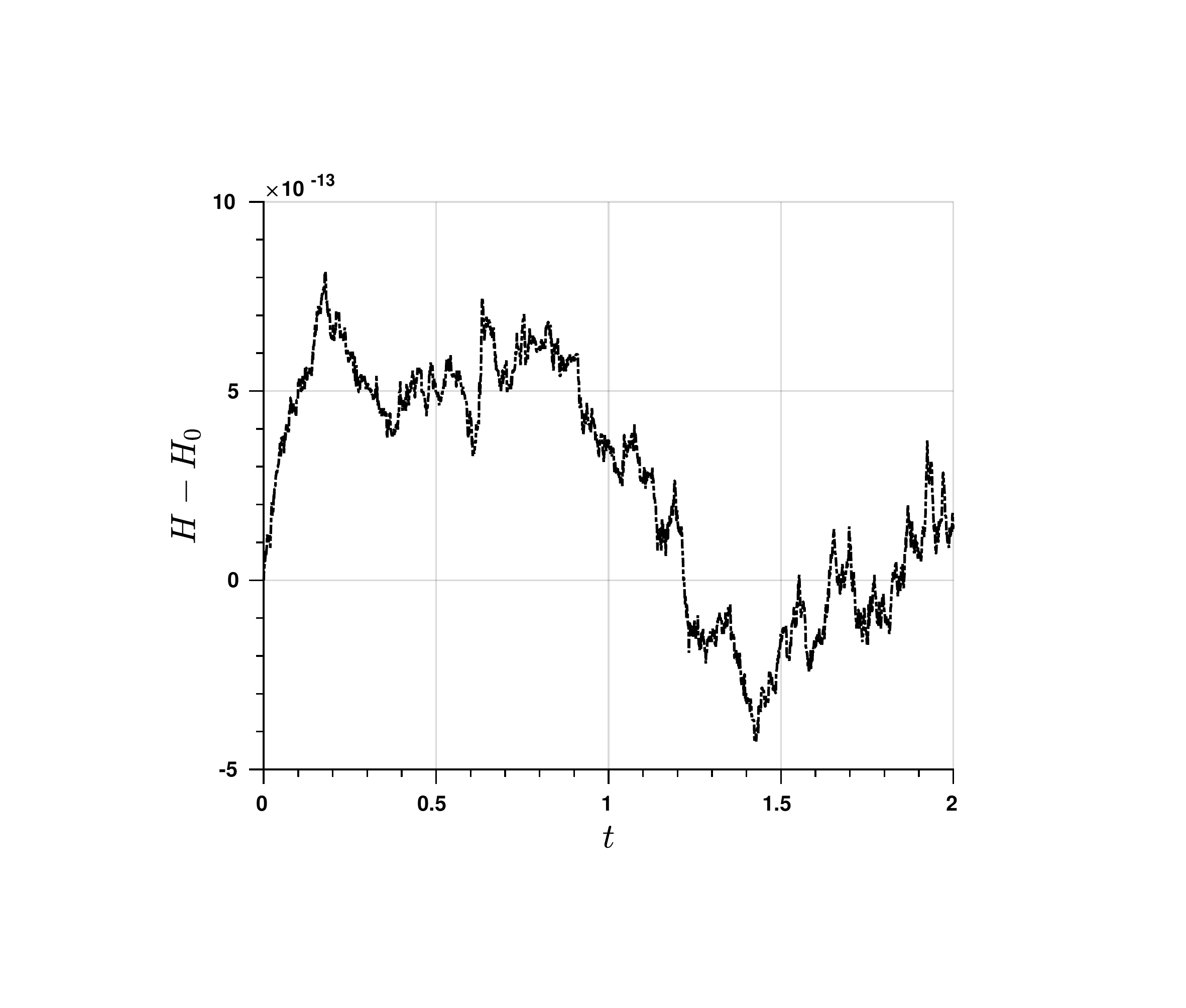} &
\includegraphics[angle=0, trim=80 80 120 100, clip=true, scale = 0.28]{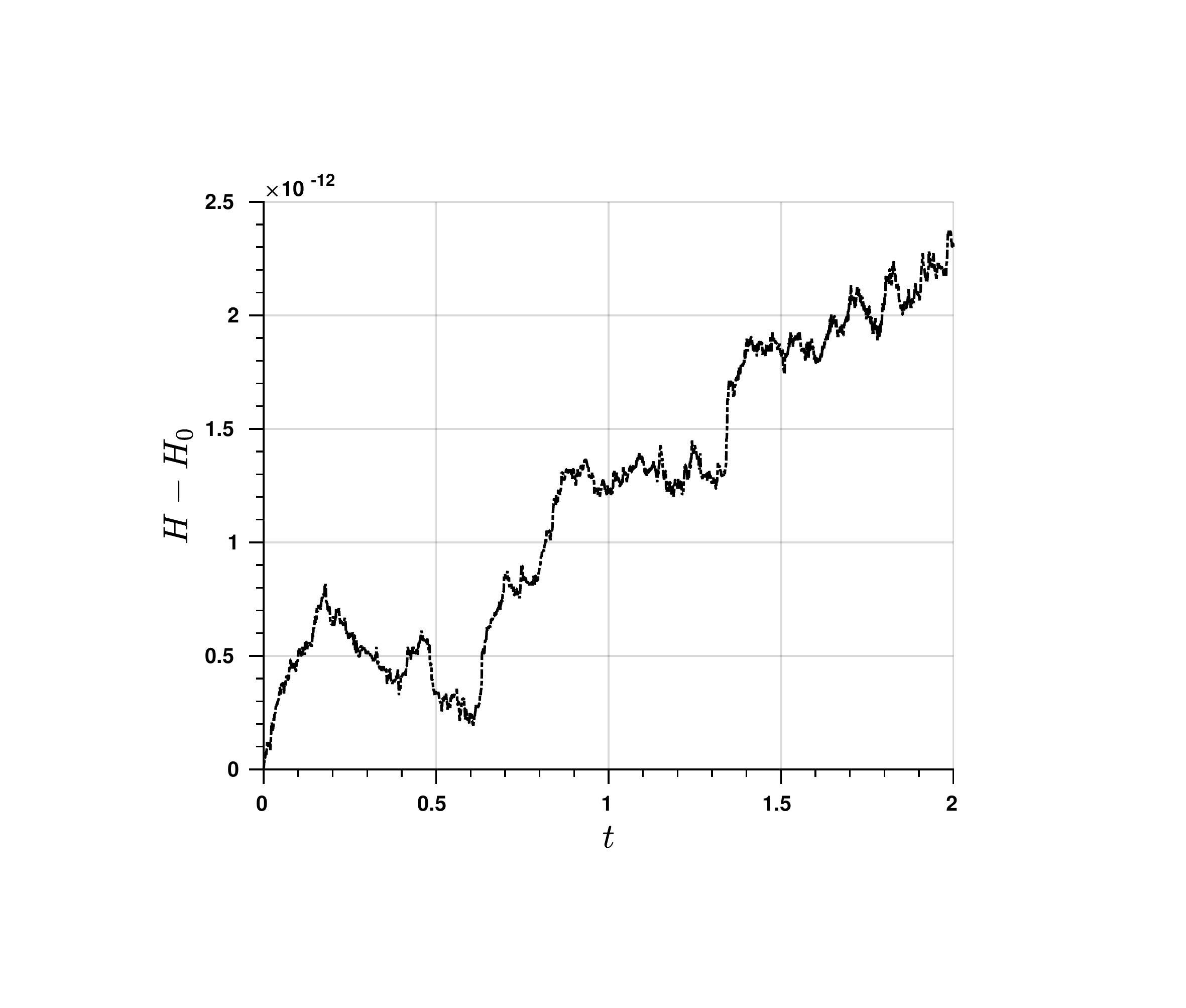} \\
$\mathrm{tol}_{\mathrm Q} = 10^{-4}$ & $\mathrm{tol}_{\mathrm Q} = 10^{-6}$ & $\mathrm{tol}_{\mathrm Q} = 10^{-8}$ \\
\hline
\end{tabular}
\caption{The error $H(\bm q_n, \bm p_n) - H(\bm q_0, \bm p_0)$ of the two-particle problem calculated by the LaBudde-Greenspan integrator combined with the default option \eqref{eq:labudde-greenspan-default-switch} (blue solid), the generalized Eyre (red dashed), and the perturbed mid-point (black dash-dotted) integrators when $|\|\bm q_{n+1,(l)}\| - \|\bm q_{n}\|| \leq \mathrm{tol}_{\mathrm Q}$. The left, middle, and right columns are results with $\mathrm{tol}_{\mathrm Q} = 10^{-4}$, $10^{-6}$, and $10^{-8}$, respectively.} 
\label{fig:two_particle_hybird_labudde_greenspan}
\end{center}
\end{figure}

\subsubsection{Ten bodies with the gravitational potential}
In this example, we examine the integrators using the standard solar system, consisting of the sun and nine planets. The mass matrix is diagonal, meaning $\bm m_{AB} = m_{A}$ if $A=B$ and $\bm m_{AB} = 0$ if $A \neq B$. Let $G$ denote the gravitational constant. The split of the gravitational potential $\hat{V}_{AB}(r) = -Gm_{A} m_{B}/r$ can be made in the following manner,
\begin{align*}
\hat{V}_{AB~\mathrm{c}}(r) = \hat{V}_{AB~+}(r) = 0, \qquad \hat{V}_{AB~\mathrm{e}}(r) = \hat{V}_{AB~-}(r) = -\frac{G m_A m_B}{r}.
\end{align*}
The initial conditions and the masses are obtained from Tables 5 and 8 of \cite{Folkner2014}. The reference mass is taken as the mass of the sun, the reference length scale is chosen as the astronomical unit, and the reference time scale is earth day. With those units, the gravitational constant is $G=2.95912208286 \times 10^{-4}$. This example has also been considered in \cite{Wan2022} to assess various integrators. Following that work, we integrated the problem with a uniform time step size $\Delta t_n=5$ up to $T=1.825\times 10^6$, which is approximately five thousand earth years. We also integrated the problem with the mid-point integrator and $\Delta t_n=0.1$ to obtain a reference solution. For the LaBudde-Greenspan integrator, the discrete conservative force can be explicitly represented by
\begin{align}
\label{eq:solar_flg_1}
\bm f^{AB}_{\mathrm{lg}} =& 2\Delta t_n \frac{\hat{V}_{AB}(d_{AB~n+1}) - \hat{V}_{AB}(d_{AB~n})}{d_{AB~n+1} - d_{AB~n}} \frac{\bm q_{A~n+\frac12} - \bm q_{B~n+\frac12}}{d_{AB~n+1} + d_{AB~n}} \\
\label{eq:solar_flg_2}
=& 2\Delta t_n Gm_A m_B \frac{\bm q_{A~n+\frac12} - \bm q_{B~n+\frac12}}{d_{AB~n+1} d_{AB~n}(d_{AB~n+1} + d_{AB~n})},
\end{align}
which does not suffer from the singular behavior of the quotient formula. 

\begin{figure}
	\begin{center}
\begin{tabular}{ccc}
\includegraphics[angle=0, trim=80 80 130 80, clip=true, scale = 0.28]{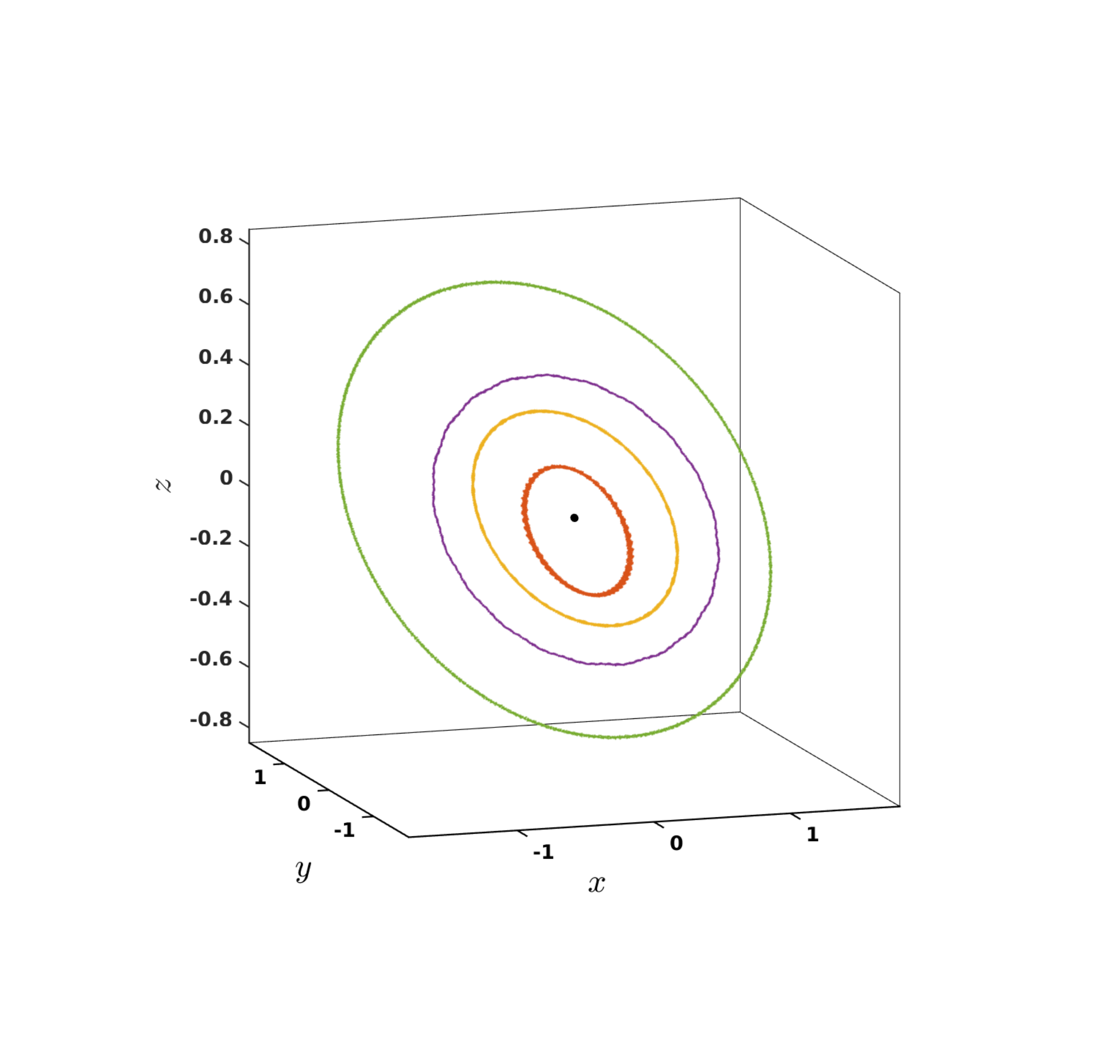} &
\includegraphics[angle=0, trim=80 80 130 80, clip=true, scale = 0.28]{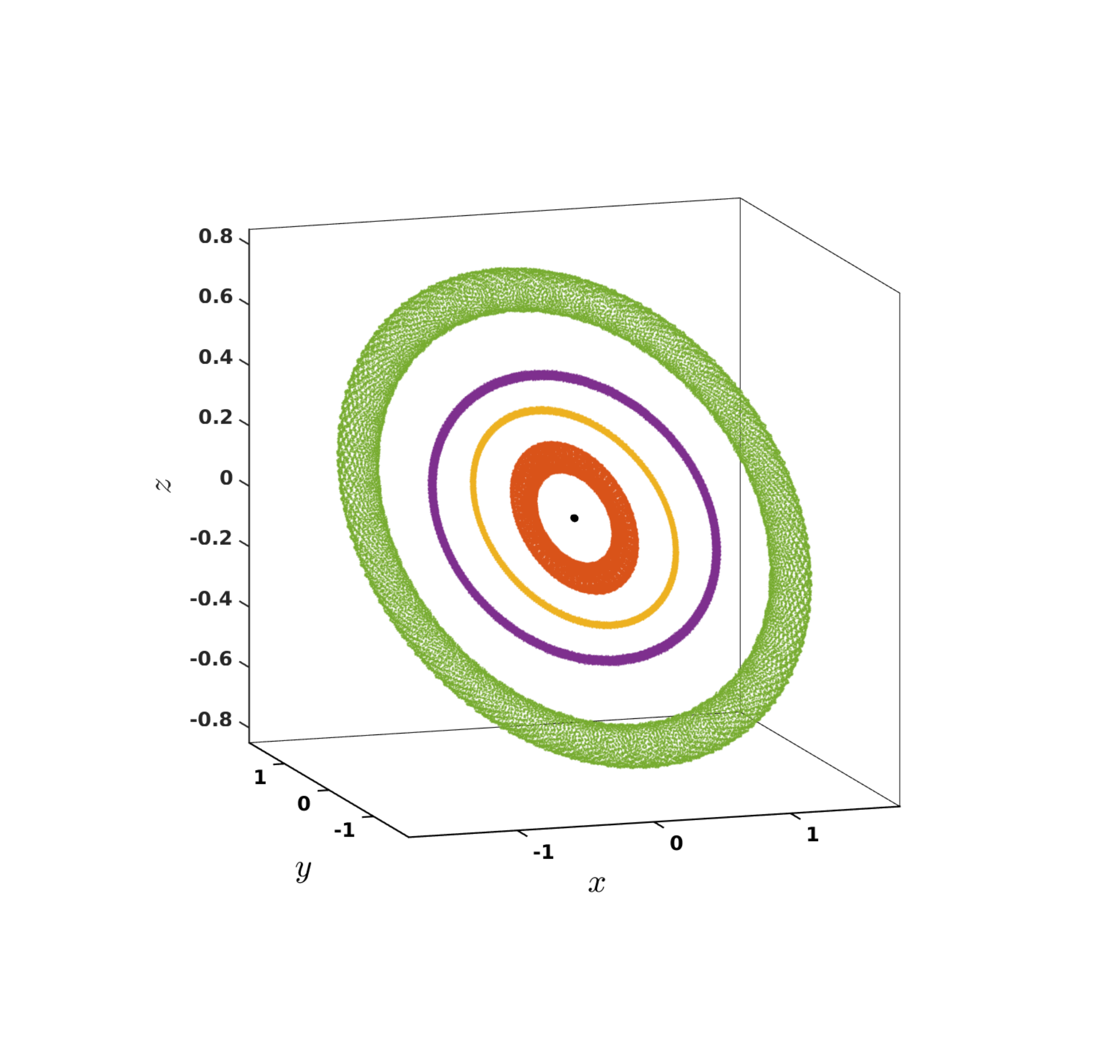} &
\includegraphics[angle=0, trim=80 80 130 80, clip=true, scale = 0.28]{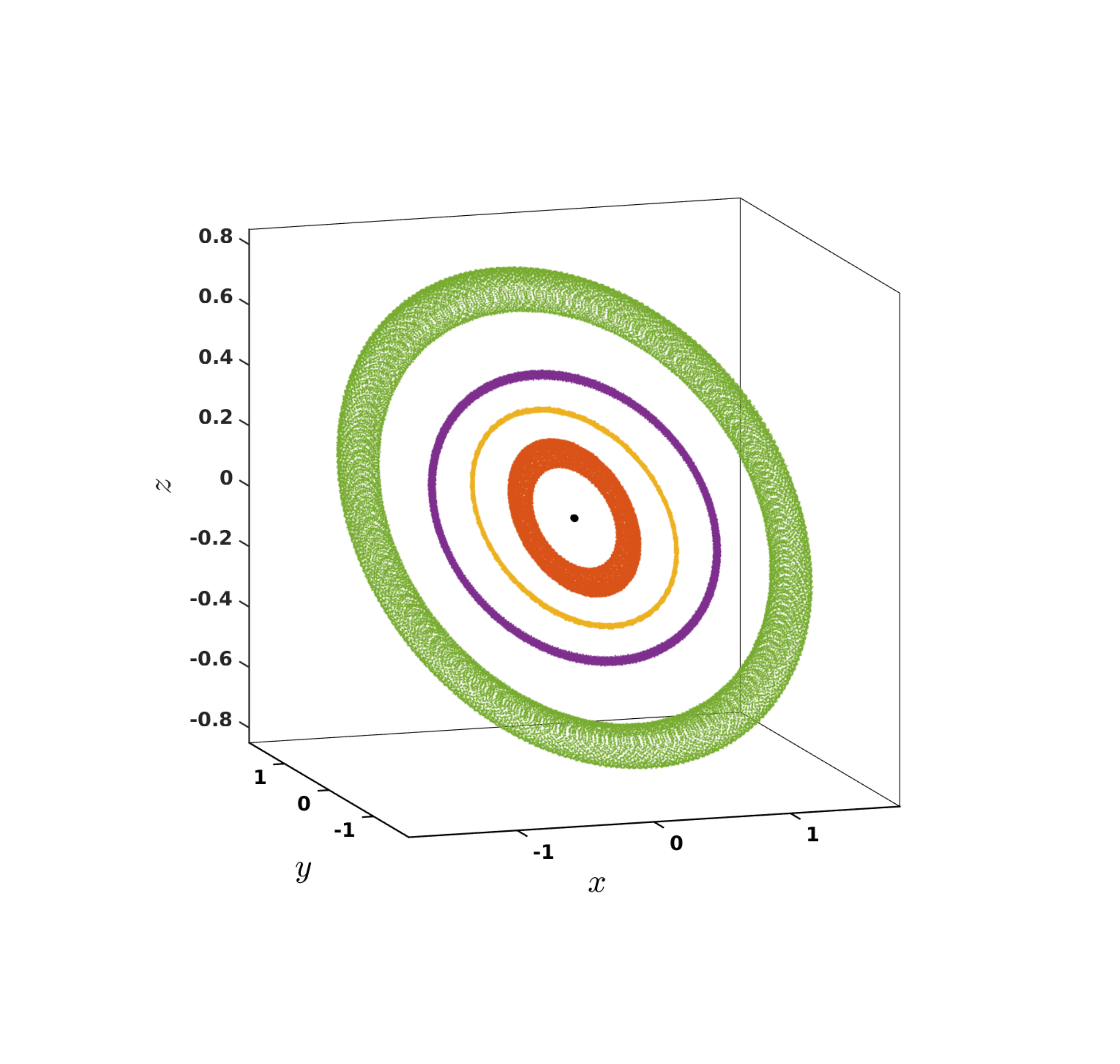} \\
$\bm f^{AB}_{\mathrm{mp}}$ and $\Delta t=0.1$ & $\bm f^{AB}_{\mathrm{mp}}$ and $\Delta t=5$ & $\bm f^{AB}_{\mathrm{lg}}$ and $\Delta t=5$ \\
\includegraphics[angle=0, trim=80 80 130 80, clip=true, scale = 0.28]{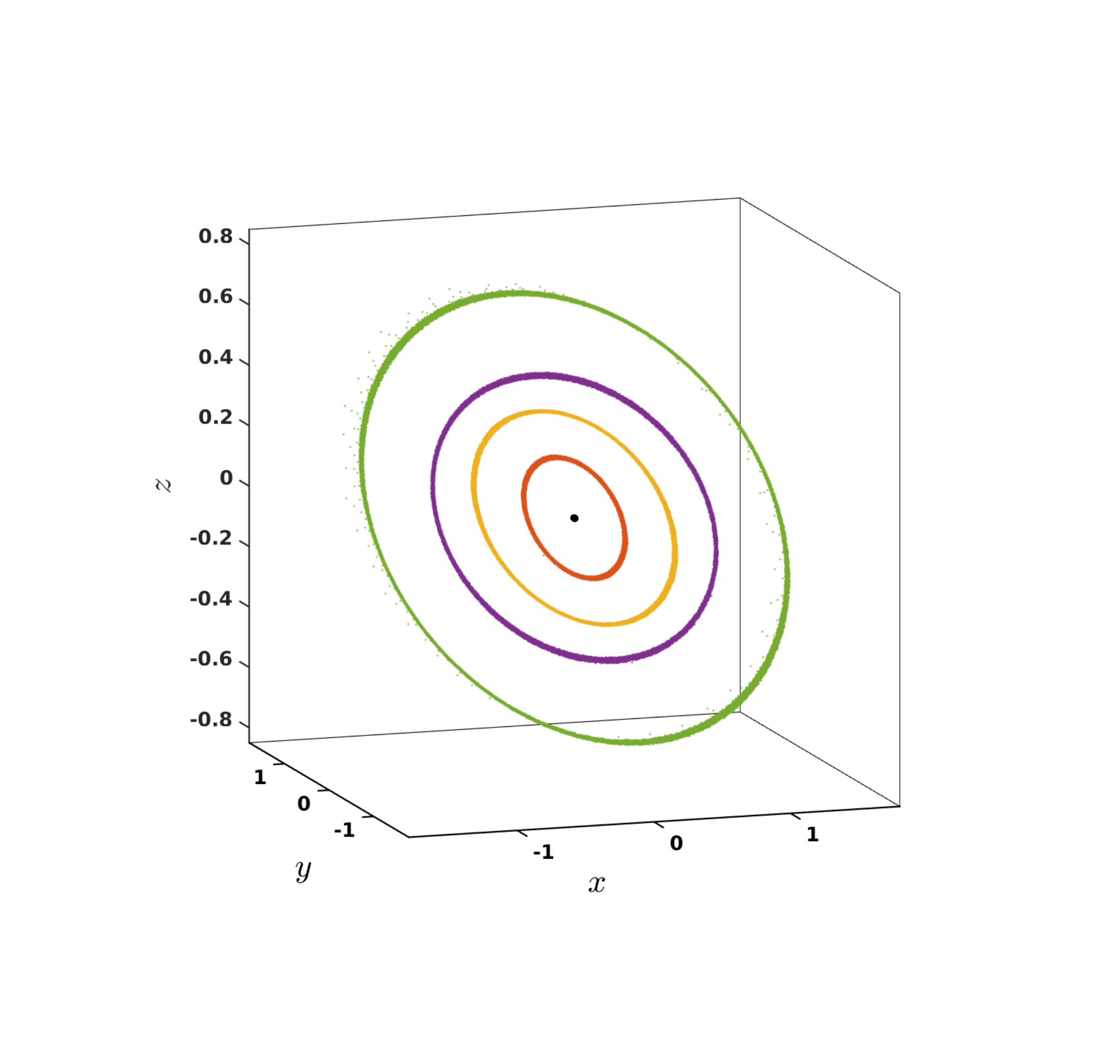} &
\includegraphics[angle=0, trim=80 80 130 80, clip=true, scale = 0.28]{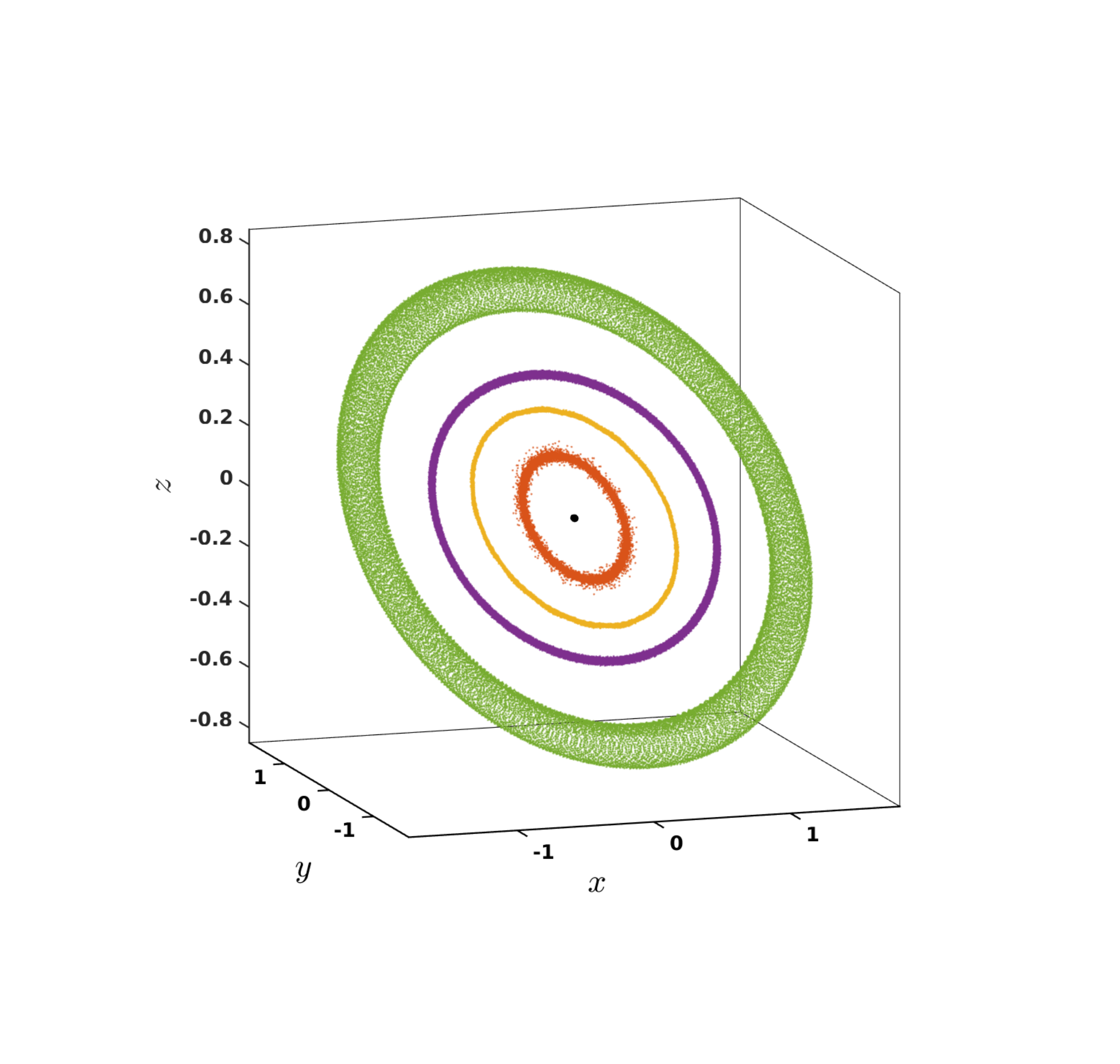} &
\includegraphics[angle=0, trim=80 80 130 80, clip=true, scale = 0.28]{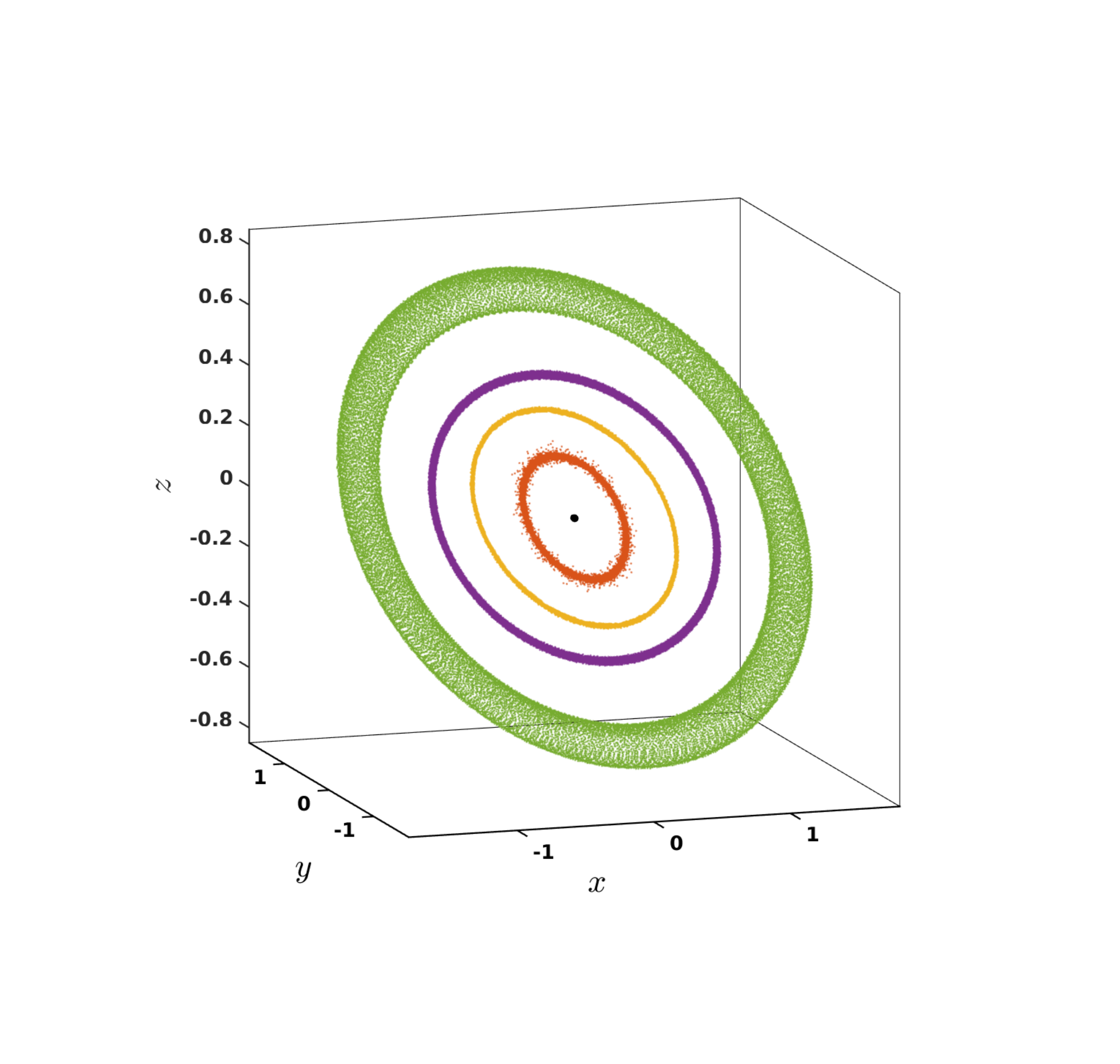} \\
$\bm f^{AB}_{\mathrm{ge}}$ and $\Delta t=5$ & $\bm f^{AB}_{\mathrm{pm}}$ and $\Delta t=5$ & $\bm f^{AB}_{\mathrm{pt}}$ and $\Delta t=5$ \\
\multicolumn{3}{c}{ \includegraphics[angle=0, trim=180 145 145 630, clip=true, scale = 0.65]{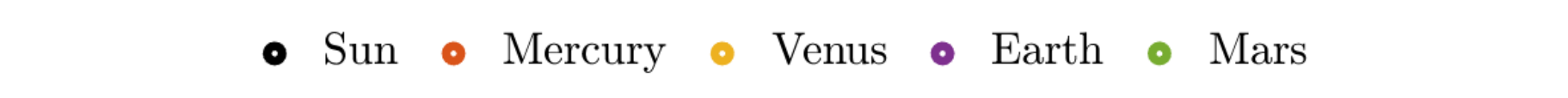} }
\end{tabular}
\caption{The trajectories of the inner planets calculated by different integrators. The visualization is made by sampling every $50$ earth days.} 
\label{fig:inner_planets_orbits}
\end{center}
\end{figure}

\begin{figure}
	\begin{center}
\begin{tabular}{ccc}
\includegraphics[angle=0, trim=80 80 120 80, clip=true, scale = 0.28]{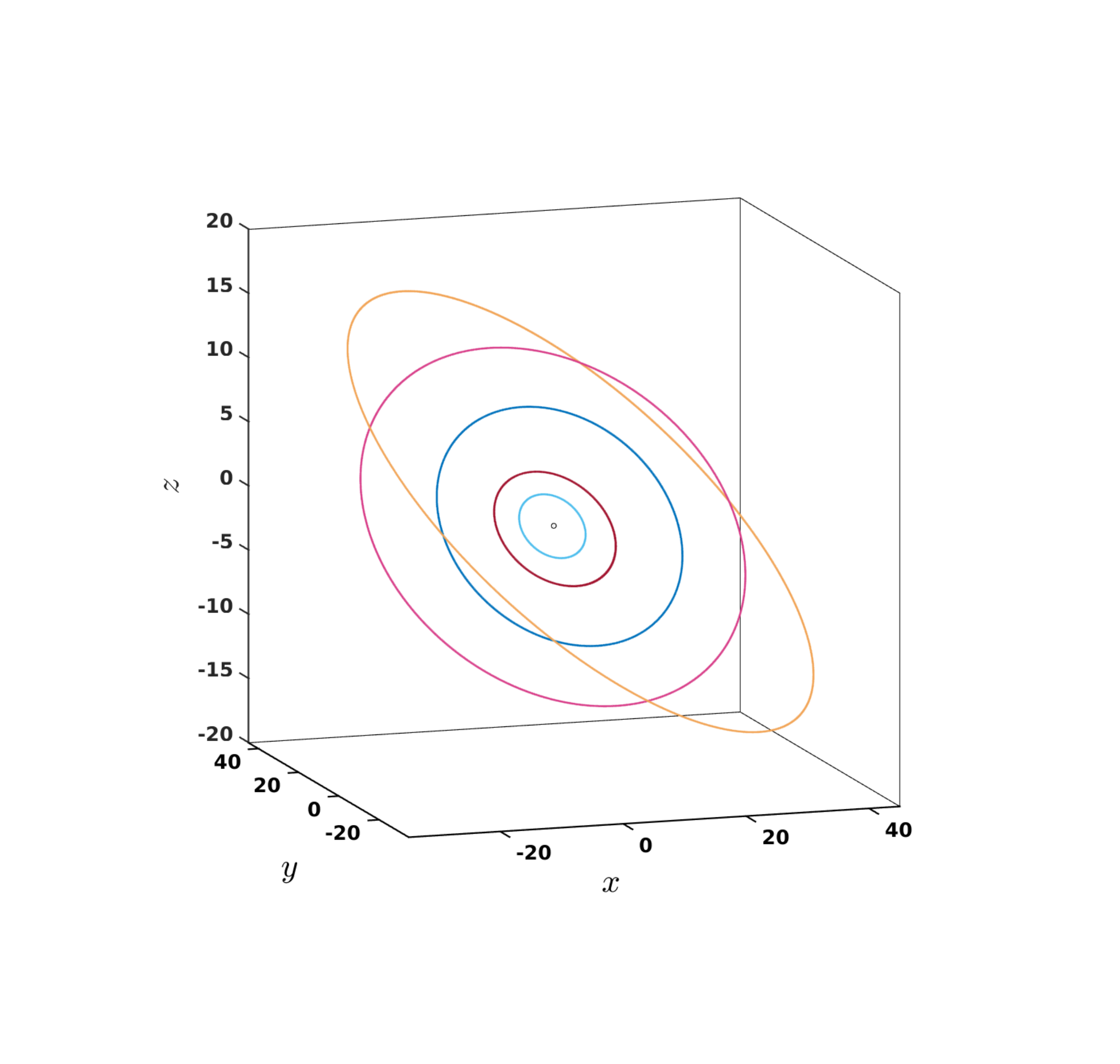} &
\includegraphics[angle=0, trim=80 80 120 80, clip=true, scale = 0.28]{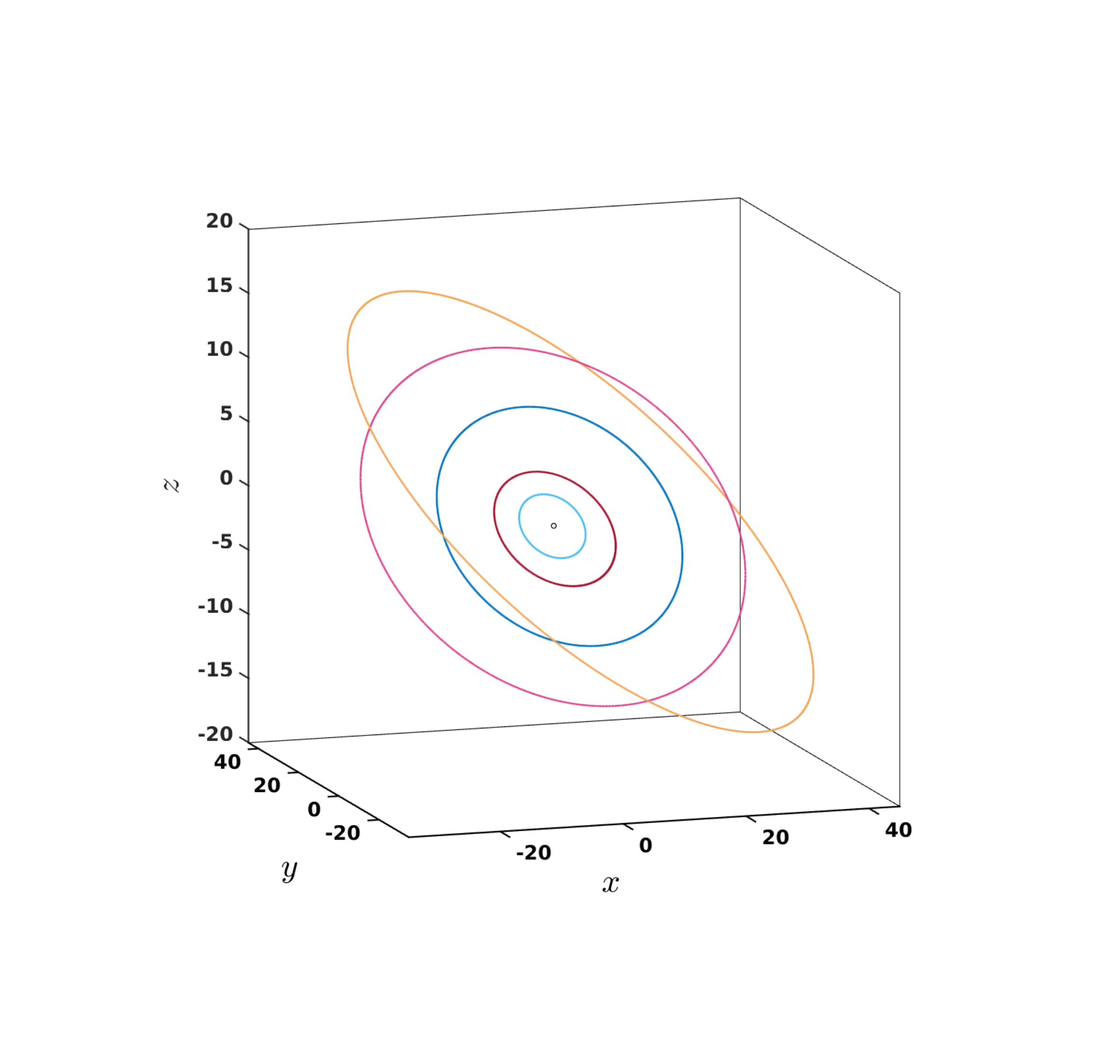} &
\includegraphics[angle=0, trim=80 80 120 80, clip=true, scale = 0.28]{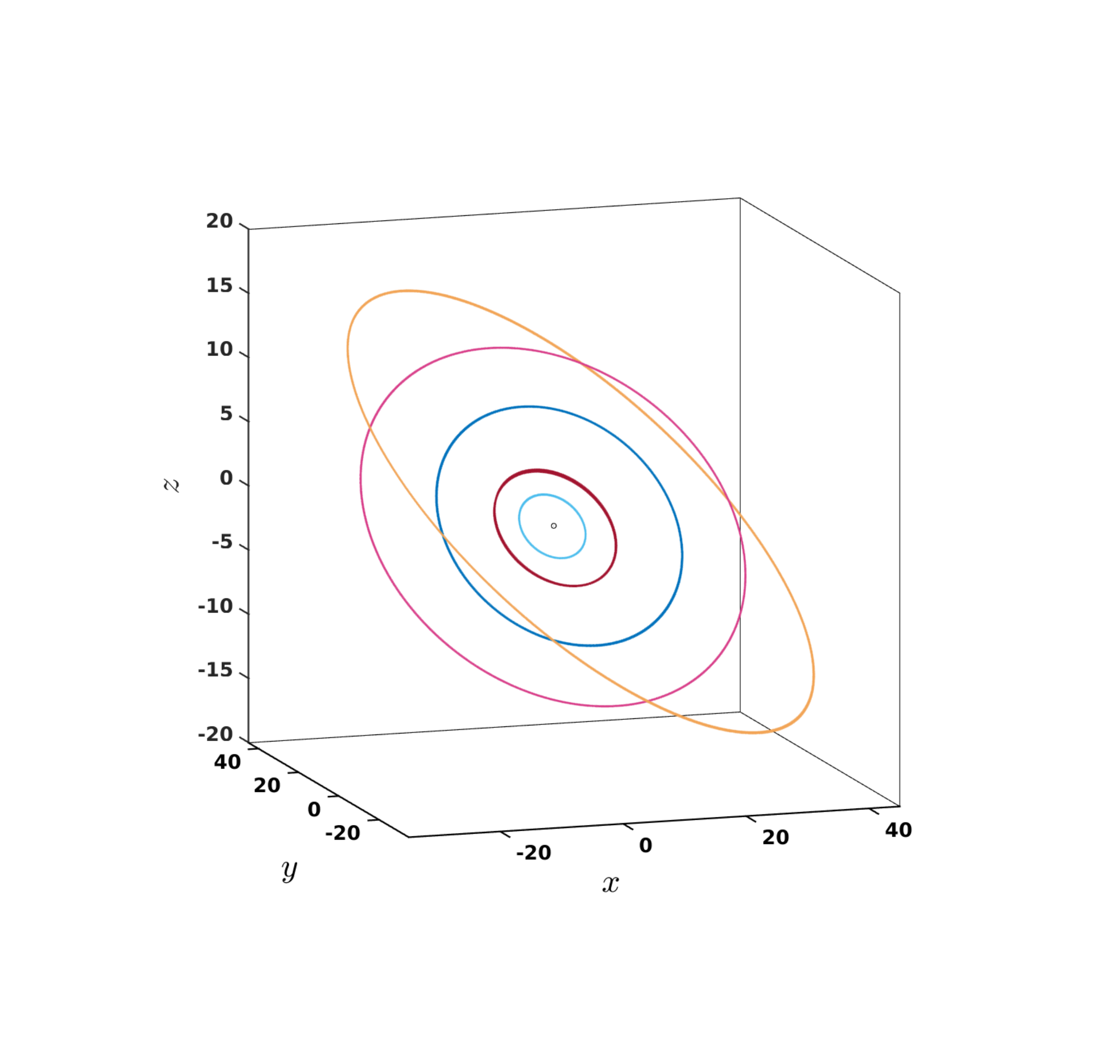} \\
$\bm f^{AB}_{\mathrm{mp}}$ and $\Delta t=0.1$ & $\bm f^{AB}_{\mathrm{mp}}$ and $\Delta t=5$ & $\bm f^{AB}_{\mathrm{lg}}$ and $\Delta t=5$ \\
\includegraphics[angle=0, trim=80 80 120 80, clip=true, scale = 0.28]{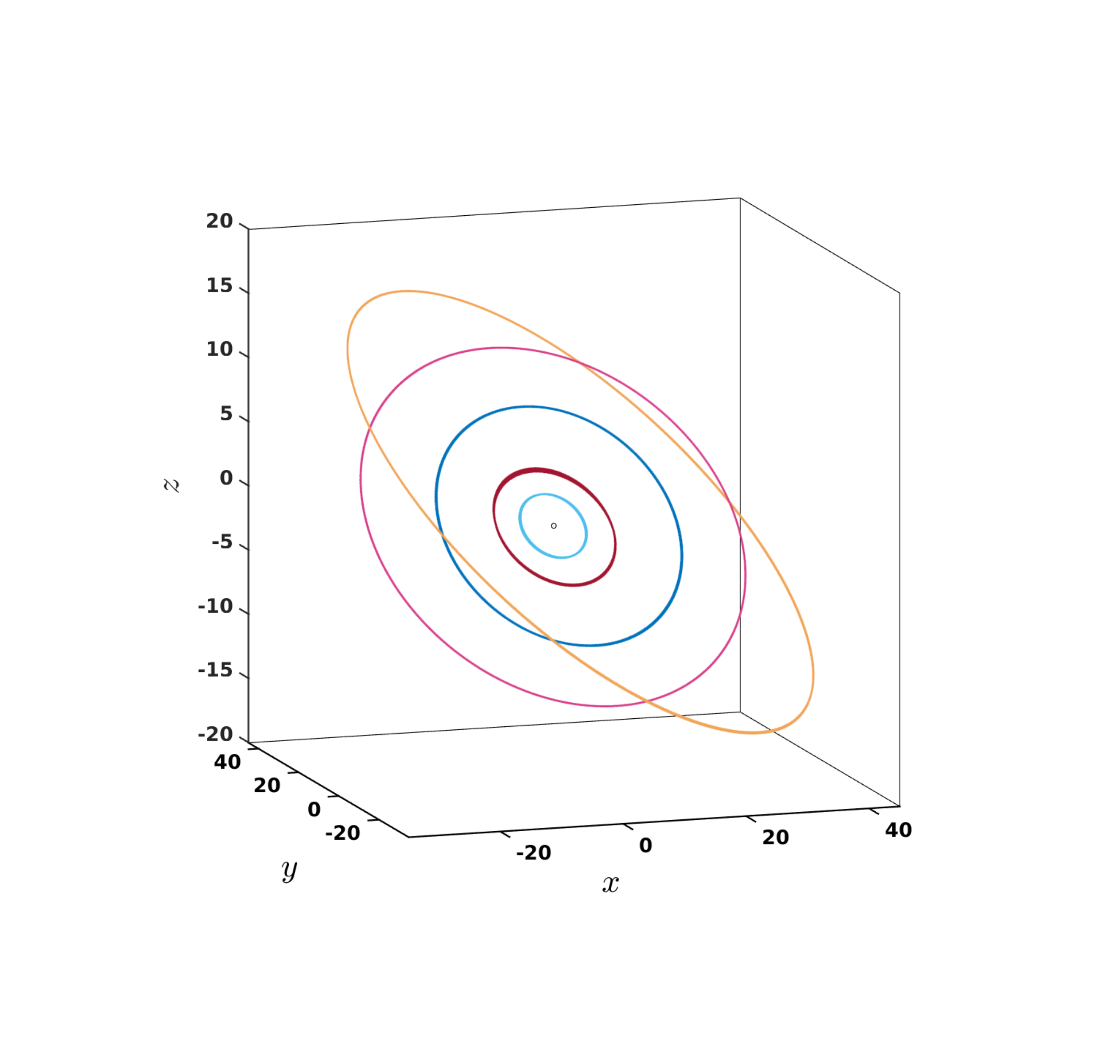} &
\includegraphics[angle=0, trim=80 80 120 80, clip=true, scale = 0.28]{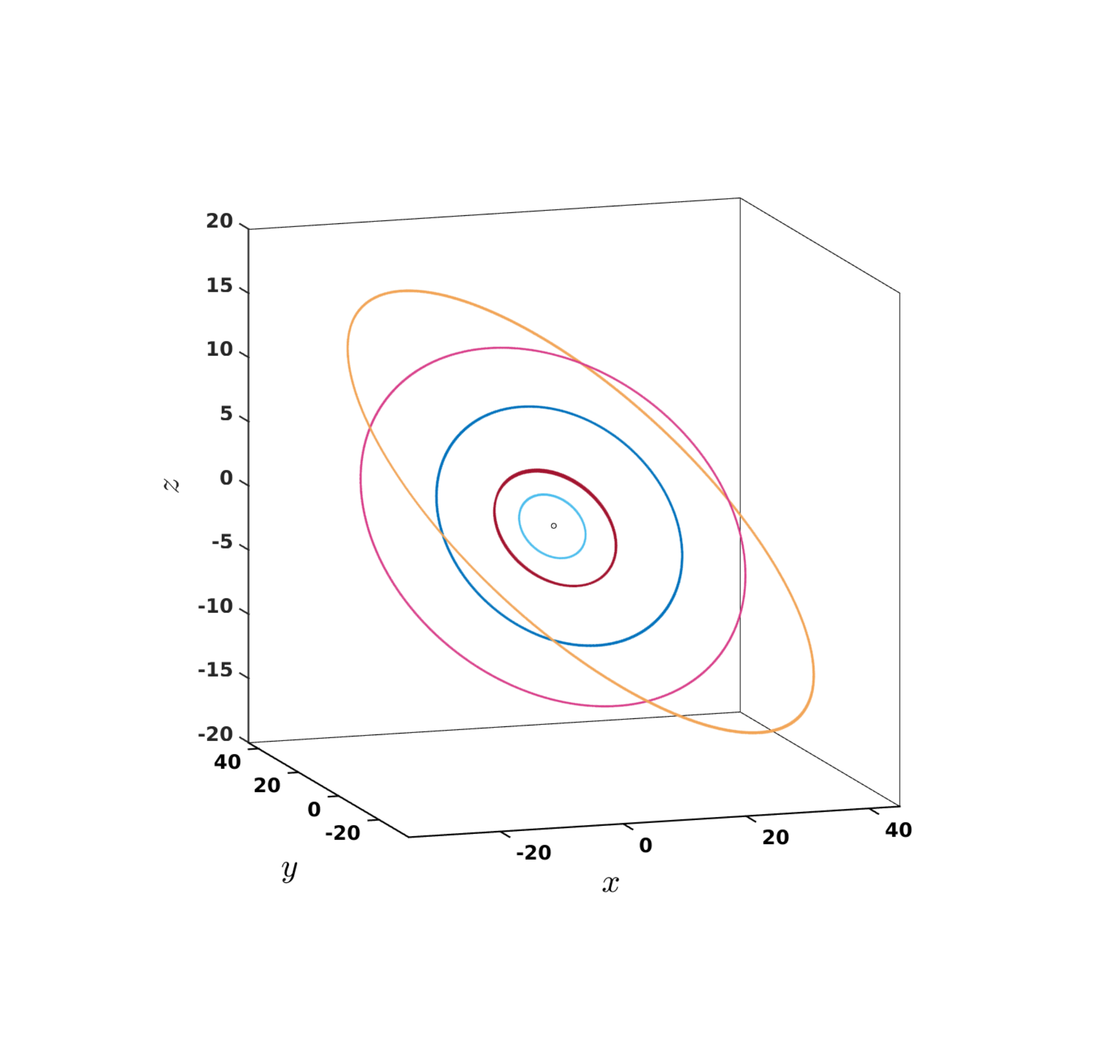} &
\includegraphics[angle=0, trim=80 80 120 80, clip=true, scale = 0.28]{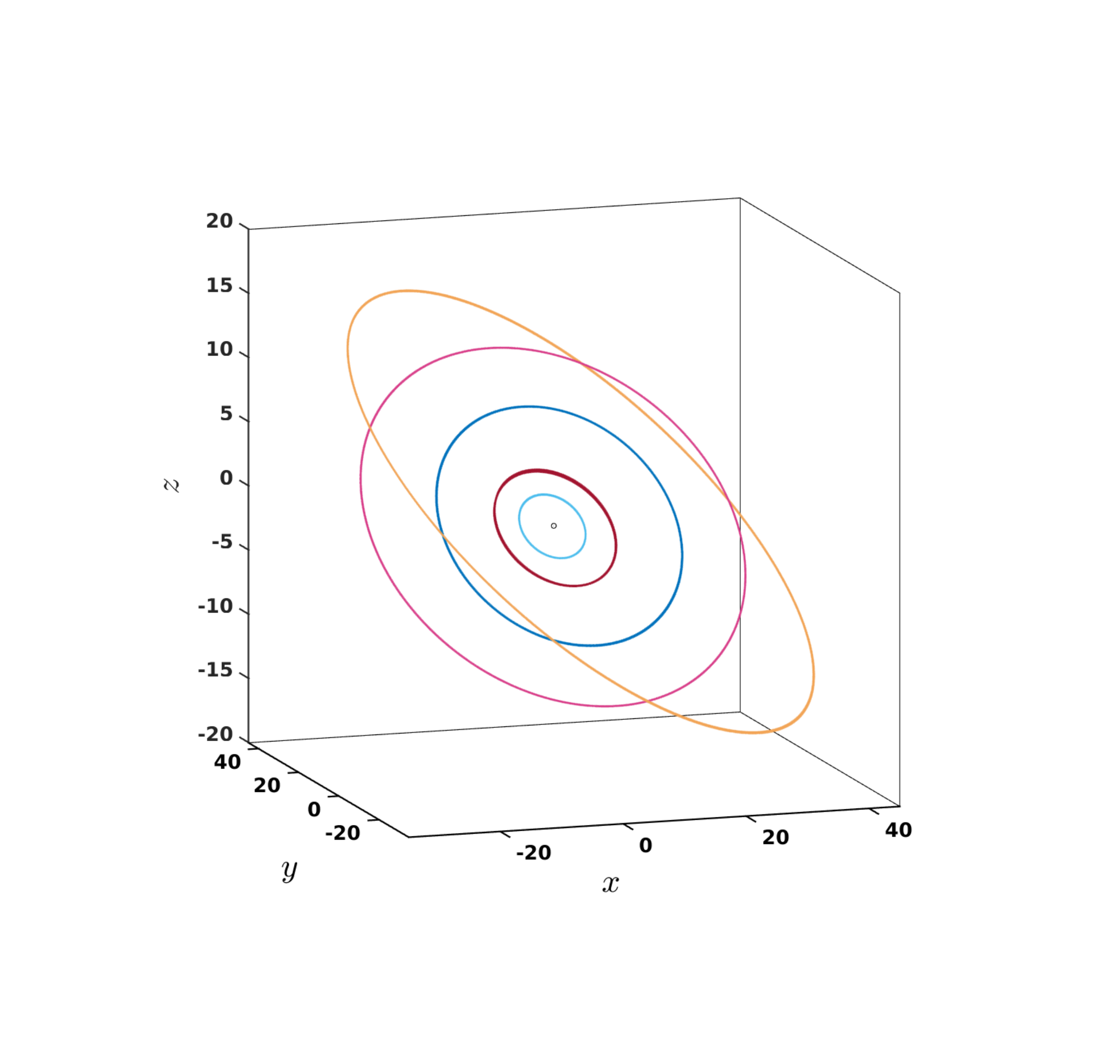} \\
$\bm f^{AB}_{\mathrm{ge}}$ and $\Delta t=5$ & $\bm f^{AB}_{\mathrm{pm}}$ and $\Delta t=5$ & $\bm f^{AB}_{\mathrm{pt}}$ and $\Delta t=5$ \\
\multicolumn{3}{c}{ \includegraphics[angle=0, trim=180 145 145 630, clip=true, scale = 0.65]{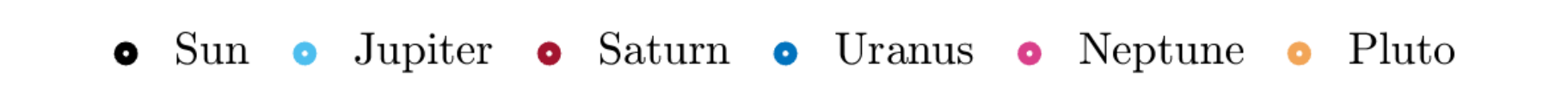} }
\end{tabular}
\caption{The trajectories of the outer planets calculated by different integrators. The visualization is made by sampling every $50$ earth days.} 
\label{fig:outer_planets_orbits}
\end{center}
\end{figure}

The trajectories of the inner and outer planets are illustrated in Figures \ref{fig:inner_planets_orbits} and \ref{fig:outer_planets_orbits}. While the simulated trajectories of the five outer planets show less differences among the integrators, there are dramatic differences in the simulated orbits for the four inner planets. The orbits of Mercury and Mars exhibit strong variations in the mid-point and LaBudde-Greenspan integrators by comparing them with the reference solution. The orbits calculated from the two second-order energy-decaying integrators are slightly sharper than the aforementioned two integrators. The results of the generalized Eyre integrator show the closest pattern to those of the reference solution, from which we can see the oscillatory modes are damped in this integrator. 

\begin{figure}
\begin{center}
\begin{tabular}{cc}
\includegraphics[angle=0, trim=80 80 120 100, clip=true, scale = 0.28]{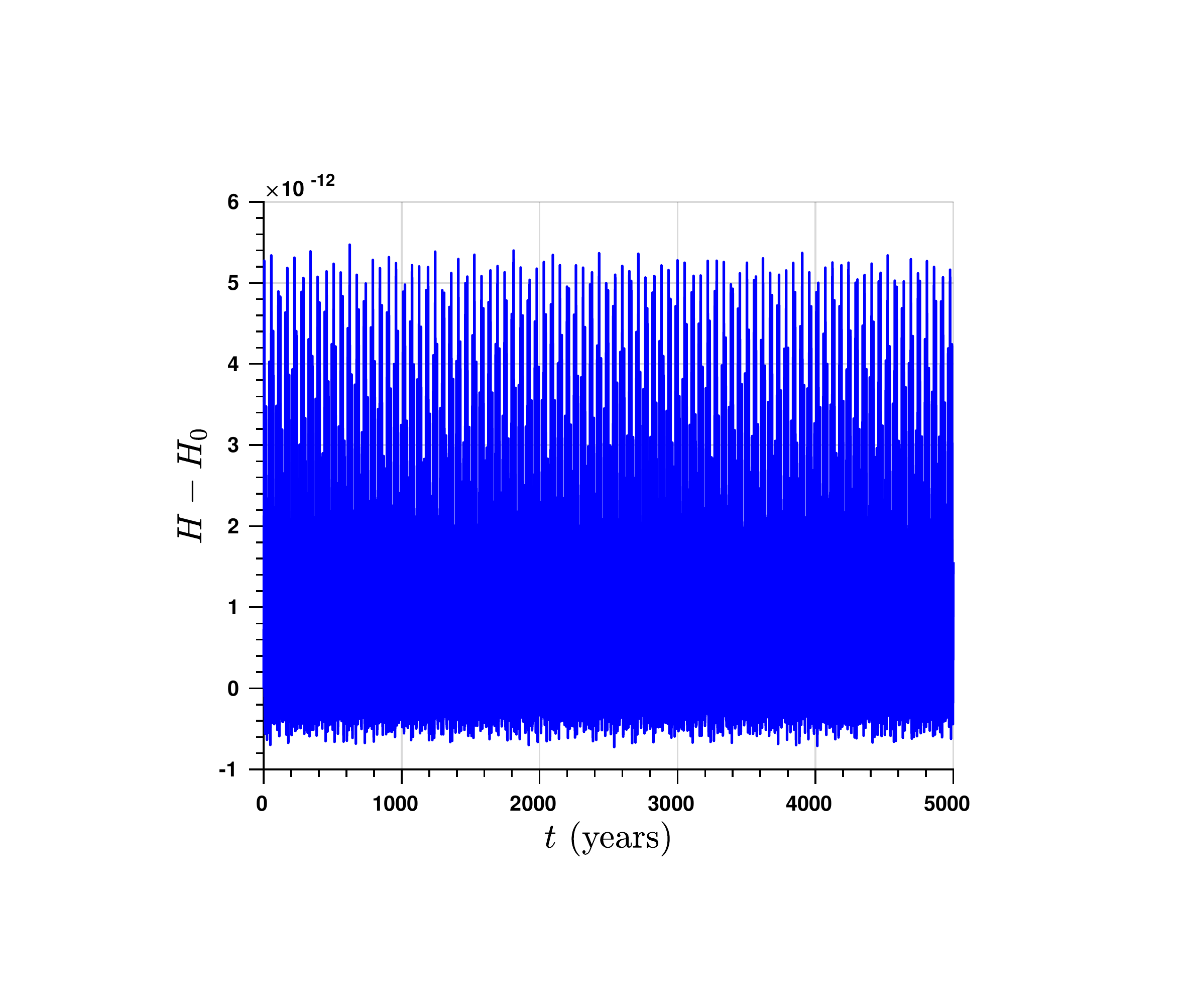} &\includegraphics[angle=0, trim=80 80 120 100, clip=true, scale = 0.28]{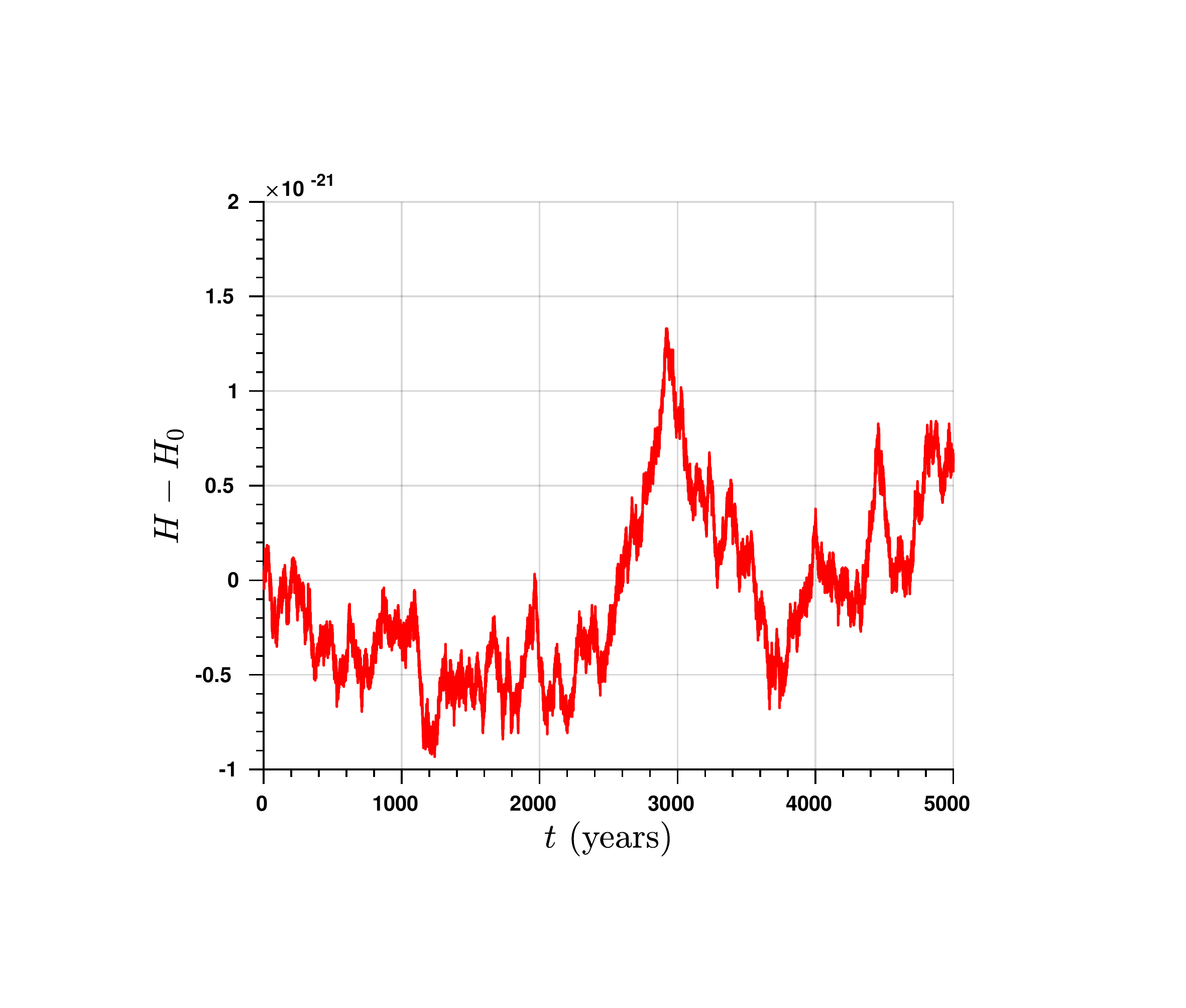} \\
(a) & (b) 
\end{tabular}
\begin{tabular}{ccc}
\includegraphics[angle=0, trim=80 80 120 100, clip=true, scale = 0.28]{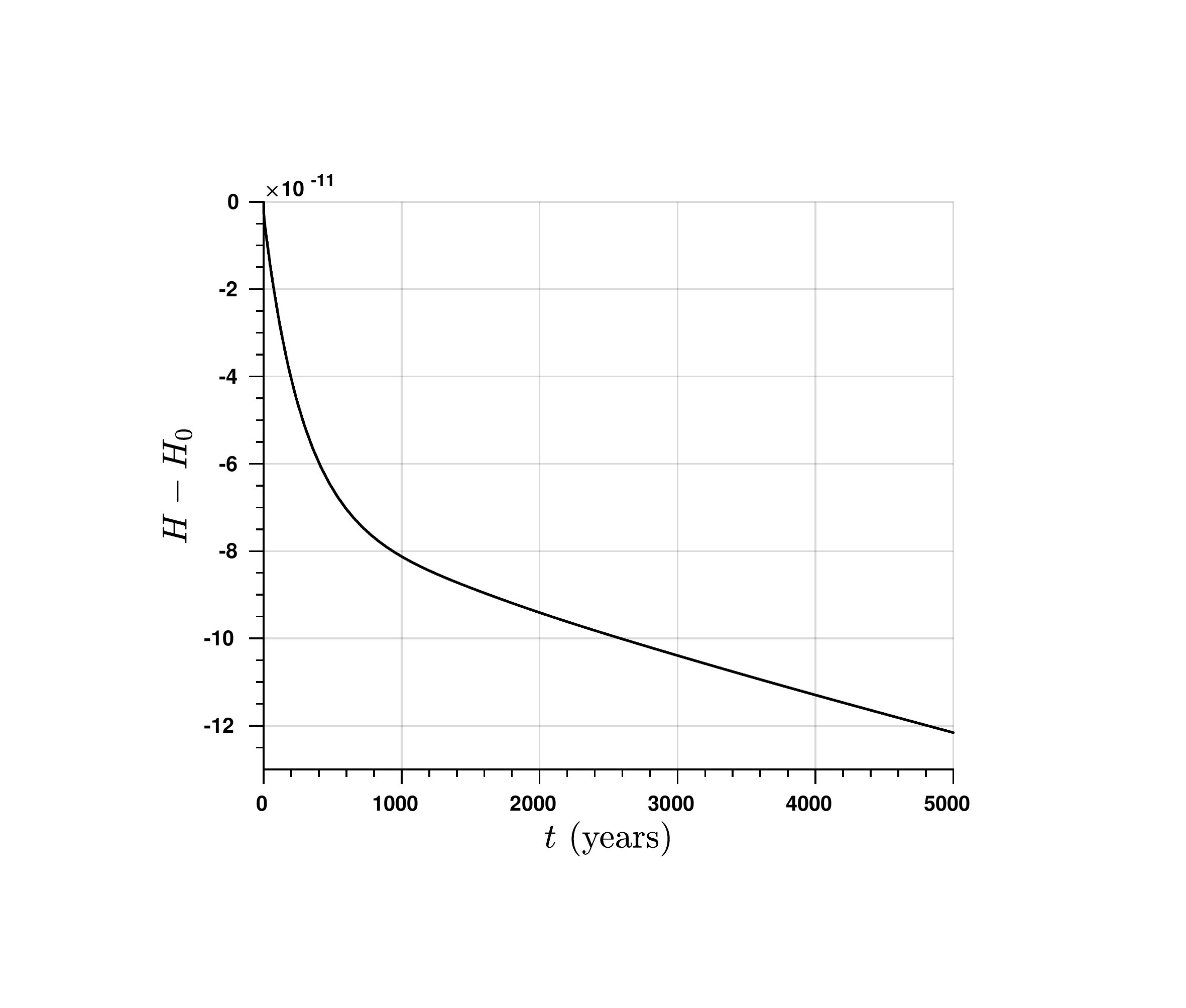} &
\includegraphics[angle=0, trim=80 80 120 100, clip=true, scale = 0.28]{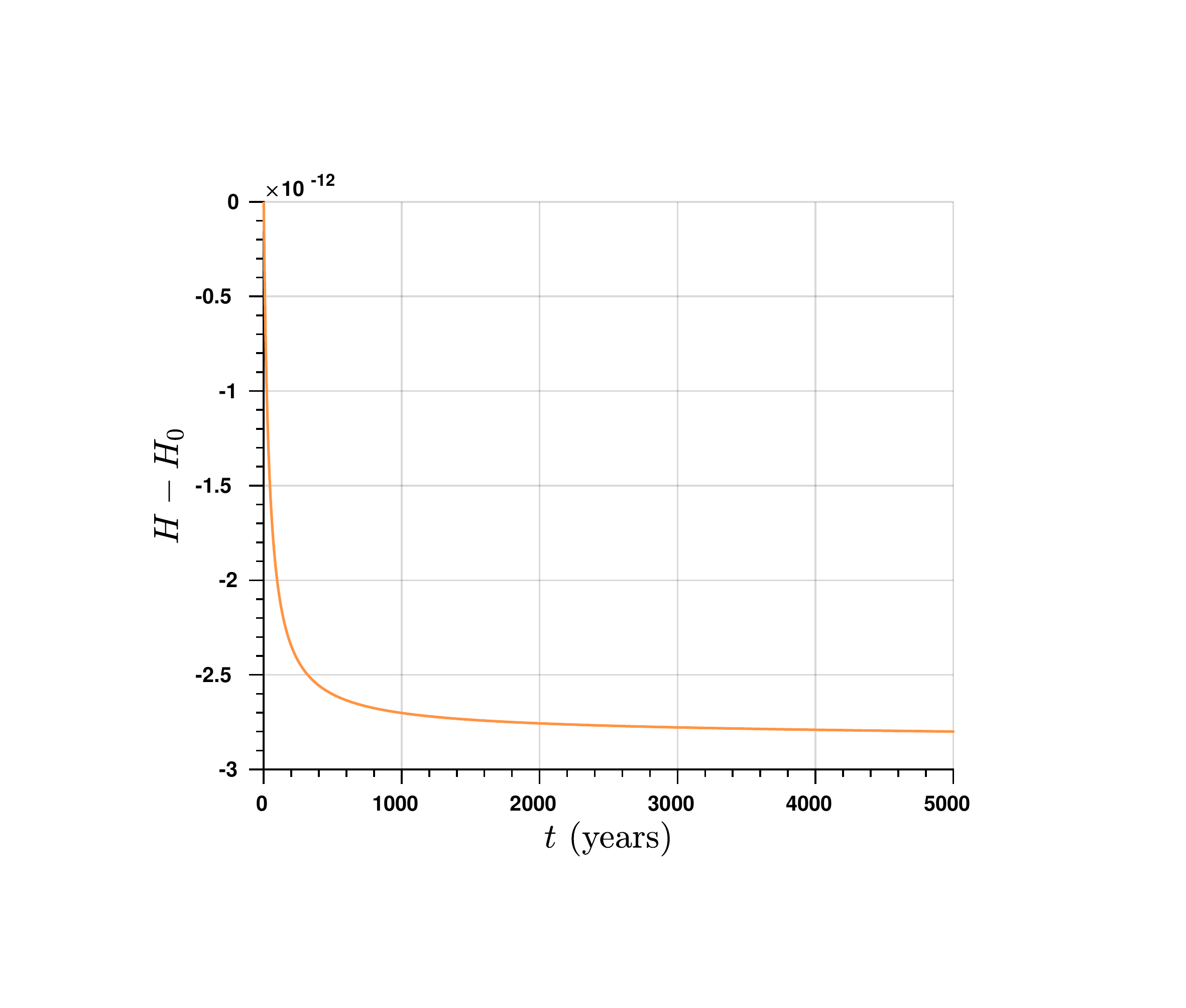} &
\includegraphics[angle=0, trim=80 80 120 100, clip=true, scale = 0.28]{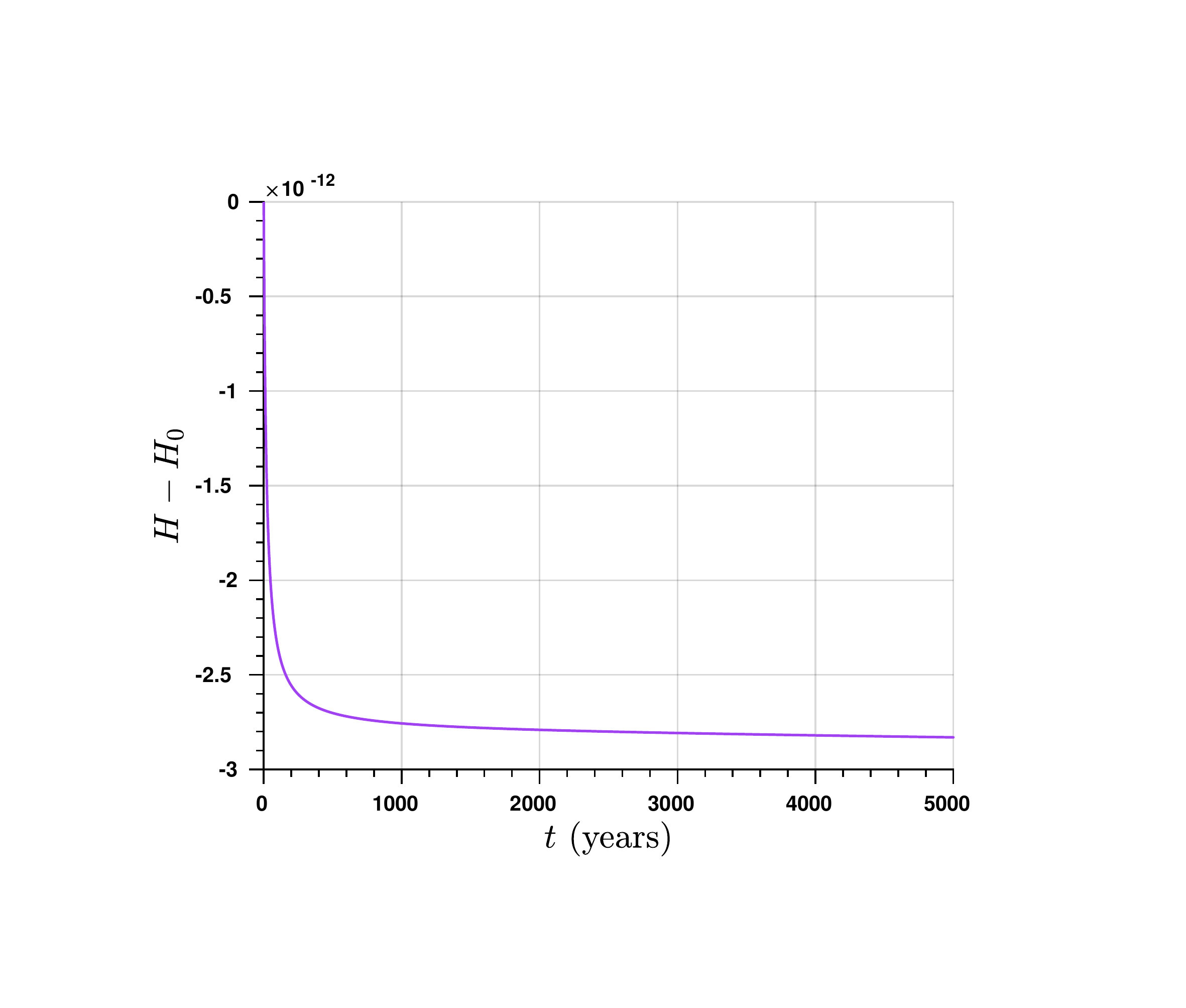} \\
(c) & (d) & (e)
\end{tabular}
\caption{The error of the total energy $H(\bm q_n, \bm p_n) - H(\bm q_0, \bm p_0)$ of the ten-body problem over time calculated by (a) the mid-point, (b) LaBudde-Greenspan, (c) generalized Eyre, (d) perturbed mid-point, and (e) perturbed trapezoidal integrators. } 
\label{fig:solar_energy}
\end{center}
\end{figure}

\begin{figure}
	\begin{center}
\begin{tabular}{ccc}
\multicolumn{3}{c}{ \includegraphics[angle=0, trim=780 265 690 1730, clip=true, scale = 0.45]{./momentum_legend} }\\
\includegraphics[angle=0, trim=80 80 120 100, clip=true, scale = 0.28]{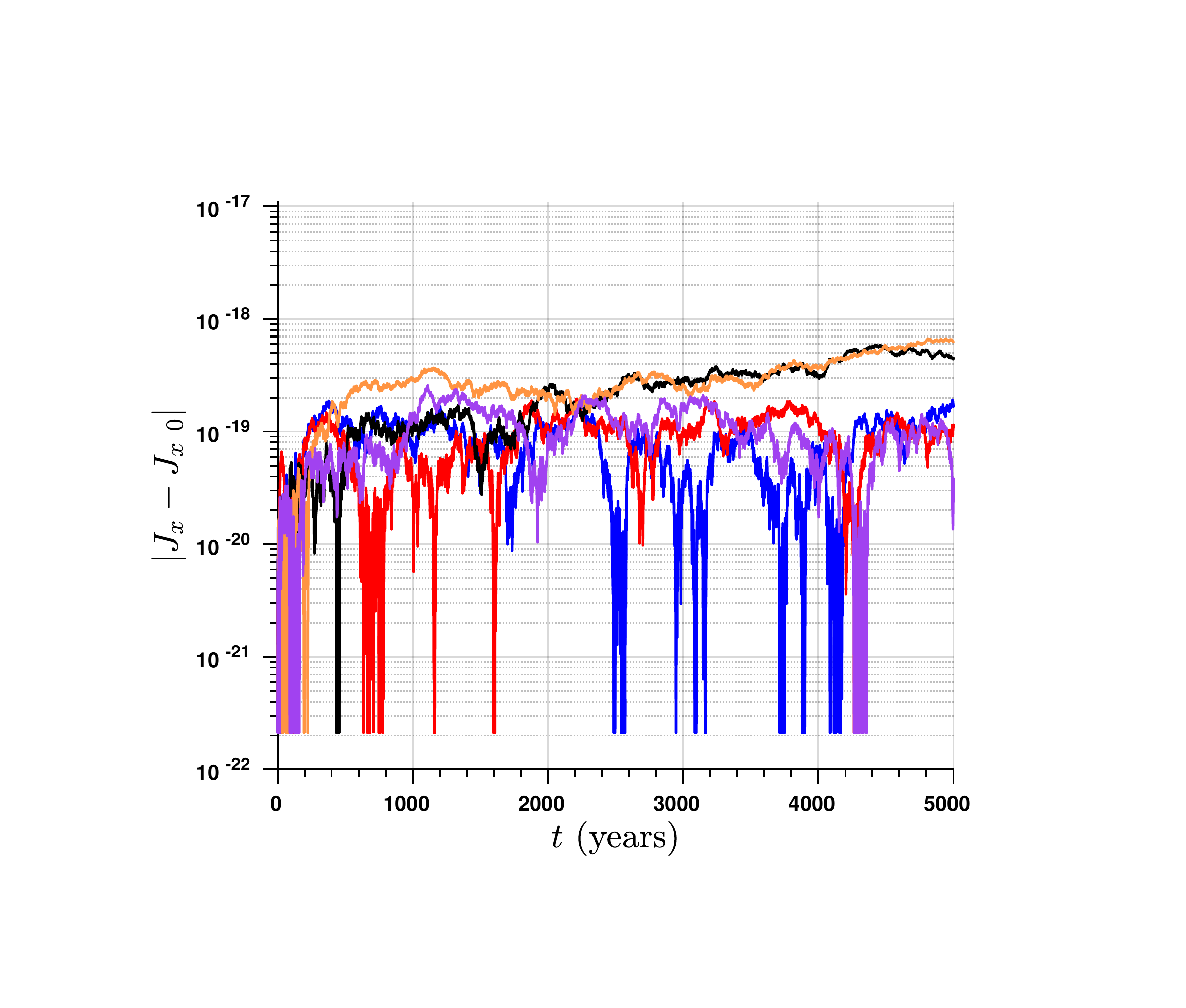} &
\includegraphics[angle=0, trim=80 80 120 100, clip=true, scale = 0.28]{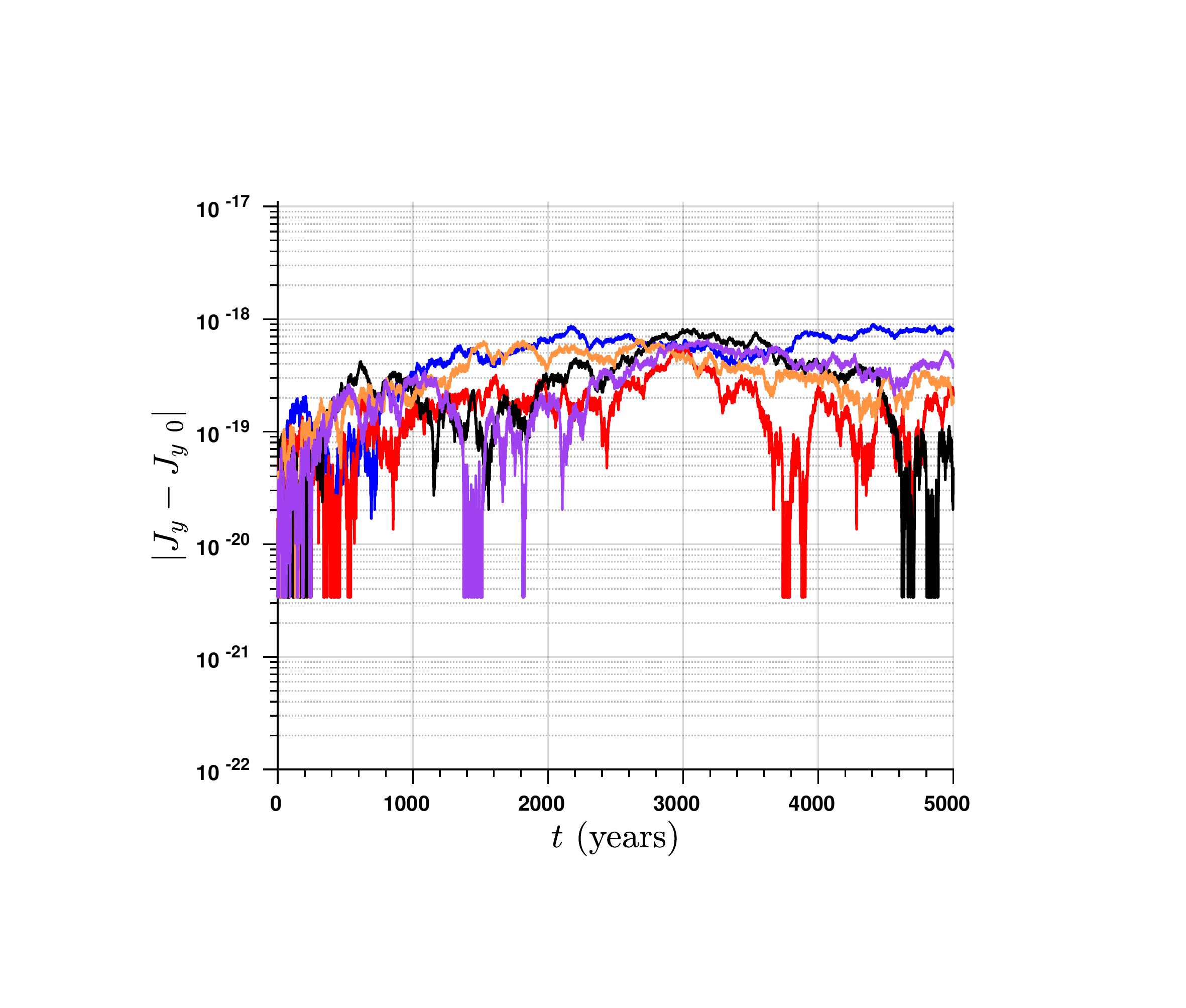} &
\includegraphics[angle=0, trim=80 80 120 100, clip=true, scale = 0.28]{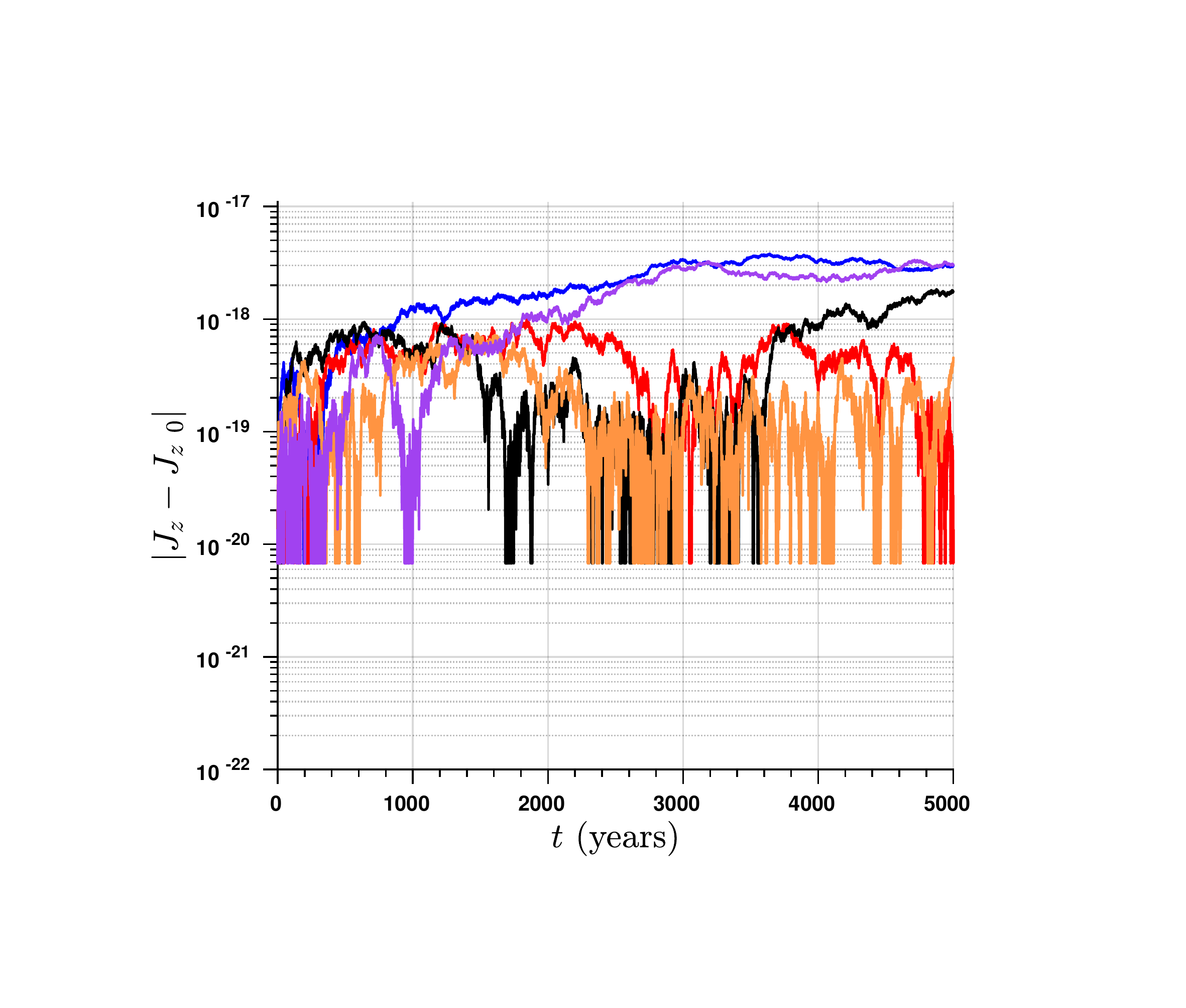} \\
(a) & (b) & (c) \\
\includegraphics[angle=0, trim=80 80 120 100, clip=true, scale = 0.28]{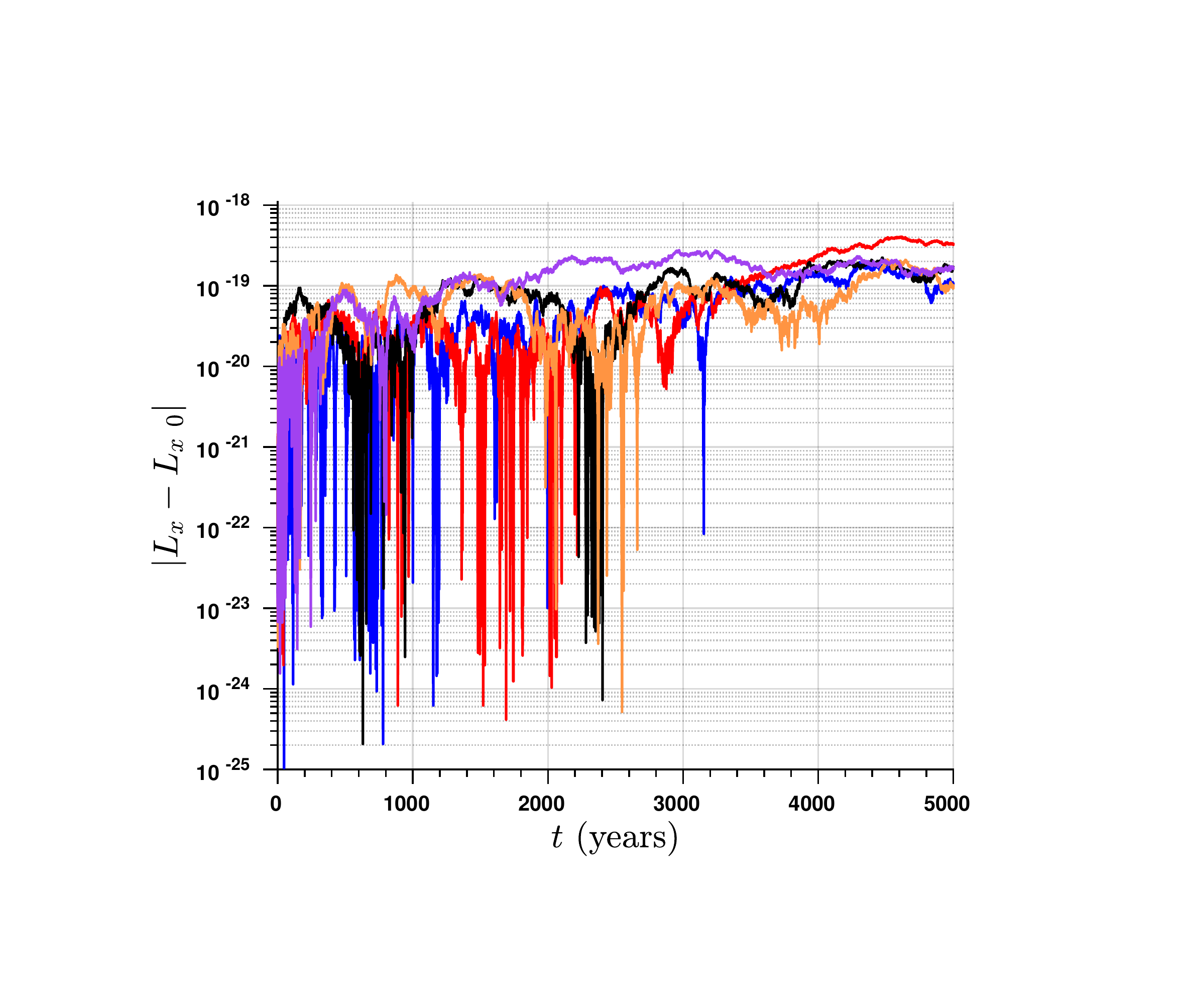} &
\includegraphics[angle=0, trim=80 80 120 100, clip=true, scale = 0.28]{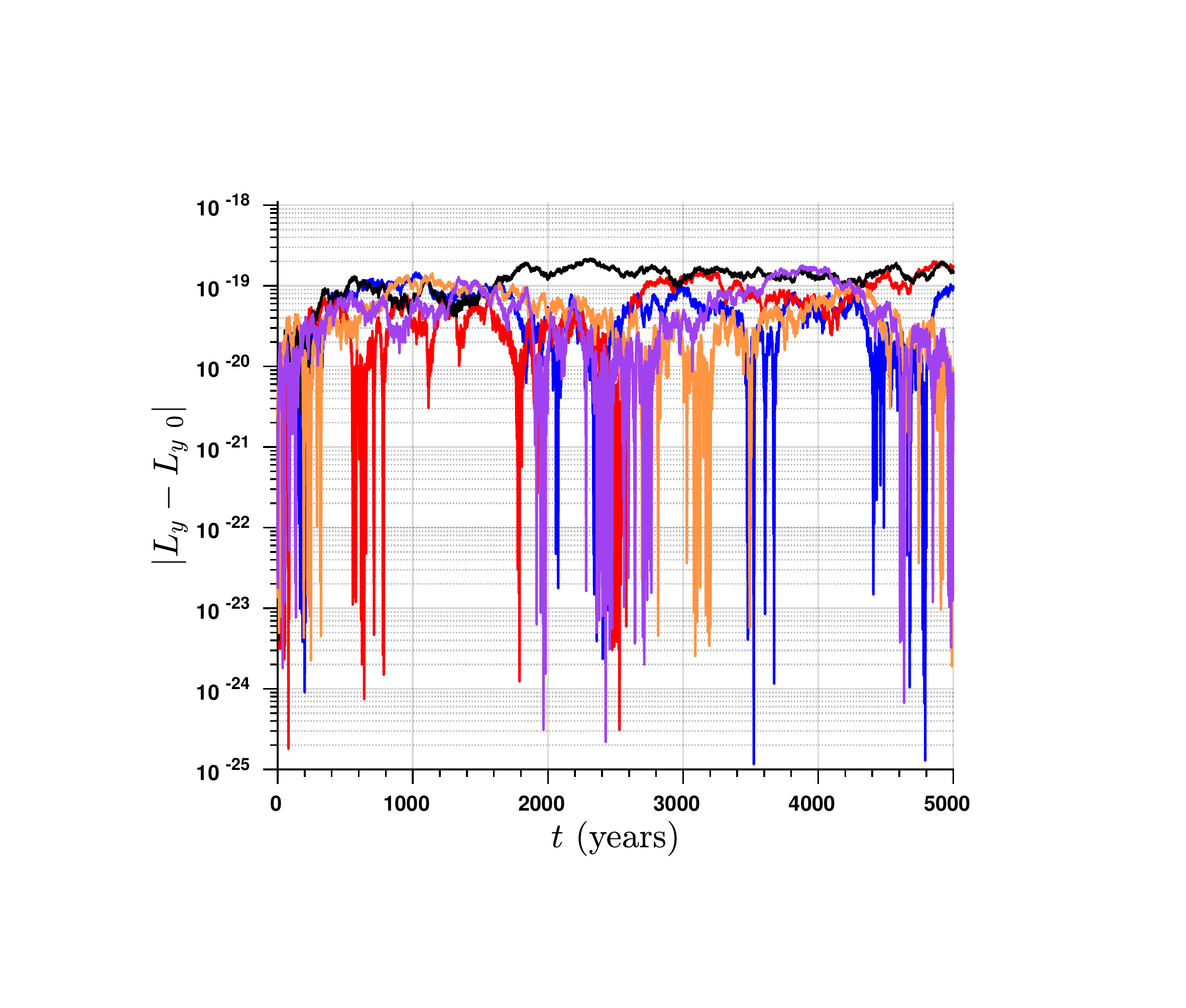} &
\includegraphics[angle=0, trim=80 80 120 100, clip=true, scale = 0.28]{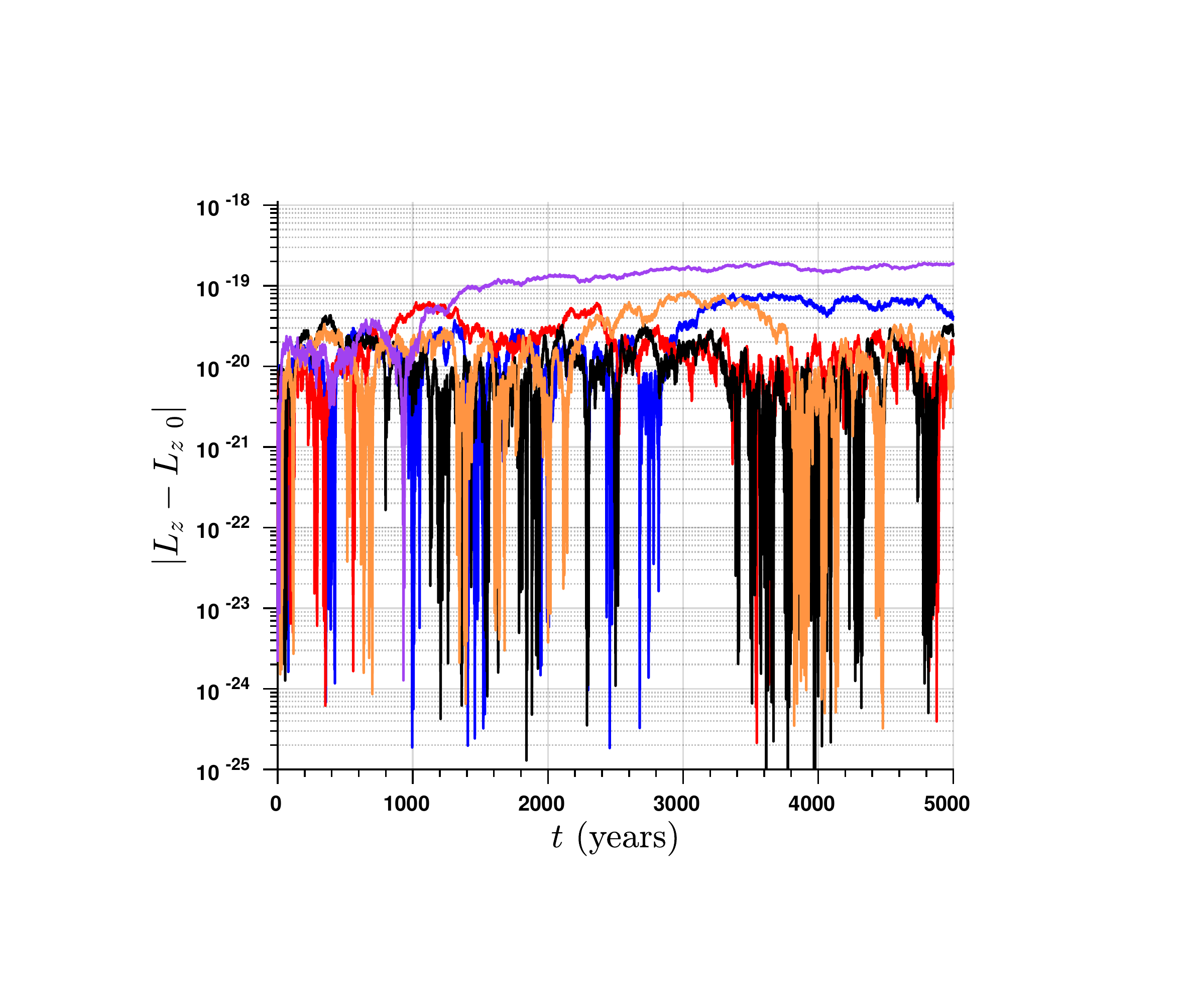} \\
(d) & (e) & (f) \\
\includegraphics[angle=0, trim=80 80 120 100, clip=true, scale = 0.28]{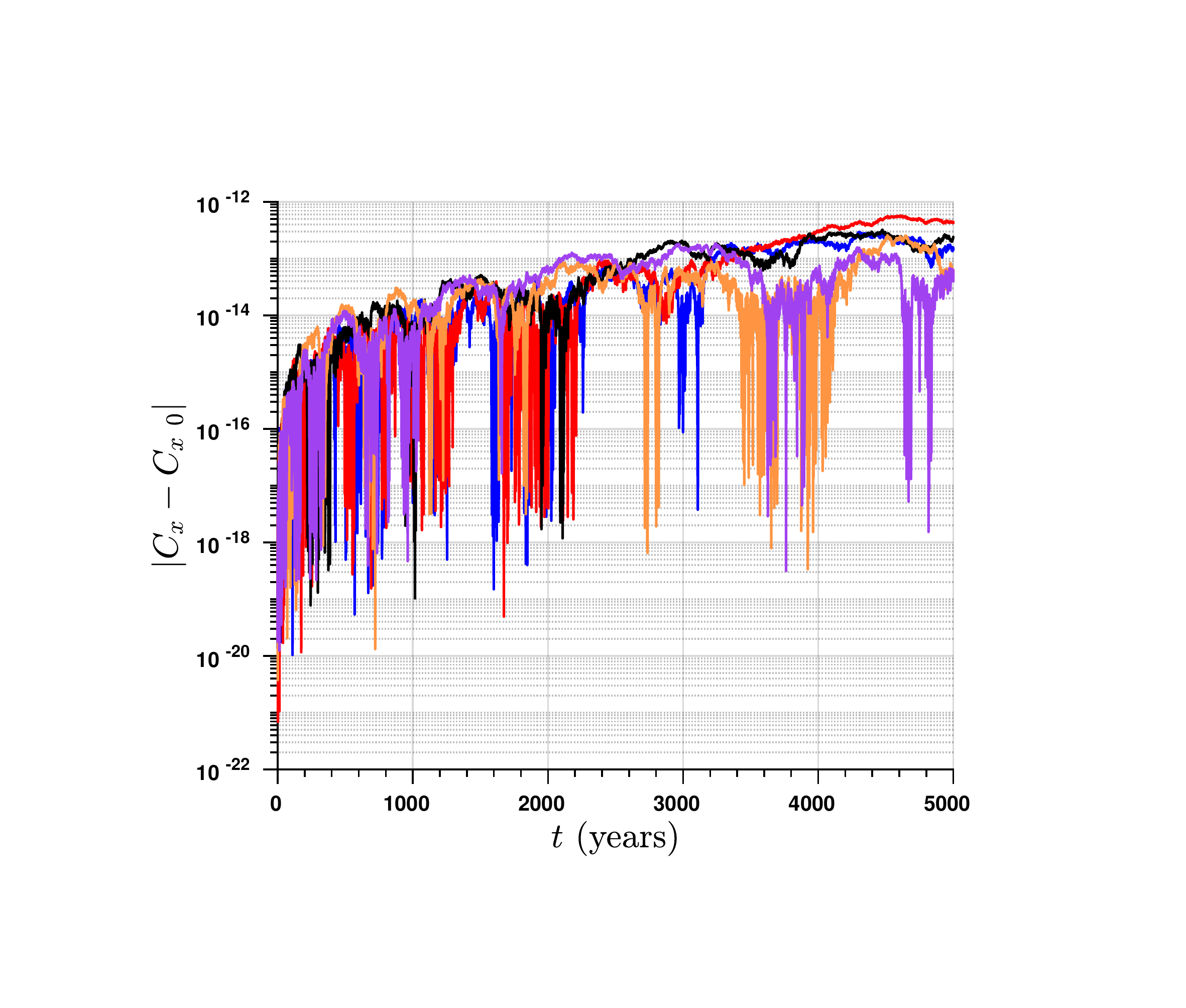} &
\includegraphics[angle=0, trim=80 80 120 100, clip=true, scale = 0.28]{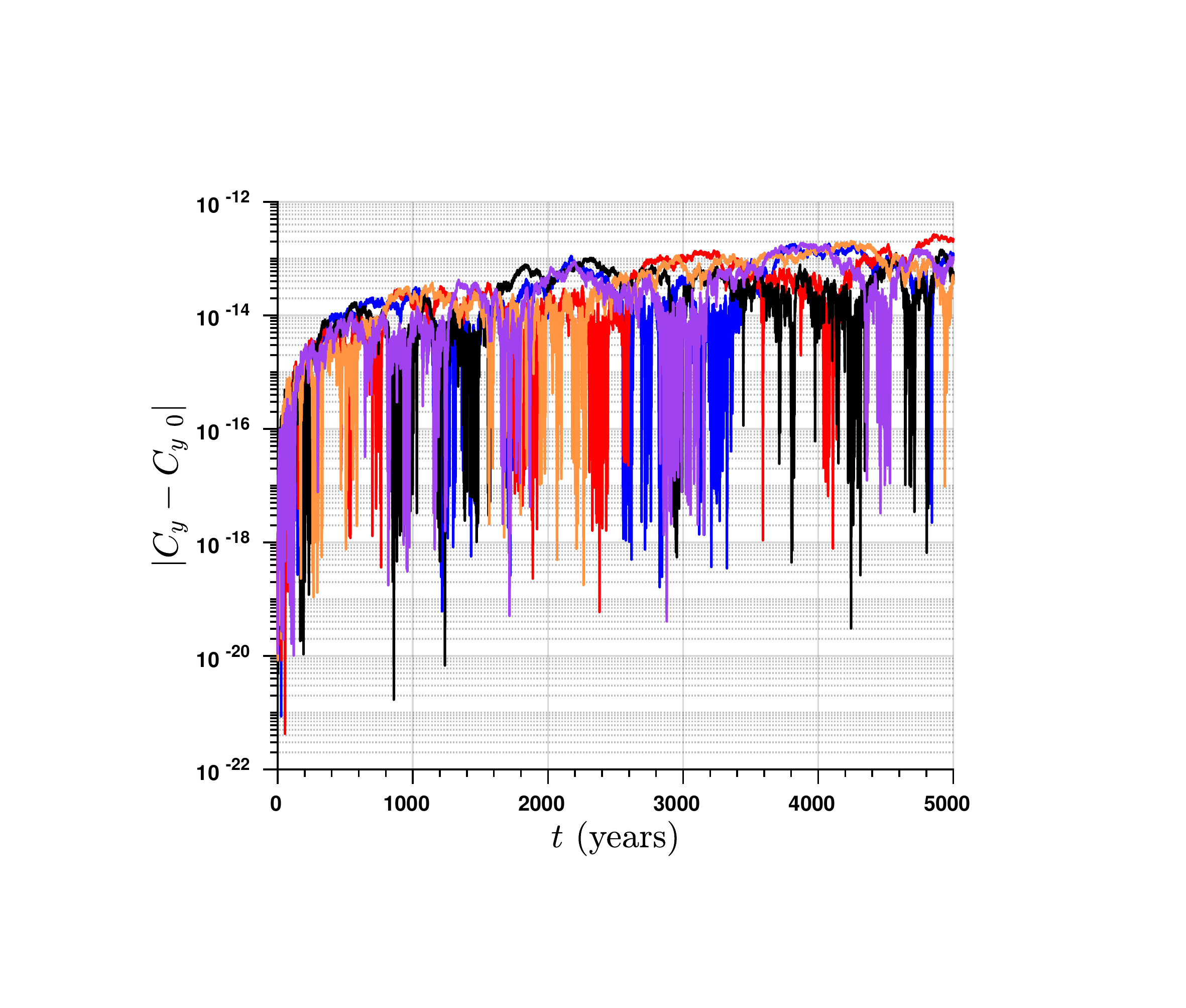} &
\includegraphics[angle=0, trim=80 80 120 100, clip=true, scale = 0.28]{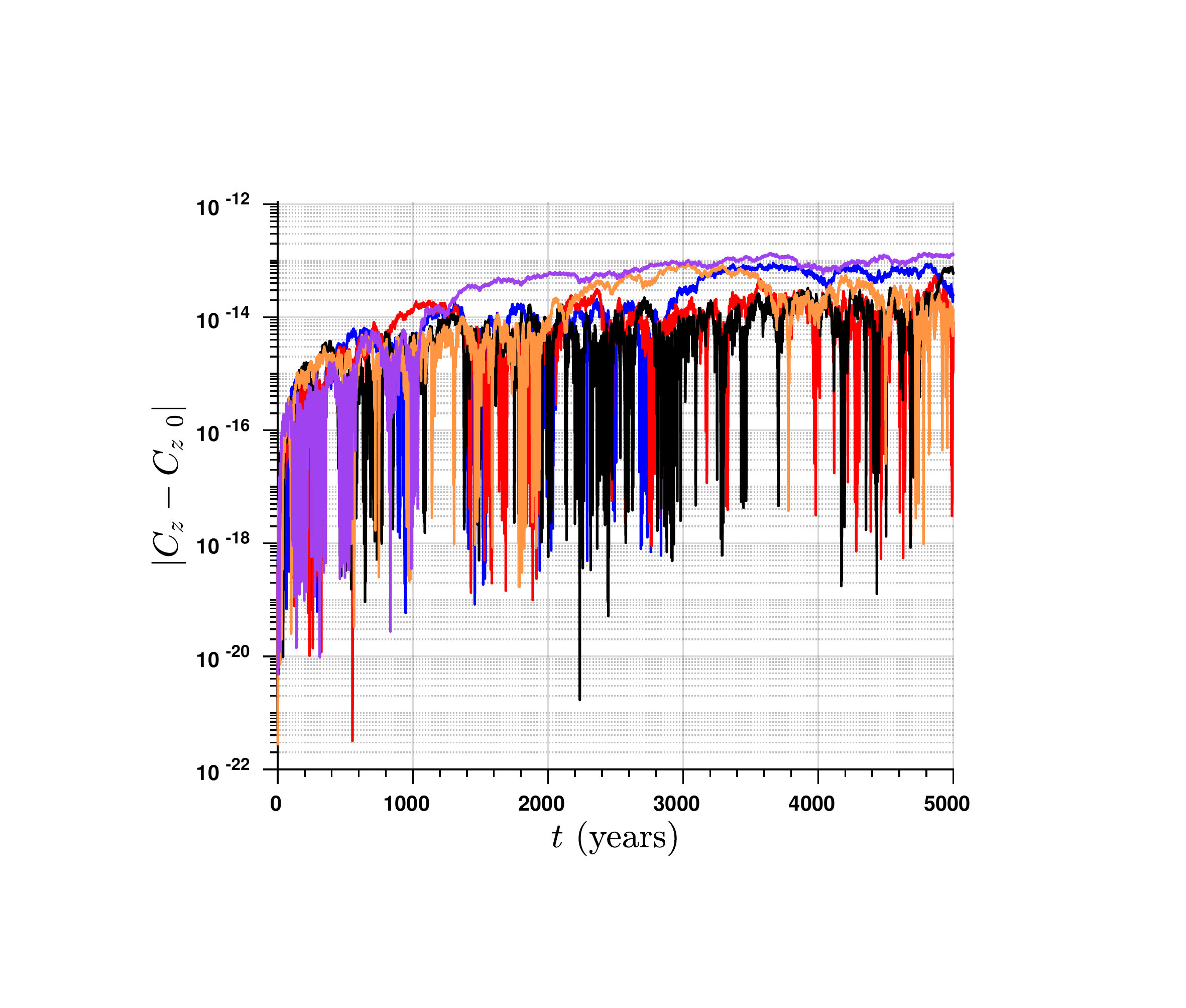} \\
(h) & (i) & (j)
\end{tabular}
\caption{The errors of the angular momentum (first row), linear momentum (second row), and center of mass (third row) of the ten-body problem over time calculated the mid-point, LaBudde-Greenspan, generalized Eyre, perturbed mid-point, and perturbed trapezoidal integrators.} 
\label{fig:solar_momentum}
\end{center}
\end{figure}

Looking at the energy error depicted in Figure \ref{fig:solar_energy}, we can see that the LaBudde-Greenspan integrator exhibits the best energy preservation property, with the energy error several orders of magnitude smaller than those of the rest four integrators. Also, the energy here does not exhibit the unfavored growth presented in the previous two-body example, due to the adoption of non-quotient formula \eqref{eq:solar_flg_2}. The energy errors of the three proposed integrators are comparable in this case, and they all tend to reach a limiting value after a certain simulated time. The limit is set by the conserved momenta in the integrators. Inspecting the momenta and center of mass, the five considered integrators all exhibit similar performances in terms of invariant preservation. The errors of $\bm C$ is a few orders of magnitude larger than those of $\bm J$ and $\bm L$. As was explained in \cite{Wan2022}, this is due to its explicit dependence on time, which has an amplifying effect on the error. 

\begin{figure}
	\begin{center}
\begin{tabular}{|c|c|c|}
\hline
\includegraphics[angle=0, trim=80 80 120 80, clip=true, scale = 0.28]{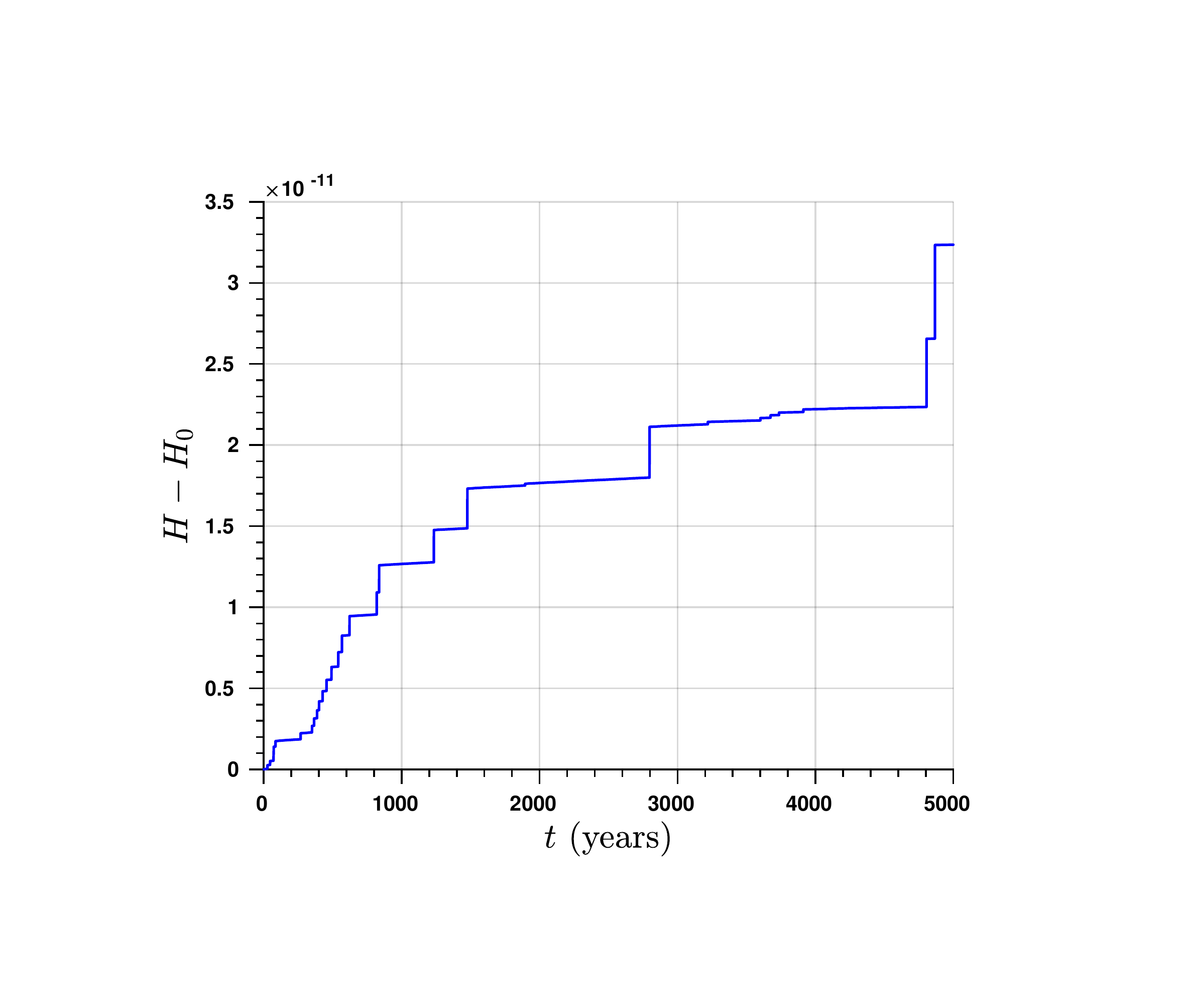} &
\includegraphics[angle=0, trim=80 80 120 80, clip=true, scale = 0.28]{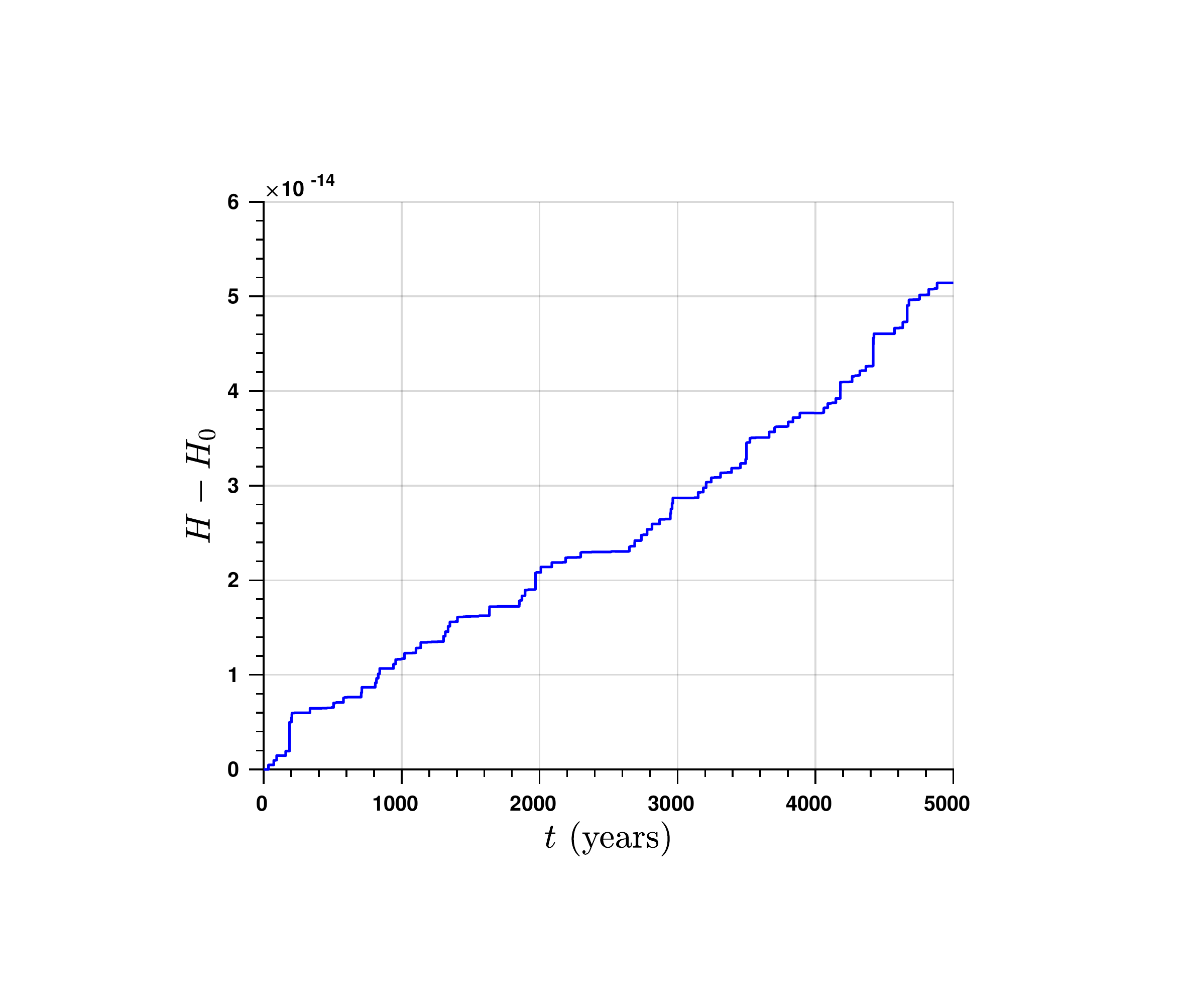} &
\includegraphics[angle=0, trim=80 80 120 80, clip=true, scale = 0.28]{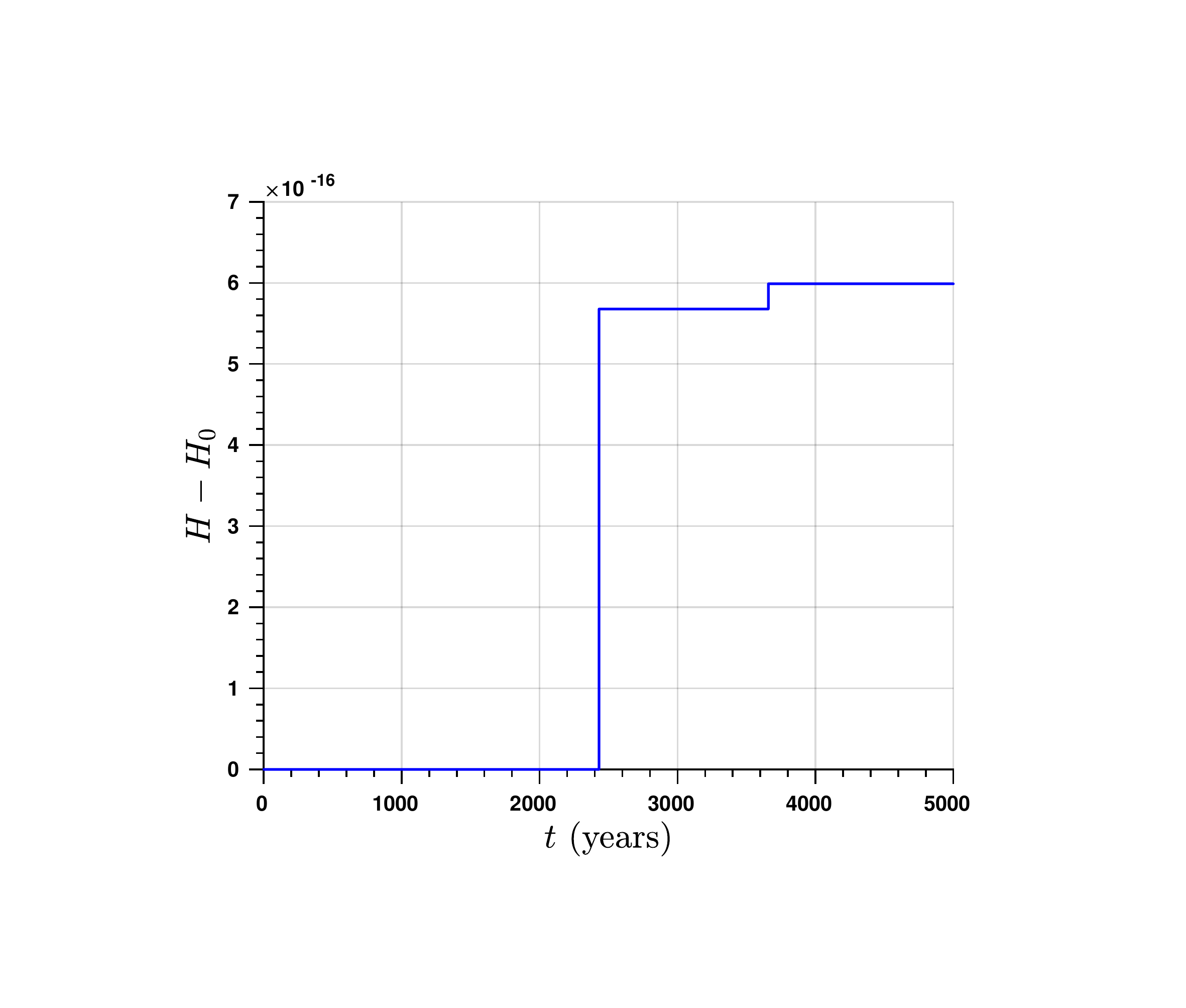} \\
\includegraphics[angle=0, trim=80 80 120 100, clip=true, scale = 0.28]{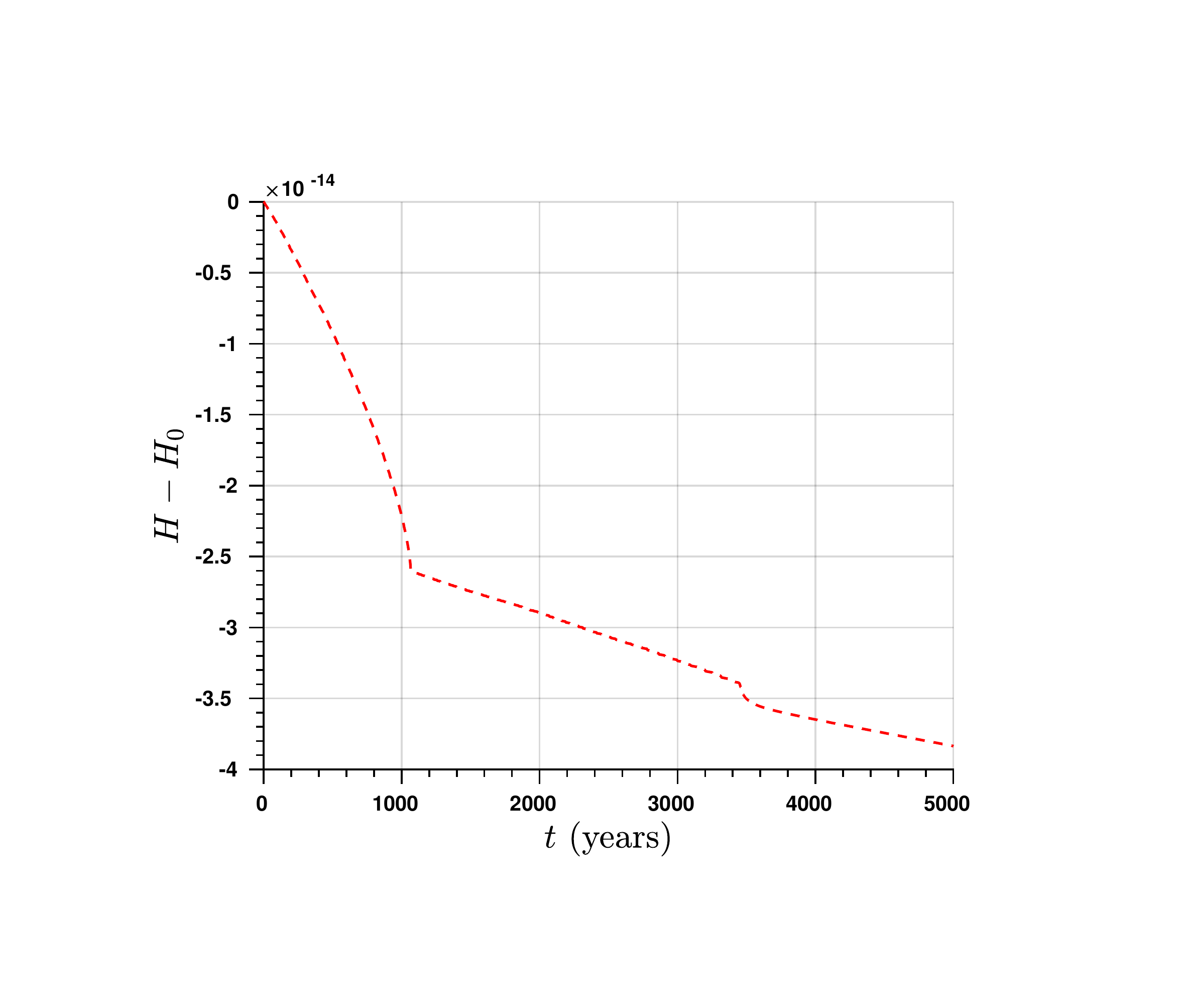} &
\includegraphics[angle=0, trim=80 80 120 100, clip=true, scale = 0.28]{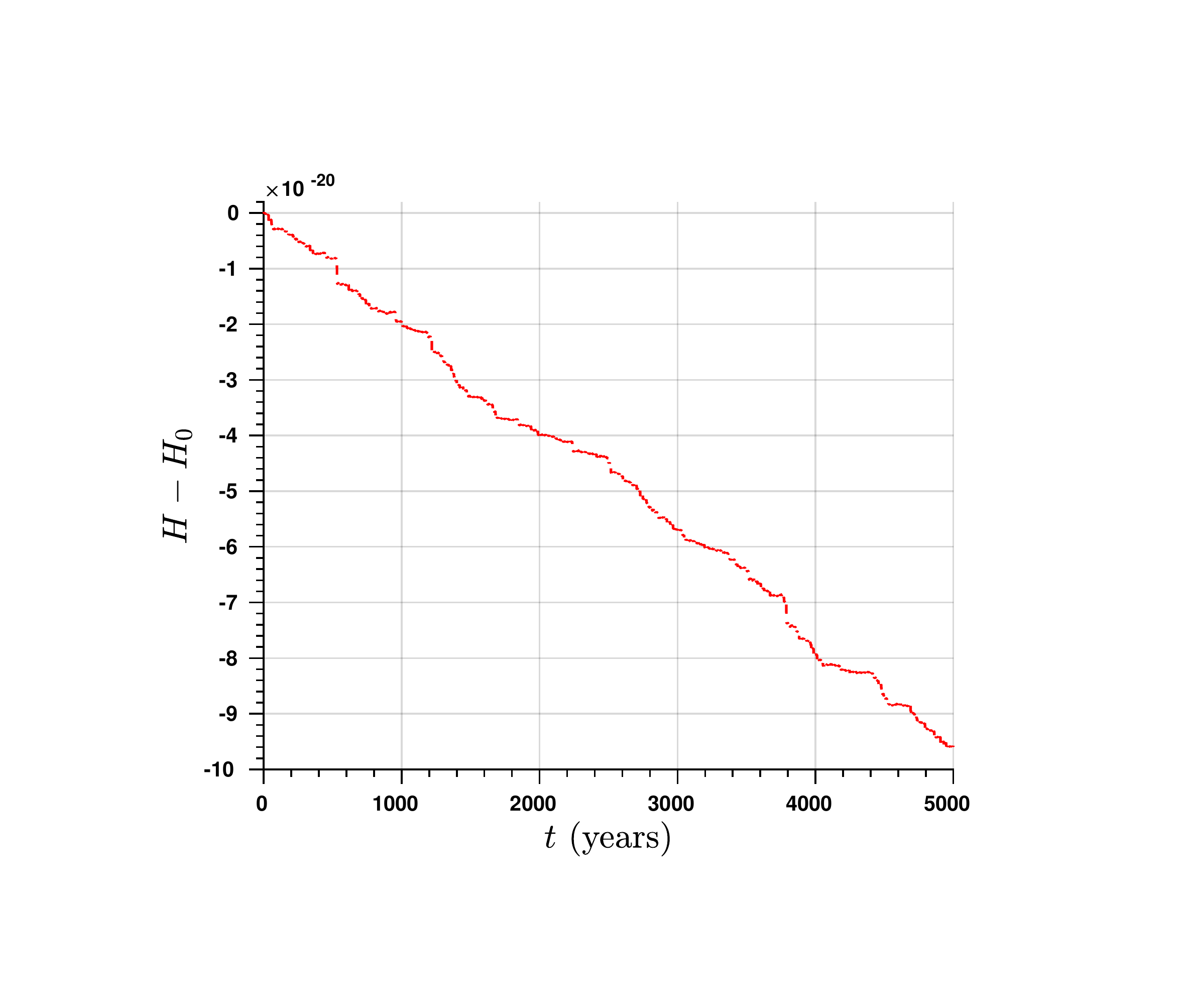} &
\includegraphics[angle=0, trim=80 80 120 100, clip=true, scale = 0.28]{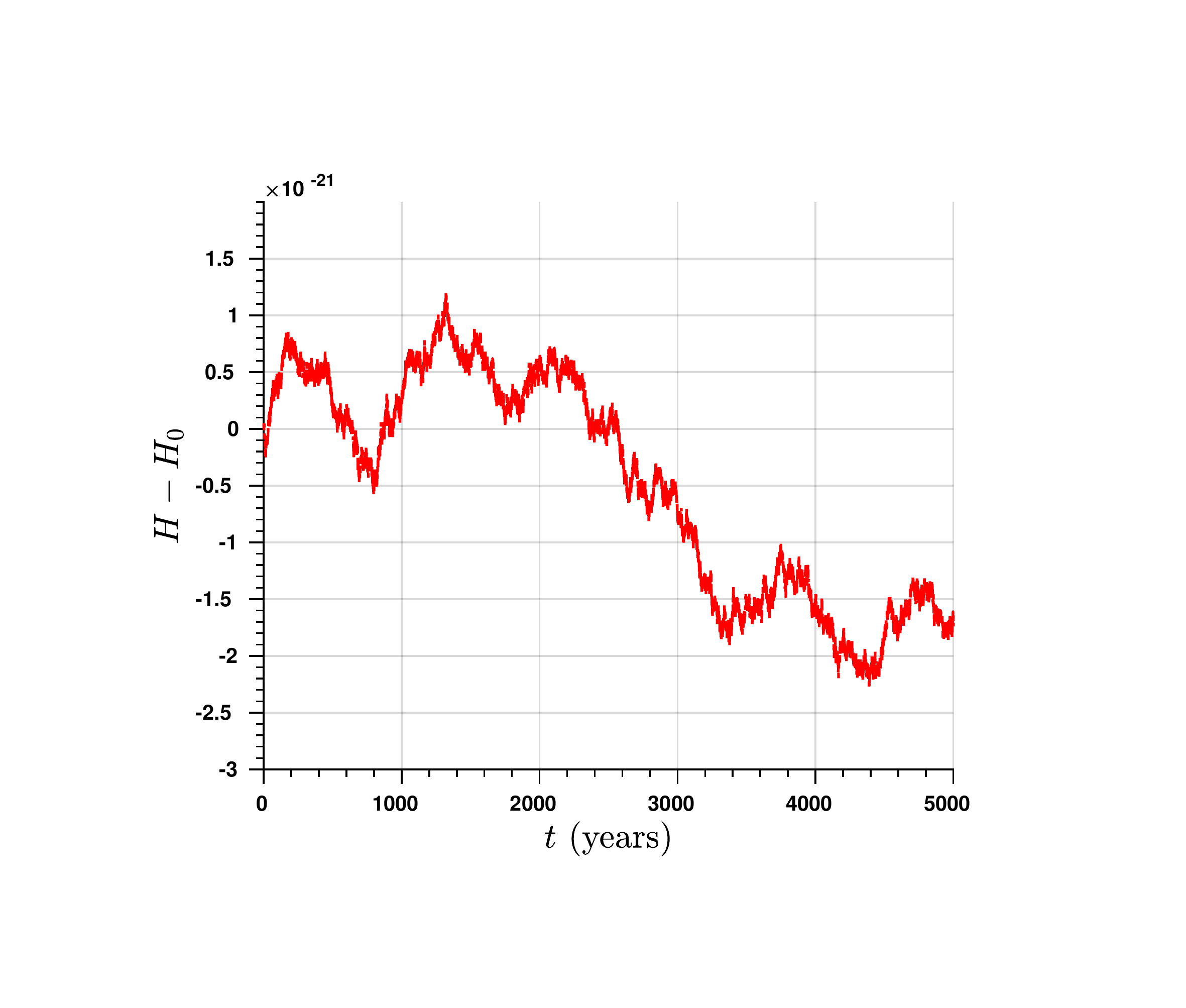} \\
\includegraphics[angle=0, trim=80 80 120 100, clip=true, scale = 0.28]{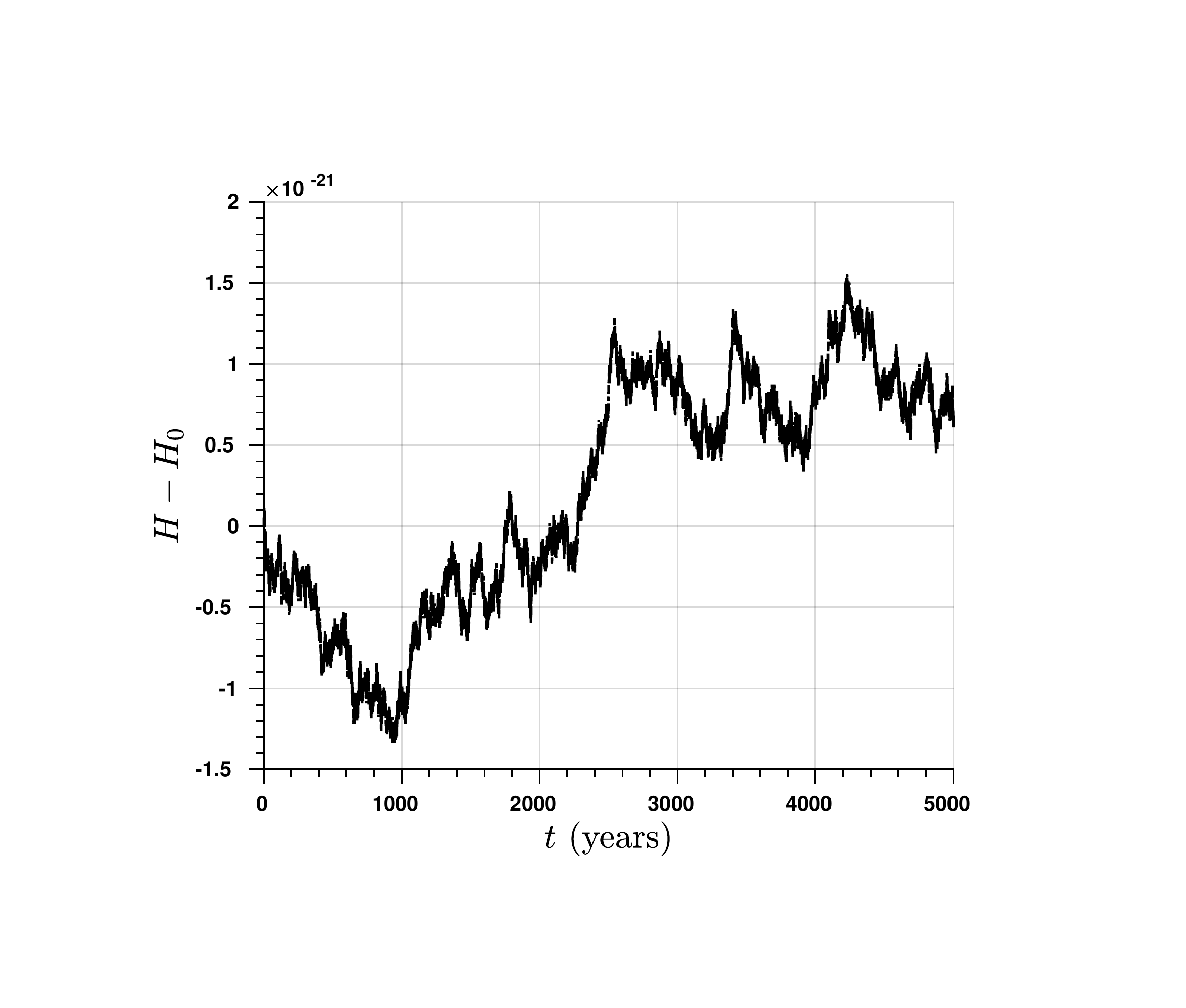} &
\includegraphics[angle=0, trim=80 80 120 100, clip=true, scale = 0.28]{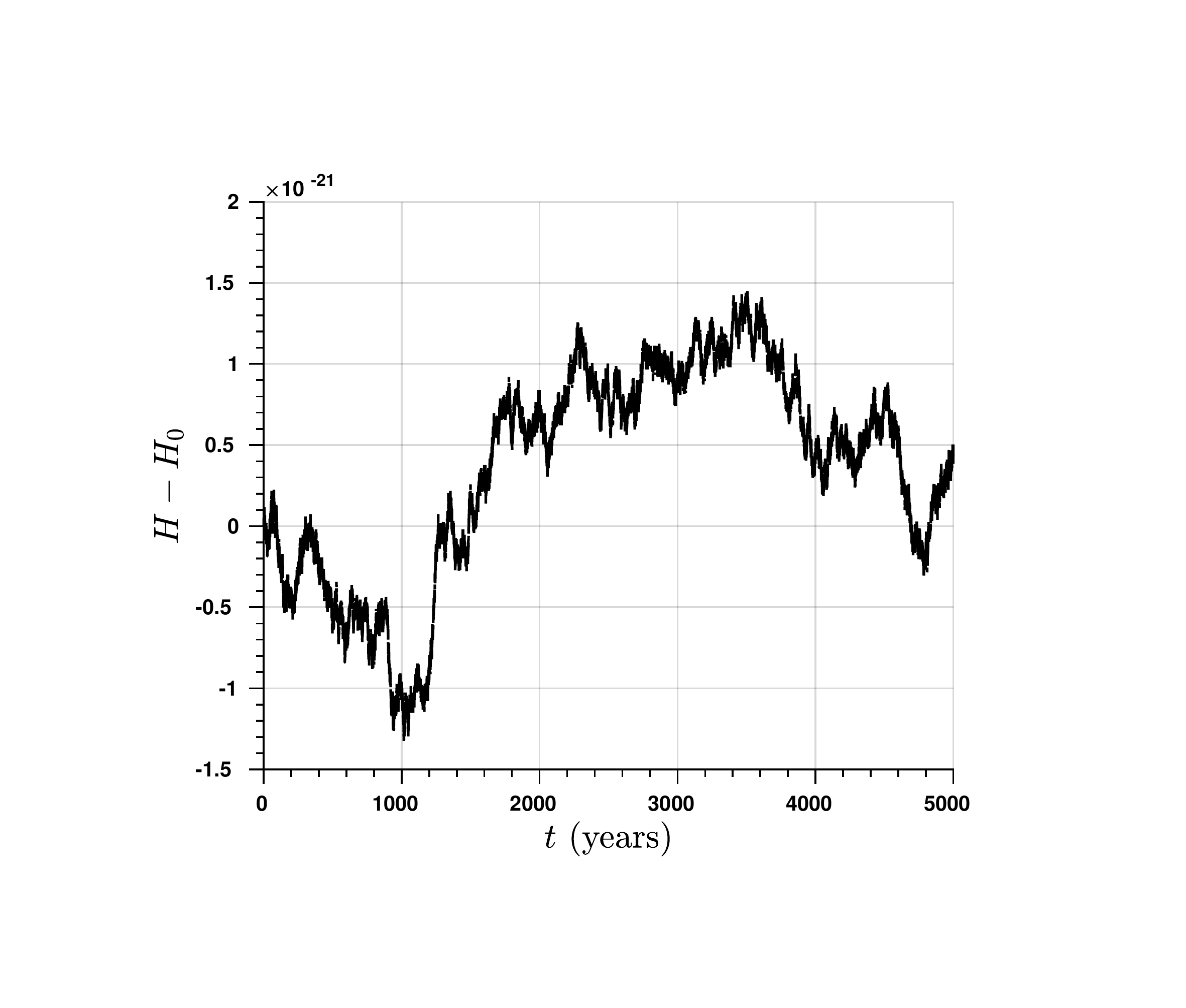} &
\includegraphics[angle=0, trim=80 80 120 100, clip=true, scale = 0.28]{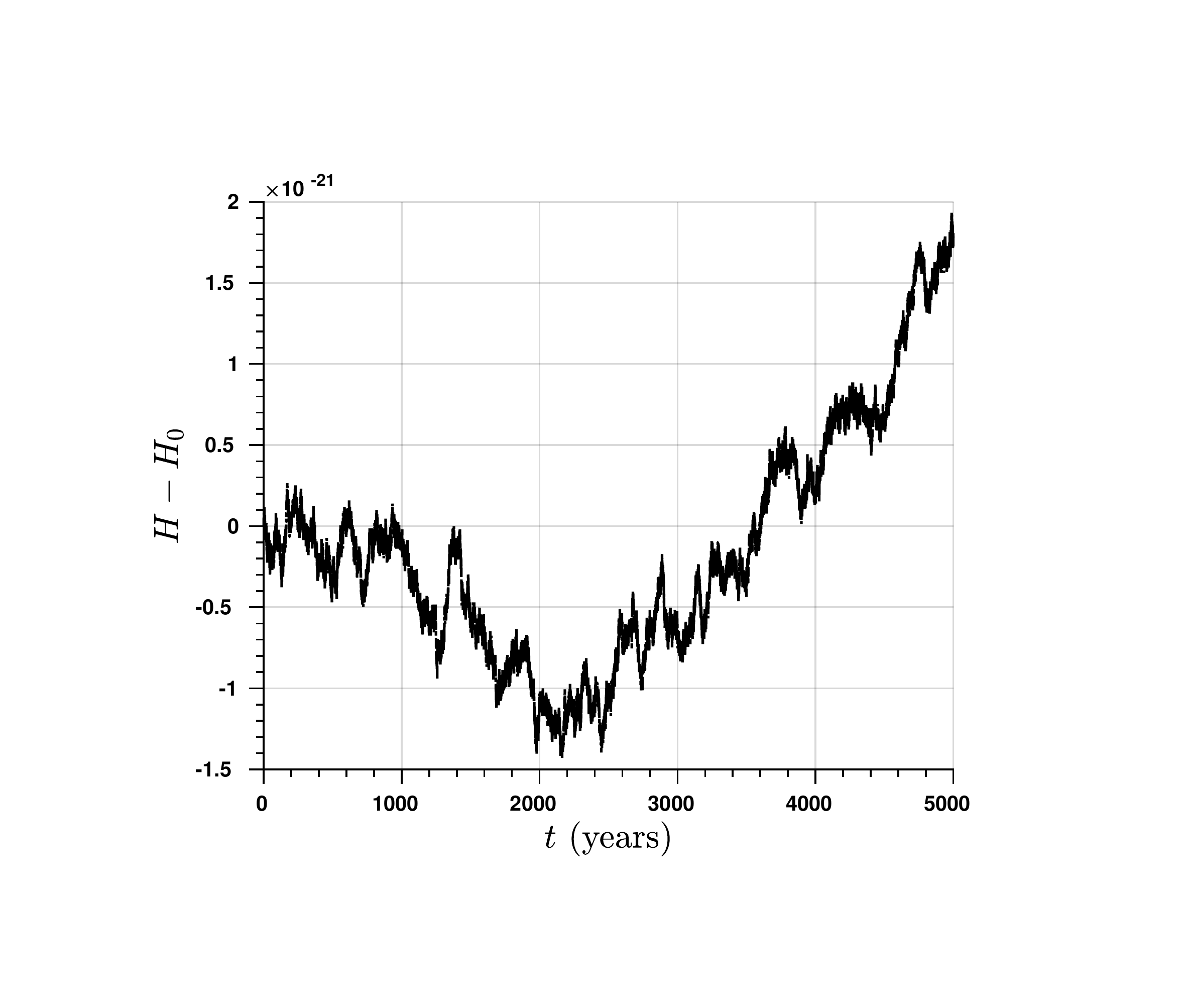} \\
$\mathrm{tol}_{\mathrm Q} = 10^{-4}$ & $\mathrm{tol}_{\mathrm Q} = 10^{-6}$ & $\mathrm{tol}_{\mathrm Q} = 10^{-8}$ \\
\hline
\end{tabular}
\caption{The error $H(\bm q_n, \bm p_n) - H(\bm q_0, \bm p_0)$ of the ten-body problem calculated by the LaBudde-Greenspan integrator with different alternate options, including the default option \eqref{eq:labudde-greenspan-default-switch} (blue solid), the generalized Eyre (red dashed), and the perturbed mid-point (black dash-dotted) integraotrs. The left, middle, and right columns are results with $\mathrm{tol}_{\mathrm Q} = 10^{-4}$, $10^{-6}$, and $10^{-8}$, respectively.} 
\label{fig:solar_hybird_labudde_greenspan}
\end{center}
\end{figure}

In the last, we again examine the robustness of the quotient formula of the LaBudde-Greenspan integrator by switching the definition of $\bm f^{AB}_{\mathrm{lg}}$ from \eqref{eq:solar_flg_2} to \eqref{eq:solar_flg_1}. In Figure \ref{fig:solar_hybird_labudde_greenspan}, the first row depicts the energy error of the default option \eqref{eq:janz-mid-point-rescue} with varying values of $\mathrm{tol}_{\mathrm Q}$. It can be gleaned from the figure that the error grows over time, similar to that of the two-particle example. The larger value of the tolerance $\mathrm{tol}_{\mathrm Q}$ results in a larger magnitude of the error, indicating the primary source of the energy growth comes from the alternate option. In the second row of Figure \ref{fig:solar_hybird_labudde_greenspan}, the results of the LaBudde-Greenspan paired with the generalized Eyre integrator are illustrated. It can be seen that for the two larger values of $\mathrm{tol}_{\mathrm Q}$, the energy decays over time. When $\mathrm{tol}_{\mathrm Q} = 10^{-8}$, the dissipation effect compensate the energy accumulation from the nonlinear solver, and the overall energy error is maintained within $\mathcal O(10^{-21})$, comparable to the LaBudde-Greenspan with \eqref{eq:solar_flg_2} (i.e., Figure \ref{fig:solar_energy} (b)). Finally, the third row of Figure \ref{fig:solar_hybird_labudde_greenspan} depicts the energy error of the LaBudde-Greenspan with the perturbed mid-point integrator. The dissipation introduced by $\bm f^{AB}_{\mathrm{mp}}$ is weaker than that of the generalized Eyre. We thus do not see the decay of the Hamiltonian. For all values of $\mathrm{tol}_{\mathrm Q}$, the energy error is bounded within $\mathcal O(10^{-21})$. The results again demonstrate that the developed integrators can be used as an effective option when the classical LaBudde-Greenspan exhibits a singular behavior numerically.

\section{Conclusions and future work}
\label{sec:conclusion}
In this work, we focused on the singular behavior of the quotient formula presented in the classical LaBudde-Greenspan integrator and energy-momentum schemes. Leveraging the specially developed quadrature rules, the energy can be split into two parts and treated separately to guarantee a dissipative nature in the numerical residual. In particular, a convex-concave split results in a first-order energy-decaying, momentum-conserving integrator. It is closely related to the celebrated Eyre's scheme in diffuse-interface problems, and we termed it the generalized Eyre integrator. The energy can be split into super-convex and super-concave parts, and they may facilitate the design of two second-order energy-decaying, momentum-conserving integrators. One can be viewed as the perturbed mid-point integrator, and the other one can be viewed as the perturbed trapezoidal integrator. These perturbation terms rectify the energy behavior of the two classical schemes in the nonlinear setting. Numerical examples are presented to corroborate the claimed properties of the developed integrators. Also, numerical evidence reveals that the classical LaBudde-Greenspan integrator paired with a classical mid-point rule in the singular limit may trigger an unfavored energy growth in long-term simulations. Combining the LaBudde-Greenspan integrator with one of the developed energy-decaying integrators helps counterbalance the error induced by the nonlinear solver.

There are several directions worthy of future study. First, the examples studied in this work are all restricted to the dynamics of the Hamilton system. The numerical strategy can be conveniently extended to non-conservative systems. The GENERIC formalism can be a suitable platform for this extension. An apparent advantage is that the proposed integrators survive when the system approaches a steady state, while the conventional energy-conserving integrators will suffer from the numerical instability rooted in its discrete force definition. Second, it is desirable to remove unwanted high-frequency modes in certain applications. The dissipative scheme developed in \cite{Armero2001} can be conveniently introduced to achieve this goal. Third, it is beneficial to extend the approach to the tensorial case, allowing an algorithm stress definition in elasticity without invoking the quotient formula \cite{Gonzalez2000}. This can be salubrious in the study of elastodynamics with a steady state and contact problems.

\section*{Acknowledgments}
This work is supported by the National Natural Science Foundation of China [Grant Number 12172160], Southern University of Science and Technology [Grant Number Y01326127], and the Guangdong-Hong Kong-Macao Joint Laboratory for Data-Driven Fluid Mechanics and Engineering Applications [Grant Number 2020B1212030001]. 

\bibliographystyle{elsarticle-num} 
\bibliography{em_dynamics}

\end{document}